\documentclass{article}
\usepackage{graphicx} 
\usepackage{geometry}
\usepackage{amsmath, amsfonts, amssymb,amsthm,nccmath,bbm,mathrsfs}
\usepackage[colorlinks,linkcolor=blue,citecolor=red]{hyperref}
\usepackage{ragged2e}
\usepackage{verbatim}
\usepackage{appendix}
\usepackage{abstract}
\numberwithin{equation}{section}

\newtheorem{theorem}{Theorem}[section]

\newtheorem{prop}{Proposition}[section]
\newtheorem{lem}{Lemma}[section]
\newtheorem{mydef}{Definition}[section]
\newtheorem{rmk}{Remark}[section]

\def\bx{\boldsymbol{x}}
\def\by{\boldsymbol{y}}

\def\bnu{\boldsymbol{\nu}}
\def\bu{\boldsymbol{u}}
\def\bv{\boldsymbol{v}}
\def\bet{\boldsymbol{\eta}}

\def\scrA{\mathscr{A}}
\def\scrJ{\mathscr{J}}

\def\lamt{\tilde{\lambda}}
\def\lamb{\bar{\lambda}}

        \def\pt{\partial_t}
        \def\ps{\partial_s}
        \def\ptau{\partial_{\tau}}
        \def\pz{\partial_z}

        \def\ptauphi{\partial_{\tau}\varphi}
        \def\psphi{\partial_{s}\varphi}
        \def\pzphi{\partial_{z}\varphi}
\def\ptauzphi{\partial_{\tau}\partial_{z}\varphi}

\def\intol{\int_0^1}

\def\intotl{\int_0^\tau\int_0^1}

\def\Flm{1+\varphi}

\def\flm{1+\psi}

\title{Nonlinear Stability of Linearly Expanding Goldreich-Weber Solutions for the Navier-Stokes-Poisson System with Degenerate Viscosity Under Radial Perturbations}
\author{Han Cao\thanks{Department of Mathematics, University of Southern California, 3620 S. Vermont Ave., Los Angeles, CA 90089-2532, USA. Email: caohan@usc.edu.}}
\date{ }
\begin{document}

\maketitle
\begin{abstract}
    In this work, we study the nonlinear stability of linearly expanding Goldreich-Weber (GW) solutions for the gravitational Navier-Stokes-Poisson system with pressure law $p(\rho)=\rho^{\frac{4}{3}}$ and degenerate viscosity. It is well known that linearly expanding GW solutions are special solutions to the Euler-Poisson system with $\gamma=\frac{4}{3}$. With bulk viscosity equal to 0, linearly expanding GW solutions are also solutions to the Navier-Stokes-Poisson equations. Choosing shear viscosity proportional to $\rho^{\alpha}$ with $0<\alpha\le\frac{2}{3}$ and zero bulk viscosity, we prove the nonlinear stability of linearly expanding GW solutions under radial perturbations.
    
\end{abstract}

\section{Introduction}
The evolution of viscous gaseous star can be described by the Navier-Stokes-Poisson equations
\begin{equation}\label{nsp1}
\begin{cases}
    &\partial_t\rho+\mathrm{div}(\rho \boldsymbol{u})=0\text{ in $\Omega(t)$},\\
&\partial_t(\rho\boldsymbol{u})+\mathrm{div}(\rho\boldsymbol{u}\otimes\boldsymbol{u})+\nabla p=\mathrm{div}T-\rho\nabla\varPhi\text{ in $\Omega(t)$},\\
    &\Delta\varPhi=4\pi\rho\text{ in $\mathbb{R}^3$ with }\lim_{|x|\rightarrow\infty}\varPhi(x)=0,\\
\end{cases}
\end{equation}
with boundary conditions
\begin{equation}\label{BC, general}
\begin{aligned}
    &\rho=0\text{ and }(p I_{3\times3}-T)\cdot\bnu=0\text{ on $\Gamma(t):=\partial\Omega(t)$}\\
\end{aligned}
\end{equation}
and initial data
\begin{equation}\label{ID, general}
     \rho(0,\bx)=\rho_0(\bx),\ \bu(0,\bx)=\bu_0(\bx)\text{ on }\Omega(0)=\Omega_0.
\end{equation}
We assume $\rho:\ \mathbb{R}^3\rightarrow\mathbb{R}_{\ge0}$ is the density initially supported on $\Omega_0\subset\subset\mathbb{R}^3$. As the system evolves, the fluid support $\Omega(t)$ changes in time, and there is no vacuum state in the interior of the domain, i.e., \begin{equation}
    \rho(t,\bx)>0\text{ in }\Omega(t).
\end{equation}
$\Gamma(t):=\partial\Omega(t)$ is the moving vacuum boundary. $\boldsymbol{u}:\ \mathbb{R}^3\rightarrow\mathbb{R}^3$ is the velocity, and $\mathcal{V}(\Gamma(t))$ denotes the normal velocity of $\Gamma(t)$ which satisfies
\begin{equation}
    \mathcal{V}(\Gamma(t))=\bu\cdot\bnu.
\end{equation}
$p$ is the pressure satisfying
\begin{equation}\label{Polytropic}
    p=p(\rho)=K\rho^{\gamma}\text{ with }\gamma\in(1,2),
\end{equation}
which is the equation of state for a polytropic gas, where $K>0$ is a constant set to be 1 for convenience. $\gamma$ is the adiabatic exponent. $\varPhi$ is the potential function of the self-gravitational force which admits the expression
\begin{equation*}
    \varPhi(t,\bx)=-\int_{\varOmega(t)}\frac{\rho(t,\by)}{|\bx-\by|}d\by
\end{equation*}
by solving the Poisson equation (\ref{nsp1})$_3$. The stress tensor $T$ is given by
\begin{equation}\label{StressTen}
    T=\mu_1(\rho)\left(\nabla\boldsymbol{u}+\nabla\boldsymbol{u}^T-\frac{2}{3}(\mathrm{div}\boldsymbol{u})I_{3\times3}\right)+\mu_2(\rho)(\mathrm{div}\boldsymbol{u})I_{3\times3},
\end{equation}
where we allow the viscosity coefficients to be dependent on the density. Here, $\mu_1(\rho)\ge0$ is the shear viscosity, and $\mu_2(\rho)\ge0$ is the bulk viscosity. One typical example is $\mu_1(\rho)=\tilde{\mu}_1\rho^{\alpha},\ \mu_2(\rho)=\tilde{\mu}_2\rho^{\alpha}$ with $\alpha>0$ \cite{LXY}. On the other hand, if $\mu_1(\rho)=0,\ \mu_2(\rho)=0$, then (\ref{nsp1}) models the inviscid self-gravitating gaseous stars, called the Euler-Poisson system.

A type of special vacuum state known as the physical vacuum boundary \cite{LY} is characterized by the fact that the interface between the fluid and the vacuum moves with a nontrivial finite normal acceleration, i.e.,
\begin{equation}\label{Eq: physical vacuum boundary}
    -\infty<\frac{\partial c^2}{\partial\bnu}<0\text{ for }x\in\Gamma(t),
\end{equation}
where $c=\sqrt{\frac{\partial\rho^\gamma}{\partial\rho}}$ denotes the sound speed. This type of boundary condition arises naturally in the study of compressible flows and has attracted much research interest. Note that $c$ is only $C^{\frac{1}{2}}$ at the boundary, making the analysis of compressible fluids with physical vacuum boundary challenging. We first discuss several special solutions admitting the physical vacuum boundary. Later in this section, we review some broader results on compressible viscous and inviscid problems involving physical vacuum boundary.

One of the most famous classes of special solutions to (\ref{nsp1}) is the Lane-Emden solutions, that is, the non-rotating steady solution of (\ref{nsp1}). We denote $(\bar{\rho}_\sigma(r),0)$ to be such a radially symmetric solution with central density $\sigma>0$. Then $\bar{\rho}_\sigma(r)$ satisfies the following ODE:
\begin{equation}\label{ode}
    \partial_r\bar{\rho}^{\gamma}_{\sigma}=-\frac{4\pi\bar{\rho}_{\sigma}}{r^2}\int_0^r\bar{\rho}_{\sigma}s^2ds.
\end{equation}
As for the solutions to (\ref{ode}), when $\gamma\in(\frac{6}{5},2)$, for given finite total mass, there exists at least one compactly supported solution to (\ref{ode}), while every solution is compactly supported and unique when $\gamma\in(\frac{4}{3},2)$. When $\gamma=\frac{6}{5}$, there exists a unique solution with explicit expression, which has infinite support. When $\gamma\in(1,\frac{6}{5})$, there are no solutions with finite mass \cite{LSS}. When $\gamma>\frac{6}{5}$, we denote the radius of the support of $\bar{\rho}_\sigma$ by $0<R_\sigma<+\infty$. We also denote the total mass and the total energy of $\bar{\rho}_\sigma$ by 

\begin{equation}
    M_{\sigma}=\int_0^{R_{\sigma}}4\pi\bar{\rho}_{\sigma}(r) r^2dr,
\end{equation}
\begin{equation}
    E_{\sigma}=\frac{1}{\gamma-1}\int_0^{R_{\sigma}}4\pi\bar{\rho}_{\sigma}^{\gamma} r^2dr-\int_0^{R_{\sigma}}4\pi\bar{\rho}_{\sigma}r\int_0^r4\pi\bar{\rho}_{\sigma}s^2dsdr.
\end{equation}
Scaling analysis reveals some important adiabatic exponents. To be specific, if $\bar{\rho}_1(r)$ is the Lane-Emden solution with central density 1, then $\beta^{\frac{2}{2-\gamma}}\bar{\rho}_1(\beta r)$ is also a Lane-Emden solution, with central density $\sigma=\beta^{\frac{2}{2-\gamma}}$. What's more \cite{Chan},
\begin{equation}
    M_{\sigma}=M_1\sigma^{\frac{3\gamma-4}{2}},\ R_{\sigma}=R_1\sigma^{\frac{\gamma-2}{2}},\ E_{\sigma}=E_1\sigma^{\frac{5\gamma-6}{2}}.
\end{equation}
We call the case $\gamma=\frac{4}{3}$ mass critical and the case $\gamma=\frac{6}{5}$ energy critical for the Lane-Emden equation. By setting $\rho_*(r):=R_1^{\frac{2}{2-\gamma}}\bar{\rho}_1(R_1r)$, we can rescale the domain to $[0,1]$. We define the enthalpy $w:=\rho_*^{\gamma-1}$, then (\ref{ode}) can be written as
\begin{equation}\label{ode2}
    w''(z)+\frac{2}{z}w'(z)+\frac{\gamma-1}{\gamma}4\pi w^{\frac{1}{\gamma-1}}(z)=0
\end{equation}
for $z\in[0,1]$. We note that the Lane-Emden solutions satisfy the physical vacuum boundary condition (\ref{Eq: physical vacuum boundary}), which translates into the fact that $w(z)$ behaves like the distance function to the vacuum boundary.

The stability of Lane-Emden solutions has attracted much research interest. At the mass critical adiabatic exponent, a transition of stability of Lane-Emden solutions occurs. In particular, the Lane-Emden solutions are linearly unstable for $\frac{6}{5}<\gamma<\frac{4}{3}$ and linearly stable for $\frac{4}{3}\le\gamma<2$ \cite{LSS}. In particular, $\gamma=\frac{4}{3}$ is neutrally stable. In this case, a type of nonlinear instability for the Euler-Poisson system was found \cite{DLYY}, in the sense that small perturbation of the Lane-Emden solutions may lead to mass escaping to infinity. Moreover, later in this section and in Subsection \ref{Subsection: Goldreich-Weber solutions}, we will see that the Goldreich-Weber solutions are explicit solutions to the Euler-Poisson equations that exhibit the nonlinear instability of mass critical Lane-Emden solutions. There are also various results for the nonlinear stability of Lane-Emden solutions. For instance, when $\gamma\in(\frac{4}{3},2)$, for the Euler-Poisson system, nonlinear conditional stability was proved in \cite{Rein} and unconditional orbital stability for weak solutions under radial perturbation was proved in \cite{LWZ}, while nonlinear stability for the Navier-Stokes-Poisson system under radial perturbation was proved in \cite{LXZ1,LXZ2} for constant viscosity and density-dependent viscosity respectively. When $\gamma\in(\frac{6}{5},\frac{4}{3})$, both the Euler-Poisson system \cite{J3} and the Navier-Stokes-Poisson system \cite{JT} are nonlinearly unstable. When $\gamma=\frac{6}{5}$, the Euler-Poisson system is nonlinearly unstable \cite{J1}. There are also some studies on the stability of gaseous stars with a more general pressure $p(\rho)$ satisfying $p(\rho)\approx\rho^{\gamma_0}$ near $\rho=0$, where transition of stability of the steady state happens at $\gamma_0=\frac{4}{3}$, which was showed in \cite{LZ} and \cite{CLW} via a turning point principle. See also \cite{LuoWangZeng2024} for the nonlinear asymptotic stability of steady states with general pressure law, which includes white dwarf and polytropes for all $\gamma>\frac{4}{3}$, for the constant viscosity Navier–Stokes–Poisson equations under radial perturbations.

In the mass critical case with pressure law $p(\rho)=\rho^{\frac{4}{3}}$, it has been proved in \cite{GW,Makino,FL} that the Euler-Poisson equations admits a special type of radial solutions with the form
\begin{equation}
    \begin{aligned}
        \begin{cases}
            \rho(t,\bx)=\lambda^{-3}(t)w^3(\frac{|\bx|}{\lambda(t)}),\\
            \bu(t,\bx)=\frac{\lambda'(t)}{\lambda(t)}\bx,
        \end{cases}
    \end{aligned}
\end{equation}
where the radial of support $\lambda(t)$ and the enthalpy $w(z)$ solve the ODE system
\begin{equation}\label{General LE}
    \begin{cases}
        \lambda^2(t)\lambda''(t)=\delta,\\
        w''(z)+\frac{2}{z}w'(z)+\pi w^3(z)=-\frac{3\delta}{4},
    \end{cases}
\end{equation}
where $w:[0,1]\rightarrow\mathbb{R}_{\ge0}$, with initial conditions of $\lambda(t)$
\begin{equation}\label{Initial condtions of lambda}
    \lambda(0)=\lambda_0>0,\ \lambda'(0)=\lambda_1
\end{equation}
and boundary conditions of $w(z)$
\begin{equation}\label{Boundary condtions of w}
    w'(0)=0,\ w(1)=0,
\end{equation}
which is called the Goldreich-Weber solutions (GW). We point out that the solution to (\ref{General LE})$_2$ also satisfies the physical vacuum boundary condition (\ref{Eq: physical vacuum boundary}). The long-time behaviors of the GW solutions depend on the parameters $(\lambda_0,\lambda_1,\delta)$ \cite{GW,Makino,FL}. In particular, when $\delta=0$, the enthalpy equation (\ref{General LE})$_2$ reduces to the classical Lane-Emden equation. In this case, $\lambda(t)=\lambda_0+\lambda_1t$. When $\delta>0$, the support expands to infinity in an asymptotically linear rate in time. While for $\delta^*<\delta<0$ for some $\delta^*<0$, there exists an escape velocity $\lambda^*_1>0$ such that when $\lambda_1>\lambda^*_1$, the support expands asymptotically linearly, when $\lambda_1=\lambda^*_1$, $\lambda(t)$ admits an explicit form and $\lambda(t)\sim_{t\rightarrow\infty}t^{\frac{2}{3}}$, which is referred to as self-similar expanding solutions, and when $\lambda_1<\lambda^*_1$, the support shrinks to a point in a finite time $T>0$ with $\lambda(t)\sim_{t\rightarrow T^-}k_1(T-k_2t)^{\frac{2}{3}}$. In \cite{HJ} and \cite{HJL}, the nonlinear stability of the self-similar expanding GW solutions and linearly expanding GW solutions for the Euler-Poisson equations under radial and non-radial perturbation were showed respectively. While for the Navier-Stokes-Poisson equations with constant viscosity, the instability of self-similar expanding solutions and nonlinear stability of linearly expanding GW solutions were proved in \cite{Liu2019Isentropic}. On the other hand, for the model of viscous radiation gaseous star, the existence of stationary solutions \cite{LXin2} and the stability of a class of linearly expanding homogeneous solutions \cite{LXin} were proved.

One related question is under what conditions is the GW solution a solution to the isentropic Navier-Stokes-Poisson equations (\ref{nsp1}) with $\gamma=\frac{4}{3}$ and whether it is stable in the vacuum free boundary framework for the Navier-Stokes-Poisson equations. In fact, the GW solution solves the continuity equation in (\ref{nsp1}), so what is left to check is the momentum equation. We can rewrite the momentum equation as follows
\begin{equation}\label{sfsfsfsfsff}
     \partial_t(\rho\boldsymbol{u})+\mathrm{div}(\rho\boldsymbol{u}\otimes\boldsymbol{u})+\nabla p+\rho\nabla\varPhi=\mathrm{div}T.
\end{equation}
Notice that the GW solutions make the left-hand side of (\ref{sfsfsfsfsff}) equal to 0, we only need to check whether $\mathrm{div}T=0$ for the GW solutions. By direct computation,
\begin{equation}
   T_{\rm GW}=\frac{3\lambda'(t)}{\lambda(t)}\mu_2(\rho)I_{3\times3}.
\end{equation}
When $\mu_2(\rho)=\mu_2$ is constant, $\mathrm{div}T_{GW}=0$. What's more, the boundary conditions (\ref{BC, general}) further forces the constant $\mu_2$ to vanish. Consequently, the GW solutions are special solutions to (\ref{nsp1}) with boundary conditions (\ref{BC, general}) if we assume zero bulk viscosity, i.e., $\mu_2(\rho)=0$. Our results focus only on the case 
\begin{equation}\label{Eq: parameter gamma, mu}
    \gamma=\frac{4}{3},\ \mu_1(\rho)=\mu\rho^{\alpha},\ \mu_2(\rho)=0
\end{equation}
with $\mu>0$ and $0<\alpha\le\frac{2}{3}$. Later in Remark \ref{Rmk: Choice of viscosity coefficient}, we will discuss the reason of restricting $\alpha$ to the range $(0,\frac{2}{3}]$. Our goal is to show that there exists a small enough $\tilde{\varepsilon}>0$ such that the linearly expanding GW solutions are nonlinearly stable under radial perturbation for $-\tilde{\varepsilon}<\delta$. The rigorous statement of our result is given by Theorem \ref{Thm: stability of GW, LG}. To the best of our knowledge, this is the first result showing the nonlinear stability of the linearly expanding GW solutions for the Navier-Stokes-Poisson equations with degenerate viscosity.

The study of this stability problem is closely connected with the theory of compressible flows with vacuum free boundaries, including both viscous and inviscid models, whose existence and dynamical properties have been extensively investigated in the literature. We refer to \cite{GLX,ZF,YZ} for the theory of global weak solutions to the free boundary value problem of compressible Navier-Stokes equations, and to \cite{DL,KL,CAO2026114572} for the corresponding theory for the Navier-Stokes-Poisson equations. The global existence of weak solutions to the Euler-Poisson equations in whole space was constructed in \cite{CHLWW,CCL} by passing to the weak limit of a family of solutions to the moving boundary Navier-Stokes-Poisson system as the viscosity tends to zero. In particular, the solutions constructed in the works \cite{CCL,CAO2026114572} allow initial data to be taken near the Lane-Emden stars. However, due to the singular behavior at the free boundary caused by the vacuum state, it is difficult to obtain higher regularity solutions, and standard methods for establishing well-posedness are not directly applicable to either the viscous or inviscid problems. In the meantime, by imposing the physical vacuum boundary (\ref{Eq: physical vacuum boundary}) to the initial density and introducing some weighted estimates which make use of the fact that $\rho^{\gamma-1}$ behaves like distance function to the free boundary, local well-posedness of compressible Euler equations \cite{JM0,CLS,CS1,CS2,JM}, Euler-Poisson equations \cite{GL,LXZ3} and shallow water equations \cite{LWX} can be proved. See also \cite{J2} for the local well-posedness of radial strong solutions to the Navier-Stokes-Poisson equations with constant viscosity for $\gamma>\frac{6}{5}$ and the recent work \cite{WangZhangZhu2026Jun} for the corresponding result for degenerate viscosity with $\gamma>\frac{4}{3}$. There are also various results for global strong solutions. For the compressible Euler equations, a broad family of expanding affine solutions\footnote{We note that GW solutions are a class of radially symmetric affine solutions.}, which is not necessarily radially symmetric, with physical vacuum boundary exists \cite{Sideris}. Moreover, global unique strong solutions to the isentropic compressible Euler equations near the affine solutions can be constructed \cite{HJ2,shkollerSideris}. See also \cite{HJ3} for a corresponding result on the Euler-Poisson system and \cite{Rickard} for results on non-isentropic system with heat convection around Dyson’s isothermal affine solutions. For global expanding solutions with small initial density but not necessarily close to affine solutions, we refer to \cite{PHJ}. See also \cite{Parmeshwar} for results on global expanding $N$-body solutions to the Euler-Poisson equations with small mass. For the viscous problem, the existence of global radial strong solutions to the Navier-Stokes-Poisson equations with $\frac{4}{3}<\gamma<2$ for both constant viscosity and density-dependent viscosity near the Lane-Emden stars were proved in \cite{LXZ1} and \cite{LXZ2} respectively. The existence of global radial classical solutions to the degenerate Navier-Stokes equations with large data was proved in the recent work \cite{CZZ}, where the authors consider the system with $\gamma\in(\frac{4}{3},\infty)$ in the 2-dimensional case and $\gamma\in(\frac{4}{3},3)$ in the 3-dimensional case. See also \cite{WangZhangZhu2026Jul} for a related result on the degenerate Navier–Stokes–Poisson equations.

\section{Lagrangian formulation and main results}
\subsection{Goldreich-Weber solutions}\label{Subsection: Goldreich-Weber solutions}
In this subsection, we summarize the properties of Goldreich-Weber solutions. The following Proposition guarantees the existence of solutions to (\ref{General LE})$_2$ for a range of $\delta$.
\begin{prop}[\cite{FL}]\label{Proposition: General LE, w}
    There exists a constant $\delta^*<0$ such that for any $\delta\ge\delta^*$ there exists solutions $w=w(z)$ to (\ref{General LE})$_2$ with $0<w<\infty$ in $[0,1)$ and $w$ satisfies boundary conditions (\ref{Boundary condtions of w}).
\end{prop}
The following Lemma shows the regularities of $w$.
\begin{lem}[\cite{HJ}]\label{Lemma: General LE}
    Let $w$ be a solution to (\ref{General LE})$_2$. Then $w\in C^{\infty}(0,1)$ and $w$ is analytic near $z=0$ and $z=1$. Moreover,
    \begin{enumerate}
        \item $w(z)=A_1+A_2z^2+O(z^4)$ for some constant $A_1,\ A_2$ and $w^{(2k+1)}(0)=0$ for any nonnegative integer $k\ge0$;
        \item $w$ satisfies the physical vacuum condition, i.e. $-\infty<w'(1)<0$.
    \end{enumerate}
\end{lem}

The following Proposition shows the properties of solutions to (\ref{General LE})$_1$, which characterizes the behavior of the support as time goes by.
\begin{prop}[\cite{FL},\cite{Makino},\cite{HJ}]\label{Proposition: General LE, lambda}Let $\lambda(t)$ be the solution to (\ref{General LE})$_1$ with initial conditions (\ref{Initial condtions of lambda}).
    \begin{enumerate}
        \item If $\delta>0$, then $\lambda(t)>0$ for all $t>0$ with $\lim_{t\rightarrow\infty}\lambda(t)=\infty$. Moreover, there exist $c_1,c_2>0$ such that
        \begin{equation}\label{linear expanding rate for lambda}
            \lambda(t)\sim_{t\rightarrow\infty}c_1(1+c_2t).
        \end{equation}
        \item If $\delta=0$, then $\lambda(t)=\lambda_0+\lambda_1t$.
        \item If $\delta<0$, let
        \begin{equation}
            \lambda^*_1=\sqrt{\frac{2|\delta|}{\lambda_0}}.
        \end{equation}
        \begin{enumerate}
            \item If $\lambda_1=\lambda^*_1$, then $\lambda(t)>0$ for all $t>0$, and
            \begin{equation}\label{Eq: self-similar lambda}
                \lambda(t)=\left(\lambda_0^{\frac{3}{2}}+\frac{3}{2}\lambda^{\frac{1}{2}}_0\lambda_1t\right)^{\frac{2}{3}},\ t\ge0.
            \end{equation}
            \item If $\lambda_1>\lambda^*_1$, then $\lim_{t\rightarrow\infty}\lambda(t)=\infty$, and there exist $c_1,c_2>0$ such that (\ref{linear expanding rate for lambda}) holds.
            \item If $\lambda_1<\lambda^*_1$, then there is a time $0<T<\infty$ such that $\lambda(t)>0$ in $(0,T)$ and $\lambda(t)\rightarrow0$ as $t\rightarrow T^{-}$. Moreover, there exist constants $k_1,k_2>0$ such that
            \begin{equation}
                \lambda(t)\sim_{t\rightarrow T^-}k_1(T-k_2t)^{\frac{2}{3}},\ t\ge0.
            \end{equation}
        \end{enumerate}
    \end{enumerate}
\end{prop}
The term "linearly expanding GW solutions" refers to one of the following three cases:
\begin{itemize}
    \item $\delta>0$,
    \item $\delta=0$ with $\lambda_1>0$,
    \item $\delta^*<\delta<0$ with $\lambda_1>\lambda_1^*$, 
\end{itemize}
while the term "self-similar expanding GW solutions" refers to the case 3 (a) in Proposition \ref{Proposition: General LE, lambda}.

\subsection{Lagrangian formulation}
We introduce the Lagrangian coordinates and reformulate the problem into the perturbation framework. Consider the system (\ref{nsp1})-(\ref{StressTen}) with (\ref{Eq: parameter gamma, mu}) and let $\bet(t,\bx): B_1(0)\rightarrow\Omega(t)$ be the Lagrangian flow map satisfying 
\begin{equation}
    \begin{aligned}
        \frac{d\bet}{dt}(t,\bx)=\bu(t,\bet(t,\bx)),\\
        \bet(0,\bx)=\bet_0(\bx),
    \end{aligned}
\end{equation}
where $\bet_0(\bx):\ B_1(0)\rightarrow\Omega(0)$ is a diffeomorphism with positive Jacobian determinant. In our problem, we only consider radial perturbation. As a result, we can assume
\begin{equation}
    \bet(t,\bx)=\chi(t,z)\bx,\ z=|\bx|,
\end{equation}
where $\chi:\ \mathbb{R}_{\ge0}\times[0,1]\rightarrow\mathbb{R}_{\ge0}$. The Jacobian of $\bet(t,\bx)$ is given by
\begin{equation}
    \scrJ:=\det[D\bet]=\chi^2(\chi+z\partial_z\chi).
\end{equation}
We restrict ourselves to radial perturbations having the same total mass as the GW solution. Without loss of generality, we assume $\Omega_0=B_1(0)$ and choose $\bet_0$ such that
\begin{equation}
    w^3=\rho_0\chi_0^2(\chi_0+z\partial_z\chi_0),
\end{equation}
where $w$ is the solution to the enthalpy equation in (\ref{General LE})$_2$. The existence of such $\bet_0$ is ensured by \cite{DacorognaMoser}. Then, the system (\ref{nsp1})-(\ref{StressTen}) can be rewritten as
\begin{equation}\label{chiEqn}
\begin{aligned}
    \pt^2\chi+F_w[\chi]=G_{w,\alpha}[\chi],
    \end{aligned}
\end{equation}
where
\begin{equation}
    \begin{aligned}
        &F_w[\chi]=\frac{\chi^2}{zw^3}\pz\left(w^4\scrJ^{-\frac{4}{3}}\right)+\frac{1}{z^3\chi^2}\int_0^z4\pi w^3s^2ds,\\
        &G_{w,\alpha}[\chi]=\frac{1}{\chi w^3z^4}\pz\left(\frac{4}{3}\mu z^3\chi^2\left(w^3\scrJ^{-1}\right)^{\alpha}\frac{\chi\cdot z\pt\pz\chi-z\pz\chi\pt\chi}{\chi+z\pz\chi}\right),
    \end{aligned}
\end{equation}
with boundary condition
\begin{equation}
    \left.\frac{4}{3}\mu\left(w^3\scrJ^{-1}\right)^{\alpha}\left(\frac{\chi\cdot z\pt\pz\chi-\pt\chi\cdot z\pz\chi}{\chi(\chi+z\pz\chi)}\right)\right|_{z=1}=0.
\end{equation}
Full details of the derivation of (\ref{chiEqn}) are given in Appendix \ref{Chapter: Appendix Reform}. We can derive the following estimate for physical energy-dissipation identity:
\begin{equation}
    \begin{aligned}
        &\frac{d}{dt}\left\{\intol\frac{1}{2}(\pt\chi)^2w^3z^4dz+\intol3w^4z^2\scrJ^{-\frac{1}{3}}dz-\intol\frac{zw^3}{\chi}\left(\int_0^z4\pi w^3s^2ds\right)dz\right\}\\
        &+\intol\frac{4}{3}\mu z^2\left(w^3\scrJ^{-1}\right)^{\alpha}\frac{(\chi\cdot z\pt\pz\chi-\pt\chi\cdot z\pz\chi)^2}{(\chi+z\pz\chi)}dz=0.
    \end{aligned}
\end{equation}

Next we introduce the unknown $\xi$ as follows
\begin{equation}
    \xi(t,z)=\frac{\chi(t,z)}{\lambda(t)},\ J=\xi^2(\xi+z\pz\xi),
\end{equation}
where $\lambda(t)$ is the radius of support of the linearly expanding GW solution. By direct computation, 
\begin{equation}
    \begin{aligned}
        &\pt\chi=\lambda'(t)\xi+\lambda\pt\xi,\\
        &\pt^2\chi=\lambda''(t)\xi+2\lambda'(t)\pt\xi+\lambda\pt^2\xi,\\
        &\scrJ=(\lambda\xi)^2(\lambda\xi+z\pz(\lambda\xi))=\lambda^3J,
    \end{aligned}
\end{equation}
so (\ref{chiEqn}) can be rewritten as
\begin{equation}\label{xiEqn}
    \lambda^2\pt^2(\lambda\xi)+F_w[\xi]=\lambda^{4-3\alpha}G_{w,\alpha}[\xi]
\end{equation}
with boundary condition
\begin{equation}
    \left.\frac{4}{3}\mu\left(w^3J^{-1}\right)^{\alpha}\left(\frac{\xi\cdot z\pt\pz\xi-\pt\xi\cdot z\pz\xi}{\xi(\xi+z\pz\xi)}\right)\right|_{z=1}=0.
\end{equation}

To reflect the expected linear growth rate of the perturbed solution, we introduce a new time variable
\begin{equation}
    \tau(t)=\int_0^t\frac{1}{\lambda(\sigma)}d\sigma,
\end{equation}
then $\pt=\frac{1}{\tilde{\lambda}(\tau)}\partial_{\tau}$, where $\lamt(\tau):=\lambda(t(\tau))$. In the $\tau$ coordinate, (\ref{General LE})$_1$ is written as
\begin{equation}\label{Eq: ode of lamt}
    \lamt'=\text{sgn}(\lamt')\lamt\sqrt{\tilde{e}-\frac{2\delta}{\lamt}}, 
\end{equation}
where we let $\tilde{e}=\lambda_1^2+\frac{2\delta}{\lambda_0}$. We point out that in the linearly expanding regime, $\lamt'(\tau)$ will eventually become positive, even though for the $\delta>0$ case where initial velocity of GW solution $\lamt_1$ can be taken negative. Then (\ref{Eq: ode of lamt}) implies
\begin{equation}
    \lamt(\tau)\sim_{\tau\rightarrow\infty}\exp({\tilde{\beta}\tau)}
\end{equation}
for $\tilde{\beta}=\sqrt{\tilde{e}}$, which gives
\begin{equation}\label{Eq: behaviour of lamt', lamt''}
    b_0\le\frac{\lamt'}{\lamt}\le b_1,\ b_0\le\frac{\lamt''}{\lamt'}\le b_1
\end{equation}
for some constants $b_0,b_1>0$ and $\tau>0$ sufficiently large. Up to a translation of time, we can assume without loss of generality that (\ref{Eq: behaviour of lamt', lamt''}) holds for all $\tau\ge0$ in the rest of this paper. Next, we define the perturbed variable
\begin{equation}
    \varphi(\tau,z):=\xi(t(\tau),z)-1,
\end{equation}
then the Jacobian can be written in terms of $\varphi$ as $J=(\Flm)^2(\Flm+z\pz\varphi)$. Make use of (\ref{General LE}) and the fact that $\lambda^2\pt^2\lambda=\delta$, (\ref{xiEqn}) becomes
\begin{equation}\label{phiEqn}
    \lamt\ptau^2\varphi+\ptau\lamt\ptau\varphi+\delta(1+\varphi)+F_w[\Flm]=\lamt^{3-3\alpha} V_{w,\alpha}[\varphi],
\end{equation}
where $F_w[1+\varphi]$ and $V_{w,\alpha}[\varphi]$ have the form
\begin{equation}\label{Eq: F, G in varphi}
\begin{aligned}
    F_w[1+\varphi]=&\frac{(\Flm)^2}{zw^3}\pz\left(w^4J^{-\frac{4}{3}}\right)-\frac{1}{(\Flm)^2}\cdot\left(\delta+\frac{4w'}{z}\right),\\
    V_{w,\alpha}[\varphi]:=&G_{w,\alpha}[1+\varphi]\\
    =&\frac{1}{z^4w^3(\Flm)}\pz\left(\frac{4}{3}\mu z^3(1+\varphi)^2(w^3J^{-1})^\alpha\frac{(1+\varphi)z\ptau\pz\varphi-\ptau\varphi\cdot z\pz\varphi}{1+\varphi+z\pzphi}\right).\\
\end{aligned}
\end{equation}
What's more, boundary condition (\ref{BC, general}) can be written as
    \begin{equation}\label{BC, phi}
        \left.\frac{4}{3}\mu(w^3J^{-1})^\alpha\frac{(1+\varphi)z\ptau\pz\varphi-z\pz\varphi\ptauphi}{1+\varphi+z\pz\varphi}\right|_{z=1}=0.
    \end{equation}

Under $(\tau,z)$ coordinates, the physical energy-dissipation identity admits the expression
\begin{equation}\label{Eq: physical energy identity}
    \begin{aligned}
    E(1+\varphi,\ptauphi)(\tau)+\int_0^\tau D(\varphi,\ps\varphi)(s)ds=E(1+\varphi_0,\varphi_1),
    \end{aligned}
\end{equation}
where
\begin{equation}
    \begin{aligned}
        E(1+\varphi,\ptauphi)=&\intol\frac{1}{2}\left(\frac{\lamt'}{\lamt}(\Flm)+\ptauphi\right)^2w^3z^4dz\\
        &+\frac{1}{\lamt}\intol\left(3w^4z^2J^{-\frac{1}{3}}+\frac{4w'}{z}\frac{w^3z^4}{\Flm}\right)dz+\frac{1}{\lamt}\intol\frac{\delta w^3z^4}{\Flm}dz,
    \end{aligned}
\end{equation}
and
\begin{equation}
    D(\varphi,\ptauphi)=\lamt^{2-3\alpha}\intol\frac{4}{3}\mu(w^3J^{-1})^{\alpha}z^2\frac{[(\Flm)z\ptauzphi-\ptauphi\cdot z\pzphi]^2}{\Flm+z\pzphi}dz.
\end{equation}
In particular, the physical energy and dissipation of the GW solution are $E(1,0)=\intol\frac{1}{2}\tilde{e}w^3z^4dz$ and $D(1,0)=0$ respectively.

The next goal is to rewrite (\ref{phiEqn}) into linear part and nonlinear part. For the Jacobian $J$, we have 
\begin{equation}
    J=(\Flm)^2(\Flm+z\pz\varphi)=1+\frac{1}{z^2}\pz\left(z^3(\varphi+\varphi^2+\frac{1}{3}\varphi^3)\right).
\end{equation}
For the left-hand side of (\ref{phiEqn}), apply Taylor's formula to $J^{-\frac{4}{3}}$ to obtain
\begin{equation}
    J^{-\frac{4}{3}}=1-\frac{4}{3}(J-1)+(J-1)^2\intol\frac{28}{9}\left(1+\theta(J-1)\right)^{-\frac{10}{3}}(1-\theta)d\theta.
\end{equation}
Insert it into (\ref{Eq: F, G in varphi})$_1$, we get
\begin{equation}\label{Eq: F_w in varphi}
    \begin{aligned}
        F_w[1+\varphi]=-\delta+2\delta\varphi+\mathcal{L}_{w^4}[\varphi]+\mathcal{N}_1[\varphi],
    \end{aligned}
\end{equation}
where we denote
\begin{equation}\label{Important linear operator}
    \mathcal{L}_{w^k}\varphi:=-\frac{4}{3}\cdot\frac{1}{z^4w^3}\pz(w^kz^4\pzphi)
\end{equation}
and
\begin{equation}\label{Eq: nonlinearity, N1}
    \begin{aligned}
        \mathcal{N}_1[\varphi]=&\delta\left(1-2\varphi-\frac{1}{(\Flm)^2}\right)-\frac{4}{3}\cdot\frac{(\Flm)^2-1}{zw^3}\pz\left(\frac{w^4}{z^2}\pz\left(z^3\varphi\right)\right)\\
        &-\frac{4}{3}\cdot\frac{(\Flm)^2}{zw^3}\pz\left[\frac{w^4}{z^2}\pz\left(z^3(\varphi^2+\frac{1}{3}\varphi^3)\right)\right]\\
        &+\frac{(\Flm)^2}{zw^3}\pz\left(\frac{28}{9}w^4(J-1)^2\intol(1-\theta)[1+\theta(J-1)]^{-\frac{10}{3}}d\theta\right)\\
        &+\frac{4w'}{z}\left((\Flm)^2-\frac{1}{(\Flm)^2}-4\varphi\right).\\
         \end{aligned}
\end{equation}
Inserting (\ref{Eq: F_w in varphi}) into (\ref{phiEqn}), we obtain the following form of the equation to be used in the temporal and interior estimates
\begin{equation}\label{phiEqn: linearized, G}
    \begin{aligned}
        \lamt\ptau^2\varphi+\lamt'\ptau\varphi+3\delta\varphi+\mathcal{L}_{w^4}\varphi+\mathcal{N}_1[\varphi]=\lamt^{3-3\alpha} V_{w,\alpha}[\varphi].
    \end{aligned}
\end{equation}
Further, the viscosity term (\ref{Eq: F, G in varphi})$_2$ can be written as 
\begin{equation}\label{Eq: Linearized, G term}
    \begin{aligned}
         V_{w,\alpha}[\varphi]=-\mu\mathcal{L}_{w^{3\alpha}}\ptauphi+\mathcal{N}_2[\varphi],
    \end{aligned}
\end{equation}
where
\begin{equation}\label{Eq: N_2}
    \begin{aligned}
        \mathcal{N}_2[\varphi]=&\frac{1}{z^4w^3}\left(\frac{1}{\Flm}-1\right)\pz\left(\frac{4}{3}\mu z^4w^{3\alpha}\ptau\pz\varphi\right)\\
        &+\frac{1}{z^4w^3(\Flm)}\pz\left(\frac{4}{3}\mu z^3w^{3\alpha}(J^{-\alpha-1}-1)z\ptau\pz\varphi\right)\\
        &+\frac{1}{z^4w^3(\Flm)}\pz\left(\frac{4}{3}\mu z^3w^{3\alpha}J^{-\alpha-1}([(1+\varphi)^5-1]z\ptau\pz\varphi-(\Flm)^4z\pz\varphi\ptauphi)\right).
    \end{aligned}
\end{equation}
We then derive the equation to be analyzed in the elliptic estimate
\begin{equation}\label{phiEqn: linearized}
    \begin{aligned}
        \lamt\ptau^2\varphi+\lamt'\ptau\varphi+3\delta\varphi+\mathcal{L}_{w^4}\varphi+\mathcal{N}_1[\varphi]=-\mu\lamt^{3-3\alpha}\mathcal{L}_{w^{3\alpha}}\ptauphi+\lamt^{3-3\alpha}\mathcal{N}_2[\varphi].
    \end{aligned}
\end{equation}

\subsection{Main results}
To characterize the stability of the linearly expanding solution, we introduce the total energy:
\begin{equation}\label{total energy}
        \begin{aligned}
            \mathfrak{E}(\tau)=\mathcal{E}_{\rm tem}(\tau)+\mathcal{E}_{\rm int}(\tau)+\mathcal{E}_{\rm spt}(\tau),
        \end{aligned}
    \end{equation}
where
\begin{equation}
    \begin{aligned}
       \mathcal{E}_{\rm tem}(\tau)=&\intol\lamt(\ptau\varphi)^2w^3z^4+(z\pzphi)^2w^4z^2+\varphi^2w^3z^4dz\\
       &+\lamt^{3\alpha-2}\left(\intol\lamt(\ptau^{2}\varphi)^2w^3z^4+(z\ptau\pzphi)^2w^4z^2+(\ptau\varphi)^2w^3z^4dz\right),\\
       \mathcal{E}_{\rm int}(\tau)=&\lamt^{3\alpha-2}\intol\lamt\hat{\chi}(\ptau\varphi)^2w^3z^2+\hat{\chi}(z\pzphi)^2w^4+\hat{\chi}\varphi^2w^3z^2dz\\
       &+\lamt^{6\alpha-4}\left(\intol\lamt\hat{\chi}(\ptau^{2}\varphi)^2w^3z^2+\hat{\chi}(z\ptau\pzphi)^2w^4dz+\hat{\chi}(\ptau\varphi)^2w^3z^2dz\right),\\
       \mathcal{E}_{\rm spt}(\tau)=&\lamt^{3\alpha-2}\intol\mu\left(\frac{1}{w^3z^4}\pz\left(w^{3\alpha}z^4\pz\varphi\right)\right)^2w^{7-3\alpha}z^{2}dz\\
       &+\intol\mu^2\left(\frac{1}{w^3z^4}\pz\left(w^{3\alpha}z^4\pzphi\right)\right)^2w^{3}z^{2}dz,
    \end{aligned}
\end{equation}
and the total dissipation
\begin{equation}
    \begin{aligned}
        \mathfrak{D}(\tau)=\mathcal{D}_{\rm tem}(\tau)+\mathcal{D}_{\rm int}(\tau)+\mathcal{D}_{\rm spt}(\tau),
    \end{aligned}
\end{equation}
where
\begin{equation}
        \begin{aligned}
        \mathcal{D}_{\rm tem}(\tau)=&\lamt\intol(\ptau\varphi)^2w^3z^4dz+\lamt^{3-3\alpha}\intol\mu[(\Flm)z\ptau\pzphi-\ptauphi\cdot z\pzphi]^2w^{3\alpha}z^2dz\\
        &+\lamt^{3\alpha-2}\left(\lamt\intol(\ptau^2\varphi)^2w^3z^4dz+\lamt^{3-3\alpha}\intol\mu[(\Flm)z\ptau^2\pzphi-\ptau^2\varphi\cdot z\pzphi]^2w^{3\alpha}z^2dz\right),\\
        \mathcal{D}_{\rm int}(\tau)=&\lamt\intol\left(\mu\hat{\chi}(\ptauphi)^2+\mu\hat{\chi}(z\ptau\pzphi)^2\right)w^{3\alpha}dz+\lamt^{3\alpha-1}\intol\left(\mu\hat{\chi}(\ptau^2\varphi)^2+\mu\hat{\chi}(z\ptau^2\pz\varphi)^2\right)w^{3\alpha}dz,\\
        \mathcal{D}_{\rm spt}(\tau)=&\lamt\intol\mu^2\left(\frac{1}{z^4w^3}\pz\left(w^{3\alpha}z^4\ptau\pzphi\right)\right)^2w^{3}z^{2}dz\\
        &+\lamt^{6\alpha-5}\intol\left(\frac{1}{w^{3}z^{4}}\pz\left(w^{3\alpha}z^4\pzphi\right)\right)^2w^{11-6\alpha}z^2 dz.
        \end{aligned}
\end{equation}
The function $\hat{\chi}$ appearing in $\mathcal{E}_{\rm int}(\tau)$ and $\mathcal{D}_{\rm int}(\tau)$ is a smooth cutoff defined as
\begin{equation}\label{Eq: Cutoff}
    \hat{\chi}(z)=\begin{cases}
        1\ \text{if } z\in[0,\frac{1}{2}],\\
        0\ \text{if } z\in[\frac{3}{4},1].
    \end{cases}
\end{equation}
Next, we define the strong solution to the problem (\ref{phiEqn})-(\ref{BC, phi}).
\begin{mydef}
Let $T>0$. A function $\varphi$ is called a strong solution to
(\ref{phiEqn})--(\ref{BC, phi}) on $[0,T]$ if all the derivatives
appearing below are understood in the distributional sense and
\begin{equation}
    \operatorname*{ess\,sup}_{\tau\in[0,T]}
    \mathfrak{E}(\tau)
    +\int_0^T\mathfrak{D}(s)ds<\infty,
\end{equation}
as well as
\begin{equation}
    \varphi,\ \partial_\tau\varphi
    \in C([0,T];L^2([0,1],w^3z^4dz)),
\end{equation}
\begin{equation}
   \pz\varphi,\ \partial_\tau\pz\varphi
    \in C([0,T];L^2([0,1],w^4z^4dz)),
\end{equation}
\begin{equation}
    \varphi,\ z\pz\varphi,\ 
    \ptau\varphi,\  z\ptau\pz\varphi
    \in L^\infty((0,T)\times(0,1)),
\end{equation}
and there exists a constant $c_T>0$ such that
\begin{equation}
    1+\varphi\ge c_T,\qquad
    1+\varphi+z\pz\varphi\ge c_T
\end{equation}
almost everywhere in $(0,T)\times(0,1)$. Moreover,
(\ref{phiEqn}) holds almost everywhere in
$(0,T)\times(0,1)$, and (\ref{BC, phi}) holds in the
sense of traces for almost every $\tau\in(0,T)$. A function $\varphi$ is called a global strong solution if it is a
strong solution on $[0,T]$ for every $T>0$.
\end{mydef}
Now, we are ready to state the main theorem.
\begin{theorem}\label{Thm: stability of GW, LG}Suppose $-\tilde\varepsilon<\delta$, where
$\tilde\varepsilon>0$ is a fixed sufficiently small constant. There exists a constant $\varepsilon_0>0$ such that if the initial data $(\varphi(0),\ptauphi(0))=(\varphi_0,\varphi_1)$ satisfies
\begin{equation}
    \mathfrak{E}(0)\le\varepsilon_0
\end{equation}
and the boundary condition
\begin{equation}
    \left.\frac{4}{3}\mu(w^3J_0^{-1})^\alpha\frac{(1+\varphi_0)z\pz\varphi_1-z\pz\varphi_0\varphi_1}{1+\varphi_0+z\pz\varphi_0}\right|_{z=1}=0,
\end{equation}
then there exists a strong solution $\varphi$ to the problem (\ref{phiEqn})-(\ref{BC, phi}) on $[0,\infty)\times[0,1]$ with the estimate
\begin{equation}\label{Eq: Total energy est}
    \mathfrak{E}(\tau)+\int_0^\tau\mathfrak{D}(s)ds\le C\mathfrak{E}(0)+C\int_0^\tau\lamt^{-\min\{\frac{3}{2}\alpha,\frac{1}{2}\}}\mathfrak{E}(s)ds
\end{equation}
for any $\tau\in[0,\infty)$.
\end{theorem}

We illustrate the main ideas of the proof. The first step is to prove the high-order weighted temporal estimates. For $j=0,1$, we take $\ptau^j$ of (\ref{phiEqn: linearized, G}) and test the resulting equation with $\lamt^{j(3\alpha-2)}w^3\ptau^{j+1}\varphi$ to prove the bound of
\begin{equation}\label{Eq: Time est, simple}
    \begin{aligned}
        &\sum_{j=0}^1\lamt^{j(3\alpha-2)}\left(\intol\lamt(\ptau^{j+1}\varphi)^2w^3z^4+\frac{4}{3}(z\ptau^j\pzphi)^2w^4z^2+3\delta(\ptau^{j}\varphi)^2w^3z^4dz\right)\\
        &+\sum_{j=0}^1\int_0^\tau\lamt^{j(3\alpha-2)}\left(\intol\lamt(\ps^{j+1}\varphi)^2w^3z^4+\lamt^{3-3\alpha}\mu[(\Flm)z\ps^{j+1}\pzphi-\ps^{j+1}\varphi\cdot z\pzphi]^2w^{3\alpha}z^2dz\right)ds.
    \end{aligned}
\end{equation}
For $\delta>0$, (\ref{Eq: Time est, simple}) is automatically positive definite, thus the estimates of $\mathcal{E}_{\rm tem}(\tau)$ and $\int_0^\tau\mathcal{D}_{\rm tem}(s)ds$ are implied. While for $\delta\le0$, (\ref{Eq: Time est, simple}) may not be positive definite. Fortunately, the spectrum property of the operator $\mathcal{L}_{w^4}$ allows us to prove coercivity of the temporal energy functional for $\delta>-\tilde{\varepsilon}$ with a sufficiently small $\tilde{\varepsilon}>0$, which leads to a control of $\mathcal{E}_{\rm tem}(\tau)$ and $ \int_0^\tau\mathcal{D}_{\rm tem}(s)ds$ by $\mathfrak{E}(0)+\int_0^\tau\lamt^{-\min\{\frac{3}{2}\alpha,\frac{1}{2}\}}\mathfrak{E}(s)ds+\int_0^\tau\lamt^{-3\alpha}\mathcal{D}_{\rm spt}(s)ds$. The second step is to prove the interior estimates. By introducing the smooth cutoff function $\hat{\chi}$ which vanishes in a neighborhood of the free boundary, we obtain the following better control of the temporal derivatives near the origin
\begin{equation}\label{Eq: interior Time est, simple}
    \begin{aligned}
        &\sum_{j=0}^1\lamt^{(j+1)(3\alpha-2)}\left(\intol\lamt\hat{\chi}(\ptau^{j+1}\varphi)^2w^3z^2+\hat{\chi}(z\ptau^j\pzphi)^2w^4+\hat{\chi}(\ptau^j\varphi)^2w^3z^2dz\right)\\
        &+\sum_{j=0}^1\int_0^\tau\lamt^{j(3\alpha-2)}\left(\lamt\intol\left(\mu\hat{\chi}(\ps^{j+1}\varphi)^2+\mu\hat{\chi}(z\ps^{j+1}\pzphi)^2\right)w^{3\alpha}dz\right)ds
    \end{aligned}
\end{equation}
by $\mathfrak{E}(0)+\int_0^\tau\lamt^{-\min\{\frac{3}{2}\alpha,\frac{1}{2}\}}\mathfrak{E}(s)ds+\int_0^\tau\lamt^{-3\alpha}\mathcal{D}_{\rm spt}(s)ds$. We point out that in the temporal estimate and the interior estimate, the viscosity operator $V_{w,\alpha}[\varphi]$ must be treated as a whole. After integration by parts, the highest order term, which also carries the highest temporal weight, can be written as a squared term with a good sign, leaving no uncontrolled remainder terms. The third step is the elliptic estimates. The goal is to recover regularities of spatial derivatives. We recall the definition of the operator $\mathcal{L}_{w^4}$ in (\ref{Important linear operator}) and rewrite the operator as follows
\begin{equation*}
    \mathcal{L}_{w^4}\varphi=w^{4-3\alpha}\mathcal{L}_{w^{3\alpha}}\varphi-\frac{4}{3}(4-3\alpha)w'\pzphi.
\end{equation*}
To estimate the highest order term, we make use of the following key structure of (\ref{phiEqn: linearized})
\begin{equation}\label{Eq: key structure of the equation}
    \begin{aligned}
        -\mu\lamt^{\frac{5}{2}-3\alpha}\mathcal{L}_{w^{3\alpha}}\ptauphi-\lamt^{-\frac{1}{2}}w^{4-3\alpha}\mathcal{L}_{w^{3\alpha}}\varphi=\lamt^{\frac{1}{2}}\ptau^2\varphi+\text{lot order or nonlinear terms}.
    \end{aligned}
\end{equation}
To be specific, taking the square of the equation and integrating it after multiplying the weight $\lamt^{6\alpha-4}w^3z^2$, we get
\begin{equation}\label{Eq: Elliptic est, simplified}
\begin{aligned}
    &\intol\mu\lamt^{3\alpha-2}(\mathcal{L}_{w^{3\alpha}}\varphi)^2w^{7-3\alpha}z^2dz+\intol\mu^2(\mathcal{L}_{w^{3\alpha}}\varphi)^2w^3z^2dz\\
    &+\int_0^\tau\left(\intol\mu^2\lamt(\mathcal{L}_{w^{3\alpha}}\ps\varphi)^2w^3z^2+\lamt^{6\alpha-5}(w^{4-3\alpha}\mathcal{L}_{w^{3\alpha}}\varphi)^2w^3z^2dz\right)ds\\
    \lesssim&\int_0^\tau\lamt^{6\alpha-3}(s)\left(\intol(\ps^2\varphi)^2w^3z^2dz\right)ds+l.o.t..
\end{aligned}
\end{equation}
Combined with Hardy's inequality, we are able to control $\mathcal{E}_{\rm spt}(\tau)$ and $\int_0^\tau\mathcal{D}_{\rm spt}(s)ds$ by $\mathfrak{E}(0)+\int_0^\tau\lamt^{-\min\{\frac{3}{2}\alpha,\frac{1}{2}\}}\mathfrak{E}(s)ds+\int_0^\tau\lamt^{-3\alpha}\mathcal{D}_{\rm spt}(s)ds$. Note that the term $\int_0^\tau\lamt^{-3\alpha}\mathcal{D}_{\rm spt}(s)ds$ is not yet controlled, we then further rewrite (\ref{Eq: key structure of the equation}) into the form
\begin{equation}\label{Eq: key structure of the equation 2}
    -\mu\lamt^{\frac{5}{2}-3\alpha}\mathcal{L}_{w^{3\alpha}}\ptauphi=\lamt^{-\frac{1}{2}}w^{4-3\alpha}\mathcal{L}_{w^{3\alpha}}\varphi+\lamt^{\frac{1}{2}}\ptau^2\varphi+\text{lot order or nonlinear terms}.
\end{equation}
Making use of the estimates (\ref{Eq: Time est, simple}), (\ref{Eq: interior Time est, simple}) and (\ref{Eq: Elliptic est, simplified}), we obtain the boundedness of
\begin{equation}
    \lamt\intol\mu^2(\mathcal{L}_{w^{3\alpha}}\ptauphi)^2w^3z^2dz,
\end{equation}
namely, $\mathcal{D}_{\rm spt}(\tau)\lesssim\mathfrak{E}(\tau)$. Collecting all the estimates, we prove (\ref{Eq: Total energy est}). The control of the total energy, together with Sobolev embedding as well as Hardy's inequality, implies the following $L^\infty$ bounds on $[0,1]$
\begin{equation}\label{Eq: pointwise bound near origin}
    \|\varphi\|_{L^\infty},\ \lamt^{\frac{1}{2}}(\tau)\|\ptau\varphi\|_{L^\infty},\ \|z\pz\varphi\|_{L^\infty},\ \lamt^{\frac{1}{2}}(\tau)\|z\ptau\pz\varphi\|_{L^\infty},
\end{equation}
which are key quantities to close the nonlinear estimates.

Before proceeding to the proof of the main result, we present several remarks on some related recent results.
\begin{rmk}[Choice of viscosity coefficient]\label{Rmk: Choice of viscosity coefficient}In our problem, we only deal with the special viscosity coefficient $\mu_1(\rho)=\mu\rho^{\alpha}$ with $0<\alpha\le\frac{2}{3}$. On the one hand, the proof of high order spatial regularities rely on the elliptic estimate (\ref{Eq: Elliptic est, simplified}), where we use $\int_0^\tau\mathcal{D}_{\rm tem}(s)+\mathcal{D}_{\rm int}(s)ds$ to control the right-hand side of (\ref{Eq: Elliptic est, simplified}). We thus require $6\alpha-3\le3\alpha-1$, i.e., $\alpha\le\frac{2}{3}$. On the other hand, our method for controlling the negetive part of the temporal energy applies only when $\alpha>0$. Indeed, it produces the term $\int_0^\tau\lamt^{-\min\{\frac{3}{2}\alpha,\frac{1}{2}\}}\mathfrak{E}(s)ds$ on the right-hand side of the differential inequality for the total energy (\ref{Eq: Total energy est}). When $\alpha=0$, this term no longer carries any exponential decay in the integrand, thus cannot be controlled by our argument. We remark that the lower bound of $\delta$ following from our argument is uniform in $\lamt_0$, $\lamt_1$ and $\alpha$. Meanwhile, in the work \cite{Liu2019Isentropic} where constant viscosity case is considered, the author applied H\"older inequality directly and made use of the dissipation coming from the velocity term to control the negative part of the temporal energy, which gave the explicit lower bound $\delta>-\frac{\lamt_0\lamt_1^2}{8}$. 
\end{rmk}
\begin{rmk}[Comparison with the inviscid problem, linearly expanding case]\label{Rmk: 1}
    In the work \cite{HJ}, one of the results is the stability of linearly expanding solutions (case 1, case 2 with $\lambda_1>0$ and case 3 (b) in Proposition \ref{Proposition: General LE, lambda}) with $-\tilde{\varepsilon}<\delta$ for a small enough $\tilde{\varepsilon}>0$ for the Euler-Poisson system. In the negative $\delta$ regime, the authors made use of the spectral gap property of the associated operator $\mathcal{L}_{w^4}$ to obtain coercivity of the total energy. However, this approach leaves a weighted average term to be controlled separately, and the control of this term relies on estimates for the perturbation of the physical energy
    \begin{equation}\label{Eq: perturbation of physical energy, LE}
        |E(1,0)-E(1+\varphi,\ptau\varphi)|,
    \end{equation}
    and the energy conservation in the Euler-Poisson case turns it into the control of 
    \begin{equation}
        |E(1,0)-E(1+\varphi_0,\varphi_1)|,
    \end{equation}
    which is directly controlled by the initial total energy. While for the Navier-Stokes-Poisson system, there is a dissipation
    \begin{equation*}
        E(1+\varphi_0,\varphi_1)-E(1+\varphi,\ptau\varphi)=\int_0^\tau D(\varphi,\ps\varphi)ds
    \end{equation*}
    in the physical energy-dissipation identity. With the choice of $0<\alpha\le\frac{2}{3}$ in this paper, we are able to prove a strong decay of the higher order spatial derivative, which provides a control of the dissipation in the physical energy-dissipation identity, thus we obtain the coercivity of $\tilde{\mathcal{E}}_{\rm tem}(\tau)$ defined in (\ref{energy functional tilde E}).
\end{rmk}
\begin{rmk}[Discussion on self-similar expanding solutions]
    For the self-similar expanding solution (case 3 (a) in Proposition \ref{Proposition: General LE, lambda}), $\delta$ is negative. To reformulate the system (\ref{chiEqn}), recalling the expression of $\lambda(t)$ in (\ref{Eq: self-similar lambda}), letting $\tau(t)=\int_0^t\frac{1}{\lambda^{\frac{3}{2}}(\sigma)}d\sigma$ and denoting $b:=\frac{\ptau\lamb}{\lamb}=\sqrt{|2\delta|}$, we have $\lamb(\tau)=e^{b\tau}$. Let $\xi(\tau,z)=\frac{\chi(\tau,z)}{\lamb(\tau)}$ and $\psi(\tau,z):=\xi(t(\tau),z)-1$, making use of (\ref{General LE}) and the fact that $\lambda^2\pt^2\lambda=\delta$, we arrive at
\begin{equation}\label{Eq: self-similar expanding solution, reformulated}
            \begin{aligned}
                \ptau^2\psi+\frac{1}{2}b\ptau\psi+\delta(1+\psi)+F_w[1+\psi]=\lamb^{\frac{5}{2}-3\alpha} V_{w,\alpha}[\psi].
            \end{aligned}
    \end{equation}
The physical energy-dissipation identity can be written as
\begin{equation}
    \bar{E}(1+\psi,\ptau\psi)+\int_0^\tau\bar{D}(\psi,\ptau\psi)ds=\bar{E}(1+\psi_0,\psi_1),
\end{equation}
where
\begin{equation}\label{Eq: physical energy, self-similar expanding solution}
\begin{aligned}
    \bar{E}(1+\psi,\ptau\psi)=&\intol\frac{1}{\lamb}\frac{1}{2}\left(b(\flm)+\ptau\psi\right)^2w^3z^4dz\\
    &+\intol\frac{1}{\lamb}\left(3w^4z^2J^{-\frac{1}{3}}+\frac{4w'}{z}\frac{w^3z^4}{\flm}\right)dz+\intol\frac{1}{\lamb}\frac{\delta w^3z^4}{\flm}dz
\end{aligned}
\end{equation}
and
\begin{equation}
    \begin{aligned}
        \bar{D}(\psi,\ptau\psi)=\intol\frac{4}{3}\mu\lamb^{\frac{3}{2}-3\alpha}z^2(w^3J^{-1})^{\alpha}\frac{\left[(1+\psi)z\ptau\pz\psi-\ptau\psi\cdot z\pz\psi\right]^2}{1+\psi+z\pz\psi}dz.
    \end{aligned}
\end{equation}
    For the Euler-Poisson case, that is, replacing the right-hand side of (\ref{Eq: self-similar expanding solution, reformulated}) by 0, because of energy conservation, the perturbed solution stays in the level set
    \begin{equation}\label{Eq: zero energy condition, Self-similar}
        \bar{E}(1+\psi,\psi)=E(1,0)=0
    \end{equation}
    if we assume $\bar{E}(1+\psi_0,\psi_1)=\bar{E}(1,0)=0$ initially. In the work \cite{HJ}, the authors proved the stability of self-similar expanding solution for $-\tilde{\varepsilon}<\delta<0$ with $\tilde{\varepsilon}>0$ small enough, where perturbation is assumed to have zero physical energy. While in the Navier-Stokes-Poisson case, due to the dissipation, we don't expect the same phenomenon, which makes the approach in \cite{HJ} not directly applicable. To be specific, the velocity term in (\ref{Eq: physical energy, self-similar expanding solution}) may exhibit exponential growth in time. Under the assumption of zero initial physical energy, adapting the coercivity argument from the Euler-Poisson case requires us to control
    \begin{equation}
        \lamb(\tau)|\bar{E}(1+\psi,\ptau\psi)-\bar{E}(1,0)|=\lamb(\tau)\int_0^\tau\bar{D}(\psi,\ps\psi)ds.
    \end{equation}
    However, such a control is not immediate due to the growing factor $\lamb(\tau)$. In \cite{Liu2019Isentropic}, for the constant-viscosity problem, the author showed that, for perturbations that are nonconstant in the spatial variable, the background solution is unstable if $\bar E(1+\psi_0,\psi_1)\le0$, in the sense that, for any strong solution $\psi$, the quantity $\|\psi\|_{L^\infty([0,1])}+\|z\pz\psi\|_{L^\infty([0,1])}$ cannot remain small for arbitrarily long times. We point out that the same instability conclusion remains valid in the density-dependent degenerate viscosity setting, since the argument in \cite[Subsection~3.2]{Liu2019Isentropic} relies only on the physical energy-dissipation identity and the non-negativity of the dissipation.
\end{rmk}
\begin{rmk}[Comparison with other works on the viscous problem]
    In the works \cite{LXZ1,LXZ2} where the authors proved the stability of Lane-Emden equation for $\gamma\in(\frac{4}{3},2)$, they mainly relied on the range of $\gamma$ to prove coercivity. While in our problem, the expansion of background solution plays an important role. One aspect has been pointed out in Remark \ref{Rmk: 1}, where we discuss making use of the temporal weight to gain coercivity for a small range of negative $\delta$. Another aspect is the coercivity of the higher-order temporal dissipation $\mathcal{D}_{\rm tem}(\tau)$. The restriction of zero bulk viscosity prevents us from gaining coercivity merely from the viscosity term $\lamt^{3-3\alpha} V_{w,\alpha}[\varphi]$, which is highly nonlinear. But the coercivity of $\mathcal{D}_{\rm tem}(\tau)$ can be gained from the combination of terms
    $$-\lamt'\ptauphi+\lamt^{3-3\alpha} V_{w,\alpha}[\varphi].$$
    However, due to the mismatch of temporal weights in these two terms, except for the case $\alpha=\frac{2}{3}$, we only recover
    $$\int_0^\tau\lamt(s)\intol\left((\ps\varphi)^2+(z\ps\pz\varphi)^2\right)w^{3\alpha}z^2dzds\lesssim\int_0^\tau\mathcal{D}_{\rm tem}(s)ds$$
    from the low order temporal estimate. As a result, when it comes to the high order temporal estimate, there has to be a loss of temporal weight by $\lamt^{3\alpha-2}(\tau)$ if $0<\alpha<\frac{2}{3}$. For the temporal weights in the interior energy estimate and the spatial energy estimate, they are determined by two related considerations: the temporal weights inherited from the preceding estimates, and the mismatch between the temporal weights of the velocity and viscosity terms in the equation.

    As for the functional framework, our approach is inspired by \cite{LXZ1,LXZ2,LXin,Liu2019Isentropic}, especially the work \cite{LXin,Liu2019Isentropic}. We follow a similar hierarchy of $L^\infty$ in time estimates of the energy and dissipation functionals to \cite{LXin,Liu2019Isentropic}, namely,
    \begin{equation*}
        \mathcal{E}_{\rm tem}(\tau)\rightarrow\mathcal{E}_{\rm int}(\tau)\rightarrow\mathcal{E}_{\rm spt}(\tau)\rightarrow\mathcal{D}_{\rm spt}(\tau),
    \end{equation*}
    although in \cite{LXin} the author was considering a different system called the gravitational Navier-Stokes-Fourier-Poisson system. Meanwhile, in \cite{LXZ2}, the authors employed the key multipliers $r^3-x^3$ together with several variants such as $\int_0^x\bar{\rho}^{-\beta}(r^3-x^3)_ydy$, in their notations. By exploiting the structure of the viscosity term, they were able to extract a perfect time derivative of a coercive quantity, thereby obtaining improved estimates in terms of spatial weight near the free boundary. They then rewrote the system as a transport equation and derived $L^\infty$ estimates of $r_x-1,\ v\text{ and}\ v_{x}$ near the free boundary. We point out that an analogous approach to our problem is not applicable because of the growing temporal weight in the viscosity term of (\ref{phiEqn: linearized, G}). Alternatively, to derive the $L^\infty$ bounds for the quantities $\varphi,\ z\pzphi,\ \ptauphi
    \text{ and }z\ptauzphi$, we linearize the viscosity operator and rewrite the system as (\ref{phiEqn: linearized, G}). By exploiting the spatial weight structure of the linear operator $\lamt^{3-3\alpha}\mathcal{L}_{w^{3\alpha}}\ptauphi$ in the elliptic estimate and applying both Hardy's inequalities (\ref{hardy 1}) and (\ref{hardy 2}) under different ranges of $k$, we obtain the desired $L^\infty$ bounds.
\end{rmk}
For the sake of clarity, we do not explicitly track the dependence of the coefficient $\mu$ in the estimates. We emphasize, however, that the estimates in this paper are not uniform in $\mu$. Therefore, our framework does not allow us to study the inviscid limit. 

In the rest of this article, we will frequently use the following Hardy's inequality. See, for example, \cite{KMP} for a proof.

\begin{lem}[Hardy's inequality]
    Let $k$ be a real number and let $g$ be a function satisfying $\intol d^k(g^2+g'^2)ds<\infty$. Then
\begin{enumerate}
    \item If $k>1$, then there holds 
    \begin{equation}\label{hardy 1}
        \intol d^{k-2}g^2ds\le C\intol d^k(g^2+g'^2)ds,
    \end{equation}
    \item If $k<1$, then $g$ has a trace at $s=0$ and $1$, and there holds
    \begin{equation}\label{hardy 2}
        \int_0^{\frac{1}{2}} d^{k-2}(g-g(0))^2ds+\int_{\frac{1}{2}}^1d^{k-2}(g-g(1))^2ds\le C\intol d^kg'^2ds.
    \end{equation}
\end{enumerate}
\end{lem}

\subsection{A priori assumptions}
We make the following a priori assumptions for all $\tau\in[0,\infty)$:
\begin{equation}\label{a priori assumptions}
    \begin{aligned}
    &\mathfrak{E}(\tau)\le\omega_0^2,\\
\end{aligned}
\end{equation}
as well as
\begin{equation}\label{a priori assumptions 2}
    |\varphi|,\ |z\pz\varphi|,\ |\ptau\varphi|,\ |z\ptau\pz\varphi|\le\omega_0,
\end{equation}
for small enough $\omega_0>0$. And we let
\begin{equation}
    \omega_1=\sup_{\tau}\{|\varphi|,\ |z\pz\varphi|,\ |\ptau\varphi|,\ |z\ptau\pz\varphi|\}.
\end{equation}

\section{Temporal derivatives estimates}
In this section, we aim to establish estimates for temporal derivatives of $\varphi$. The functional
\begin{equation}\label{energy functional tilde E}
    \begin{aligned}
        \tilde{\mathcal{E}}_{\rm tem}(\tau)=&\frac{1}{2}\sum_{j=0}^1\lamt^{j(3\alpha-2)}\left(\intol\lamt(\ptau^{j+1}\varphi)^2w^3z^4+\frac{4}{3}(z\ptau^j\pzphi)^2w^4z^2+3\delta(\ptau^{j}\varphi)^2w^3z^4dz\right)
    \end{aligned}
\end{equation}
arises naturally while we apply $\ptau$ to the equation and take inner product with the corresponding temporal derivative of $\varphi$. But notice that when $\delta<0$, $\tilde{\mathcal{E}}_{\rm tem}(\tau)$ may not be positive definite. To obtain coercivity, we make use of the spectrum of the operator $\mathcal{L}_{w^4}$ as in \cite[Lemma 4.2]{HJ} to prove the coercivity of $\tilde{\mathcal{E}}_{\rm tem}(\tau)$ for $-\tilde{\varepsilon}<\delta\le0$ with $\tilde{\varepsilon}>0$ small enough. To be precise, the operator $\mathcal{L}_{w^4}$ admits the expression (\ref{Important linear operator}) for $\varphi\in C_c^{\infty}([0,1])$, and it has the following Friedrichs extension
\begin{equation}
    \begin{aligned}
        \mathcal{L}_{w^4}:\mathbf{H}_w^2\rightarrow \mathbf{L}^2,
    \end{aligned}
\end{equation}
where
\begin{equation}
    \begin{aligned}
        &\mathbf{L}^2:=L^2([0,1],w^3z^4dz),\\
        &\mathbf{H}_{w}^j:=\{\varphi\in\mathbf{L}^2:\pz^k\varphi\in L^2([0,1],w^{3+k}z^4dz)\text{ for }0\le k\le j\}\ (j\ge0)
    \end{aligned}
\end{equation}
In \cite[Appendix A]{HJ}, it is proved that $\mathcal{L}_{w^4}$ is self-adjoint under the inner product of $\mathbf{L}^2$ and the spectrum of $\mathcal{L}_{w^4}$ is a discrete sequence $\{\mu_k\}$ with $0=\mu_0<\mu_1<\mu_2<...\rightarrow\infty$. In addition, the eigenfunctions of $\mu_0=0$ are constant functions. As a result, there exists a small enough $\tilde{\varepsilon}>0$ such that for $-\tilde{\varepsilon}<\delta$ and for $\varphi\in\mathbf{H}_w^1$ with zero mean $\intol\varphi w^3z^4dz=0$, 
\begin{equation}\label{Eq: spectral gap}
\begin{aligned}
    &\frac43\int_0^1w^4z^4|\phi_z|^2dz
+3\delta\int_0^1w^3z^4|\phi|^2dz\\
\ge& c_0\left(\|w^{\frac{3}{2}}z^2\varphi\|_{L^2([0,1])}^2+\|w^2z^2\pz\varphi\|_{L^2([0,1])}^2\right)
\end{aligned}
\end{equation}
holds for a fixed positive constant $c_0>0$ independent of $\delta$.

We begin with following lemma showing coercivity for the temporal dissipation.
\begin{lem}\label{Lemma: Coercivity of temporal dissipation} Let $-\tilde\varepsilon<\delta$ with $\tilde\varepsilon>0$ small enough. Suppose (\ref{a priori assumptions}) holds for small constant $\omega_0$. Then
    \begin{equation}\label{Est1-improved 1}
\begin{aligned}
     \int_0^\tau\lamt(s)\intol\left((\ps\varphi)^2+(z\ps\pz\varphi)^2\right)w^{3\alpha}z^2dzds\lesssim\int_0^\tau\mathcal{D}_{\rm tem}(s)ds,\\
      \int_0^\tau\lamt^{3\alpha-1}(s)\intol\left((\ps^2\varphi)^2+(z\ps^2\pzphi)^2\right)w^{3\alpha}z^2dzds\lesssim\int_0^\tau\mathcal{D}_{\rm tem}(s)ds.
\end{aligned}
\end{equation}
\end{lem}
\begin{proof}
    For the low order temporal dissipation, it follows from Hardy's inequality that
\begin{equation}
    \begin{aligned}
        &\int_0^\tau\lamt(s)\intol\left((\psphi)^2+(z\ps\pzphi)^2\right)w^{3\alpha}z^2dzds\\
        \le&C\int_0^\tau\lamt(s)\intol(\psphi)^2w^{3}z^4+(z\ps\pzphi)^2w^{3\alpha}z^2dzds\\
        \le&C\int_0^\tau\intol\lamt(\psphi)^2w^{3}z^4ds+\int_0^\tau\intol\lamt(\psphi)^2(z\pz\varphi)^2w^{3\alpha}z^2dzds\\
        &+C\int_0^\tau\intol\lamt^{3-3\alpha} [(\Flm)\cdot z\ps\pzphi-\psphi\cdot z\pz\varphi]^2w^{3\alpha}z^2ds\\
        \le&C\int_0^\tau\mathcal{D}_{\rm tem}(s)ds.
    \end{aligned}
\end{equation}
For the high order temporal dissipation, it holds similarly that
\begin{equation}
    \begin{aligned}
        \int_0^\tau\lamt^{3\alpha-1}(s)\intol\left((\ps^2\varphi)^2+(z\ps^2\pzphi)^2\right)w^{3\alpha}z^2dzdz\lesssim\int_0^\tau\mathcal{D}_{\rm tem}(s)ds.
    \end{aligned}
\end{equation}
\end{proof}

The following Lemma shows coercivity of the energy functional (\ref{energy functional tilde E}).
\begin{lem}\label{Lemma: Coercivity of temporal energy}Let $\delta\in(\delta^*,0]$. For $-\tilde{\varepsilon}<\delta$ with $\tilde{\varepsilon}>0$ small enough, there holds
\begin{equation}
    \begin{aligned}
        \tilde{\mathcal{E}}_{\rm tem}(\tau)
        \ge&\tilde{C}_1\mathcal{E}_{\rm tem}(\tau)-\tilde{C}_2\mathcal{E}_{\rm tem}(0)\\
        &-\tilde{C}_3\int_0^\tau\lamt^{-\frac{3}{2}\alpha}(s)\mathcal{E}_{\rm tem}(s)ds-\frac{\tilde{C}_4}{\sigma}\int_0^\tau\lamt^{-3\alpha}(s)\mathcal{D}_{\rm spt}(s)ds-\sigma\int_0^\tau\mathcal{D}_{\rm tem}(s)ds
    \end{aligned}
\end{equation}
for positive constants $\tilde{C}_1,\tilde{C}_2,\tilde{C}_3$ and a small enough positive constant $\sigma$ to be determined in Lemma \ref{Lemma: Est2}.
\end{lem}
\begin{proof}
We let
\begin{equation}
    \tilde{g}^j=\ptau^j\varphi-\frac{1}{\intol w^3z^4dz}\intol\ptau^j\varphi w^3z^4dz,
\end{equation}
where $j=0,1$, so $\int_0^1\tilde{g}^j w^3z^4dz=0$. For $-\tilde{\varepsilon}<\delta$ with $\tilde{\varepsilon}>0$ small enough, it follows from (\ref{Eq: spectral gap}) that
\begin{equation}
    \intol(\mathcal{L}_{w^4}\tilde{g}^j+3\delta\tilde{g}^j)\tilde{g}^j\cdot w^3z^4dz\ge c_0\left(\intol (\tilde{g}^j)^2w^3z^4dz+\intol (\pz\tilde{g}^j)^2w^4z^4dz\right).
\end{equation}
Meanwhile, by direct computation,
\begin{equation}
   \begin{aligned}
        &\intol(\mathcal{L}_{w^4}\tilde{g}^j+3\delta\tilde{g}^j)\tilde{g}^j\cdot w^3z^4dz\\
        =&\intol\frac{4}{3}(z\ptau^j\pzphi)^2w^4z^2dz+\intol3\delta(\tilde{g}^j)^2w^3z^4dz\\
        =&\intol\frac{4}{3}(z\ptau^j\pzphi)^2w^4z^2dz+\intol3\delta(\ptau^j\varphi)^2w^3z^4dz-3\delta\frac{1}{\intol w^3z^4dz}\left(\intol\ptau^j\varphi w^3z^4dz\right)^2,
    \end{aligned}
\end{equation}
we thus have
\begin{equation}
    \begin{aligned}
        &\intol\frac{4}{3}(z\ptau^j\pzphi)^2w^4z^2dz+\intol3\delta(\ptau^j\varphi)^2w^3z^4dz\\
        \ge& c_0\left(\intol (\ptau^j\varphi)^2w^3z^4dz+\intol (z\ptau^j\pz\varphi)^2w^4z^2dz\right)-(3|\delta|+c_0)\frac{1}{\intol w^3z^4dz}\left(\intol\ptau^j\varphi w^3z^4dz\right)^2.
    \end{aligned}
\end{equation}
In particular, for $j=1$, we multiply $\lamt^{3\alpha-2}$ and apply H\"older inequality to the last term on the right-hand side to get
\begin{equation}
    \begin{aligned}
        &\lamt^{3\alpha-2}\intol\frac{4}{3}(z\ptau\pzphi)^2w^4z^2dz+\lamt^{3\alpha-2}\intol3\delta(\ptau\varphi)^2w^3z^4dz\\
        \ge& c_0\lamt^{3\alpha-2}\left(\intol (\ptau\varphi)^2w^3z^4dz+\intol (z\ptau\pz\varphi)^2w^4z^2dz\right)-C_1\lamt^{3\alpha-2}\intol(\ptau\varphi)^2w^3z^4dz,
    \end{aligned}
\end{equation}
while for $j=0$, we have
\begin{equation}
    \begin{aligned}
        &\intol\frac{4}{3}(z\pzphi)^2w^4z^2dz+\intol3\delta\varphi^2w^3z^4dz\\
        \ge& c_0\left(\intol\varphi^2w^3z^4dz+\intol (z\pz\varphi)^2w^4z^2dz\right)-(3|\delta|+c_0)\frac{1}{\intol w^3z^4dz}\left(\intol\varphi w^3z^4dz\right)^2.
    \end{aligned}
\end{equation}
Now let $d_0,d_1>0$ to be determined, consider the linear combination of the energies
\begin{equation}\label{linear combination of the energies}
    \begin{aligned}
        &\sum_{j=0}^1d_j\lamt^{j(3\alpha-2)}\left(\intol\lamt(\tau)(\ptau^{j+1}\varphi)^2w^3z^4+\frac{4}{3}(z\ptau^j\pz\varphi)^2w^4z^2+3\delta(\ptau^j\varphi)^2w^3z^4dz\right)\\
        \ge&(d_0-C_1d_1\lamt^{3\alpha-3}(\tau))\intol\lamt(\ptau\varphi)^2w^3z^4dz+d_1\lamt^{3\alpha-2}\intol\lamt(\ptau^{2}\varphi)^2w^3z^4dz\\
        &+\sum_{j=0}^1d_jc_0\lamt^{j(3\alpha-2)}\left(\intol(\ptau^j\varphi)^2w^3z^4dz+\intol (z\ptau^j\pz\varphi)^2w^4z^2dz\right)\\
        &-(3|\delta|+c_0)\frac{d_0}{\intol w^3z^4dz}\left(\intol\varphi w^3z^4dz\right)^2.
    \end{aligned}
\end{equation}
In order to control the negative term on the right-hand side of (\ref{linear combination of the energies}), we compute directly to get
\begin{equation}\label{Eq: Linearization of the physical energy, linear expanding soln 0}
    \begin{aligned}
        E(\Flm,\ptauphi)-E(1,0)=&\intol\left[\frac{\lamt'}{\lamt}\left(\frac{\lamt'}{\lamt}\varphi+\ptauphi\right)+\frac{1}{2}\left(\frac{\lamt'}{\lamt}\varphi+\ptauphi\right)^2\right]w^3z^4dz\\
        &+\frac{1}{\lamt}\intol3[\frac{1}{3}(3\varphi+z\pzphi)+(J^{-\frac{1}{3}}-1)]w^4z^2dz\\
        &-\frac{1}{\lamt}\intol\delta\varphi w^3z^4dz+\frac{1}{\lamt}\intol\left(\frac{4w'}{z}+\delta\right)\frac{\varphi^2}{\Flm} w^3z^4dz.
    \end{aligned}
\end{equation}
Using (\ref{Eq: ode of lamt}), the following term on the right-hand side of (\ref{Eq: Linearization of the physical energy, linear expanding soln 0}) can be rewritten as 
\begin{equation}
    \intol\frac{\lamt'}{\lamt}\left(\frac{\lamt'}{\lamt}\varphi+\ptauphi\right)w^3z^4dz=\frac{\lamt'}{\lamt}\intol\ptauphi w^3z^4dz+\left(\tilde{e}-\frac{2\delta}{\lamt}\right)\intol\varphi w^3z^4dz.
\end{equation}
While for the last term on the third line of (\ref{Eq: Linearization of the physical energy, linear expanding soln 0}), it follows from integration by parts that
\begin{equation}
    \begin{aligned}
        &\frac{1}{\lamt}\intol\left(\frac{4w'}{z}+\delta\right)\frac{\varphi^2}{\Flm} w^3z^4dz\\
        =&\frac{1}{\lamt}\intol\frac{\varphi^2}{\Flm} \pz w^4\cdot z^3dz+\frac{1}{\lamt}\intol\frac{\delta\varphi^2}{\Flm} w^3z^4dz\\
        =&-\frac{1}{\lamt}\intol w^4\pz\left(z^3\frac{\varphi^2}{\Flm}\right)dz+\frac{1}{\lamt}\intol\frac{\delta\varphi^2}{\Flm} w^3z^4dz\\
        =&-\frac{1}{\lamt}\intol[\frac{3\varphi^2}{\Flm}+\frac{2\varphi\cdot z\pzphi}{\Flm}-\frac{\varphi^2\cdot z\pzphi}{(\Flm)^2}]w^4z^2dz+\frac{1}{\lamt}\intol\frac{\delta\varphi^2}{\Flm} w^3z^4dz.
    \end{aligned}
\end{equation}
As a result, (\ref{Eq: Linearization of the physical energy, linear expanding soln 0}) can be rewritten as
\begin{equation}\label{Eq: Linearization of the physical energy, linear expanding soln}
    \medmath{
    \begin{aligned}
        &E(\Flm,\ptauphi)-E(1,0)-\underbrace{\frac{\lamt'}{\lamt}\intol\ptauphi w^3z^4dz}_{\star_1}-\left(\tilde{e}-\frac{3\delta}{\lamt}\right)\intol\varphi w^3z^4dz\\
        =&\intol\frac{1}{2}\left(\frac{\lamt'}{\lamt}\varphi+\ptauphi\right)^2w^3z^4dz\\
        &+\underbrace{\frac{1}{\lamt}\intol[3\varphi+z\pzphi-\frac{3\varphi^2}{\Flm}-\frac{2\varphi\cdot z\pzphi}{\Flm}+\frac{\varphi^2\cdot z\pzphi}{(\Flm)^2}+3(J^{-\frac{1}{3}}-1)]w^4z^2dz}_{\star_2}+\frac{1}{\lamt}\intol\frac{\delta\varphi^2}{\Flm} w^3z^4dz.
    \end{aligned}
    }
\end{equation}
In order to control $(\star_1)$, we need the following interpolation inequality
\begin{equation}\label{Eq: interpolation ineq}
\begin{aligned}
    |g(t)|^2\lesssim|g(0)|^2+\|g\|_{L^2_t}\|\pt g\|_{L^2_t}
\end{aligned}
\end{equation}
for $g\in H^1([0,+\infty))$ a function of $t\in[0,+\infty)$, which can be proved by the fundamental theorem of calculus and H\"older inequality. Next, recalling the fact that $b_0\le\frac{\lamt'}{\lamt}\le b_1$, combining with (\ref{Eq: interpolation ineq}), $(\star_1)$ can be controlled as
\begin{equation}\label{Eq: control of 1st order term}
    \begin{aligned}
        (\star_1)\lesssim&\intol|\ptauphi|w^3z^4dz\\
        \lesssim&\lamt^{-\varepsilon}(\tau)\intol\|\lamt^{\varepsilon}(s)\psphi\|_{L^2_s}^{\frac{1}{2}}\left(\|\lamt^{\varepsilon}(s)\psphi\|_{L^2_s}^{\frac{1}{2}}+\|\lamt^{\varepsilon}(s)\ps^2\varphi\|_{L^2_s}^{\frac{1}{2}}\right)w^3z^4dz+\intol|\ptau\varphi(0,z)|w^3z^4dz\\
        \lesssim&\left(\intotl\lamt^{2\varepsilon}(s)(\psphi)^2w^3z^4dzds\right)^{\frac{1}{2}}+\lamt^{-2\varepsilon}(\tau)\left(\intotl\lamt^{2\varepsilon}(s)(\ps^2\varphi)^2w^3z^4dzds\right)^{\frac{1}{2}}+\sqrt{\mathcal{E}_{\rm tem}(0)}\\
        \lesssim&\left(\intotl\lamt^{2\varepsilon}(s)(\psphi)^2w^3z^4dzds\right)^{\frac{1}{2}}+\left(\intotl\lamt^{-2\varepsilon}(s)(\ps^2\varphi)^2w^3z^4dzds\right)^{\frac{1}{2}}+\sqrt{\mathcal{E}_{\rm tem}(0)},\\
    \end{aligned}
\end{equation}
where we use $\lamt(s)\lesssim\lamt(\tau)$ for $s\le\tau$, and the exact value of $\varepsilon$ will be determined later in the proof of this Lemma. For $(\star_2)$, applying Taylor's expansion to $J^{\frac{1}{3}}-1$, we obtain
$$(\star_2)\lesssim\frac{1}{\lamt}\intol(z\pzphi)^2w^4z^2dz.$$
Then it follows from (\ref{Eq: Linearization of the physical energy, linear expanding soln}) that
\begin{equation}
    \begin{aligned}
        &\left|\intol\varphi w^3z^4dz\right|\\
        \lesssim&|E(\Flm,\ptauphi)-E(1,0)|+\intol\left(\varphi^2+(\ptauphi)^2\right)w^3z^4dz+\frac{1}{\lamt}\intol(z\pzphi)^2w^4z^2dz\\
        &+\left(\intotl\lamt^{2\varepsilon}(s)(\psphi)^2w^3z^4dzds\right)^{\frac{1}{2}}+\left(\intotl\lamt^{-2\varepsilon}(s)(\ps^2\varphi)^2w^3z^4dzds\right)^{\frac{1}{2}}.
    \end{aligned}
\end{equation}
Combining with the following estimate
\begin{equation}
    \begin{aligned}
        &|E(\Flm,\ptauphi)-E(1,0)|\\
        =&\left| E(1+\varphi_0,\varphi_1)-E(1,0)\right.\\
        &\left.\ \ \ \ \ -\int_0^\tau\lamt^{2-3\alpha}(s)\intol\frac{4}{3}\mu\left(w^3J^{-1}\right)^{\alpha}z^2\frac{[(\Flm)z\ps\pzphi-\ps\varphi\cdot z\pzphi]^2}{\Flm+z\pzphi}dzds\right|\\
        \lesssim&\sqrt{\mathcal{E}_{\rm tem}(0)}+\int_0^\tau\lamt^{2-3\alpha}(s)\intol\mu[(\Flm)z\ps\pzphi-\ps\varphi\cdot z\pzphi]^2w^{3\alpha}z^2dzds,
    \end{aligned}
\end{equation}
where we use the physical energy-dissipation identity (\ref{Eq: physical energy identity}), we obtain
\begin{equation}
   \begin{aligned}
        &\left|\intol\varphi w^3z^4dz\right|\\
        \lesssim&\sqrt{\mathcal{E}_{\rm tem}(0)}+\intol\left(\varphi^2+(\ptauphi)^2\right)w^3z^4dz+\frac{1}{\lamt}\intol(z\pzphi)^2w^4z^2dz\\
        &+\int_0^\tau\lamt^{2-3\alpha}(s)\intol\mu[(\Flm)z\ps\pzphi-\ps\varphi\cdot z\pzphi]^2w^{3\alpha}z^2dzds\\
        &+\left(\intotl\lamt^{2\varepsilon}(s)(\psphi)^2w^3z^4dzds\right)^{\frac{1}{2}}+\left(\intotl\lamt^{-2\varepsilon}(s)(\ps^2\varphi)^2w^3z^4dzds\right)^{\frac{1}{2}}.
    \end{aligned}
\end{equation}
So
\begin{equation}\label{Eq: Temporal coercivity 10}
\medmath{
    \begin{aligned}
        &\sum_{j=0}^1d_j\lamt^{j(3\alpha-2)}\left(\lamt(\tau)\intol(\ptau^{j+1}\varphi)^2w^3z^4dz+\intol\frac{4}{3}(z\ptau^j\pz\varphi)^2w^4z^2dz+\intol3\delta(\ptau^j\varphi)^2w^3z^4dz\right)\\
        \ge&(d_0-C_1d_1\lamt^{3\alpha-3}(\tau))\intol\lamt(\ptau\varphi)^2w^3z^4dz+d_1\lamt^{3\alpha-2}\intol\lamt(\ptau^{2}\varphi)^2w^3z^4dz\\
        &+\sum_{j=0}^1d_jc_0\lamt^{j(3\alpha-2)}\left(\intol(\ptau^j\varphi)^2w^3z^4dz+\intol (z\ptau^j\pz\varphi)^2w^4z^2dz\right)-C_3d_0(\mathcal{E}_{\rm tem}(0)+\mathcal{E}^2_{\rm tem}(\tau))\\&
        -\underbrace{d_0C_2\left(\intotl\lamt^{2\varepsilon}(s)(\psphi)^2w^3z^4dzds+\intotl\lamt^{-2\varepsilon}(s)(\ps^2\varphi)^2w^3z^4dzds\right)}_{\triangle_1}\\
        &-\underbrace{C_4d_0\int_0^\tau\lamt^{2-3\alpha}(s)\intol\mu[(\Flm)z\ps\pzphi-\ps\varphi\cdot z\pzphi]^2w^{3\alpha}z^2dzds}_{\triangle_2}.
    \end{aligned}
    }
\end{equation}
Choosing $d_0,d_1>0$ appropriately, making use of the fact that $3\alpha-2\le 0$, we can ensure that $d_0-C_1d_1\lamt^{3\alpha-3}(\tau)\ge\tilde{C}>0$. The term $(\triangle_2)$ can be controlled as follows
\begin{equation}
\begin{aligned}
     &\int_0^\tau\lamt^{2-3\alpha}(s)\intol\mu[(\Flm)z\ps\pzphi-\ps\varphi\cdot z\pzphi]^2w^{3\alpha}z^2dzds\\
     \le&\int_0^\tau\left(\sigma+\frac{C}{\sigma}\lamt^{-2}(s)\right)\lamt^{3-3\alpha}(s)\intol\mu[(\Flm)z\ps\pzphi-\ps\varphi\cdot z\pzphi]^2w^{3\alpha}z^2dzds\\
     \le&\sigma\int_0^\tau\mathcal{D}_{\rm tem}(s)ds+\frac{C}{\sigma}\int_0^\tau\lamt^{-3\alpha}(s)\mathcal{D}_{\rm spt}(s)ds,
\end{aligned}
\end{equation}
where $\sigma>0$ is chosen small enough and to be determined in Lemma \ref{Lemma: Est2}. To control $(\triangle_1)$, we require
\begin{equation*}
    \begin{cases}
        2\varepsilon<1,\\
        -2\varepsilon<3\alpha-1,
    \end{cases}
\end{equation*}
i.e., $1-3\alpha<2\varepsilon<1$. Under the assumption $0<\alpha\le\frac{2}{3}$, by choosing $2\varepsilon=1-\frac{3}{2}\alpha$, there holds
\begin{equation}
    \begin{aligned}
        (\triangle_1)\le& d_0\tilde{C}_2\int_0^\tau\lamt^{-\frac{3}{2}\alpha}(s)\mathcal{E}_{\rm tem}(s)ds.
    \end{aligned}
\end{equation}
We then obtain
\begin{equation}
   \begin{aligned}
        &\sum_{j=0}^1d_j\lamt^{j(3\alpha-2)}\left(\lamt(\tau)\intol(\ptau^{j+1}\varphi)^2w^3z^4dz+\intol\frac{4}{3}(z\ptau^j\pz\varphi)^2w^4z^2dz+\intol3\delta(\ptau^j\varphi)^2w^3z^4dz\right)\\
        \ge&\tilde{C}_1\mathcal{E}_{\rm tem}(\tau)-\tilde{C}_2\mathcal{E}_{\rm tem}(0)-\tilde{C}_3\int_0^\tau\lamt^{-\frac{3}{2}\alpha}(s)\mathcal{E}_{\rm tem}(s)ds-\frac{\tilde{C}_4}{\sigma}\int_0^\tau\lamt^{-3\alpha}(s)\mathcal{D}_{\rm spt}(s)ds-\sigma\int_0^\tau\mathcal{D}_{\rm tem}(s)ds.
    \end{aligned}
\end{equation}
\end{proof}

With the preparation of Lemma \ref{Lemma: Coercivity of temporal energy} and Lemma \ref{Lemma: Coercivity of temporal dissipation}, we can proceed the following temporal energy estimate.
\begin{lem}\label{Lemma: Est2}Let $-\tilde\varepsilon<\delta$ for $\tilde\varepsilon>0$ small enough. Suppose (\ref{a priori assumptions})-(\ref{a priori assumptions 2}) holds for small constant $\omega_0$. Then
\begin{equation}\label{Eq: Est2}
    \mathcal{E}_{\rm tem}(\tau)+\int_0^\tau\mathcal{D}_{\rm tem}(s)ds\le C\mathcal{E}_{\rm tem}(0)+C\int_0^\tau\lamt^{-\min\{\frac{3}{2}\alpha,\frac{1}{2}\}}(s)\mathcal{E}_{\rm tem}(s)+\lamt^{-3\alpha}(s)\mathcal{D}_{\rm spt}(s)ds.
\end{equation}
\end{lem}
\begin{proof}
We only present the estimate of the highest order temporal derivative, since the low order estimate can be proceeded in a similar way. The low-order version of the temporal energy estimate is as follows
\begin{equation}\label{Eq: low order temporal est}
    \begin{aligned}
         &\intol\frac{1}{2}\lamt(\ptau\varphi)^2w^3z^4+\frac{2}{3}(z\pz\varphi)^2w^4z^2dz+\frac{3}{2}\delta\varphi^2w^3z^4dz\\
         &+\int_0^\tau\left(\lamt'\intol\frac{1}{2}(\ps\varphi)^2w^3z^4dz+\lamt^{3-3\alpha}\intol\frac{4}{3}\mu\left(w^3J^{-1}\right)^{\alpha}\frac{((1+\varphi)z\ps\pz\varphi-\ps\varphi\cdot z\pz\varphi)^2}{1+\varphi+z\pz\varphi}z^2dz\right)ds\\
         \le&C\mathcal{E}_{\rm tem}(0).
    \end{aligned}
    \end{equation}

We now proceed the high order temporal energy estimate. Apply $\ptau$ to (\ref{phiEqn: linearized, G}) to get
\begin{equation}\label{Eq: partial_t of the equation}
    \begin{aligned}
\lamt\ptau^3\varphi+2\lamt'\ptau^2\varphi+\lamt''\ptau\varphi+3\delta\ptau\varphi+\mathcal{L}_{w^4}\ptau\varphi+\ptau\mathcal{N}_1[\varphi]=\ptau(\lamt^{3-3\alpha}V_{w,\alpha}[\varphi]).
    \end{aligned}
\end{equation}
Multiply the resulting equation by $\lamt^{3\alpha-2}\ptau^2\varphi w^3z^4$, integrate the resulting equation over $[0,1]$ and integrate by parts to obtain
\begin{equation}\label{Eq: pt^2 of the eq, 1}
    \begin{aligned}
        &\frac{d}{d\tau}\left(\lamt^{3\alpha-2}\intol\frac{1}{2}\lamt(\ptau^2\varphi)^2w^3z^4+\frac{2}{3}(z\ptau\pzphi)^2w^4z^2+\frac{3}{2}\delta(\ptauphi)^2w^3z^4dz\right)\\
        &+\frac{5-3\alpha}{2}\lamt^{3\alpha-2}\intol\lamt'(\ptau^2\varphi)^2w^3z^4dz\\
        &+\underbrace{\intol\lamt''\lamt^{3\alpha-2}\ptauphi\cdot\ptau^2\varphi w^3z^4dz}_{\star_1}+\underbrace{(2-3\alpha)\lamt^{3\alpha-3}\lamt'\intol\frac{2}{3}(z\ptauzphi)^2w^4z^2+\frac{3}{2}\delta(\ptauphi)^2w^3z^4dz}_{\star_2}\\
        &+\underbrace{\lamt^{3\alpha-2}\intol\ptau\mathcal{N}_1[\varphi]\ptau^2\varphi\cdot w^3z^4dz}_{\star_3}\\
        =&\underbrace{\intol\lamt^{3\alpha-2}\ptau(\lamt^{3-3\alpha}V_{w,\alpha}[\varphi])\ptau^2\varphi w^3z^4dz}_{\star_4}.
    \end{aligned}
\end{equation}
We first estimate the commutator terms, $(\star_1)$ and $(\star_2)$, as follows
\begin{equation}
    (\star_1)\ge-\frac{5-3\alpha}{4}\lamt^{3\alpha-2}\intol\lamt'(\ptau^2\varphi)^2w^3z^4dz-C\intol\lamt^{3\alpha-1}(\ptauphi)^2w^3z^4dz,
\end{equation}
\begin{equation}
\begin{aligned}
    (\star_2)\ge&-C\lamt^{3\alpha-3}\intol\lamt'(\ptauphi)^2w^3z^4dz\\
    &-C\lamt^{3\alpha-3}\intol\lamt^{3-3\alpha}\frac{4}{3}\mu\left(w^3J^{-1}\right)^{\alpha}\frac{[(1+\varphi)z\ptau\pz\varphi-\ptau\varphi\cdot z\pz\varphi]^2}{1+\varphi+z\pz\varphi}z^2dz,
\end{aligned}
\end{equation}
where we make use of the fact $0<\alpha\le\frac{2}{3}$ to deal with the temporal weight.

The next step is to estimate the nonlinear terms $\mathcal{N}_1[\varphi]$ and $V_{w,\alpha}[\varphi]$.

{\RaggedRight\underline{\textbf{Estimate of $\mathcal{N}_1[\varphi]$}:}} Recalling the definition of $\mathcal{N}_1[\varphi]$ from (\ref{Eq: nonlinearity, N1}), by direct computation, we have
\begin{equation}\label{kth-derivative of phiEqnDt: Linearization, Nonlinearity 1}
    \begin{aligned}
       (\star_3)
        =&\lamt^{3\alpha-2}\intol\delta\ptau\left(1-2\varphi-\frac{1}{(\Flm)^2}\right)\cdot\ptau^{2}\varphi\cdot w^3z^4\\
        &\ \ \ \ \ \ \ -\frac{4}{3}\ptau\left[\left((\Flm)^2-1\right)\pz\left(\frac{w^4}{z^2}\pz\left(z^3\varphi\right)\right)\right]\ptau^{2}\varphi\cdot z^3\\
        &\ \ \ \ \ \ \ -\frac{4}{3}\ptau\left[(\Flm)^2\pz\left(\frac{w^4}{z^2}\pz\left(z^3(\varphi^2+\frac{1}{3}\varphi^3)\right)\right)\right]\ptau^{2}\varphi\cdot z^3\\
        &\ \ \ \ \ \ \ +\ptau\left[(\Flm)^2\pz\left(\frac{28}{9}w^4(J-1)^2\intol(1-\theta)[1+\theta(J-1)]^{-\frac{10}{3}}d\theta\right)\right]\ptau^{2}\varphi\cdot z^3\\
        &\ \ \ \ \ \ \ +\pz w^4\ptau\left((\Flm)^2-\frac{1}{(\Flm)^2}-4\varphi\right)\ptau^{2}\varphi\cdot z^3dz\\
        =&\lamt^{3\alpha-2}\sum_{i=1}^5I_i.
    \end{aligned}
\end{equation}
The terms $I_1,\ I_2,\ I_3$ can be estimated as follows
\begin{equation}\label{Eq: I_1, temporal est}
    \begin{aligned}
     I_1=&\intol\delta \ptau\left(\frac{\varphi^2(-2\varphi-3)}{(\Flm)^2}\right)\cdot\ptau^{2}\varphi\cdot w^3z^4dz\\
     \le&C\omega_1\intol\lamt^{\frac{1}{2}}(\tau)(\ptau^{2}\varphi)^2w^3z^4dz+C\omega_1\sum_{j=0}^1\intol\lamt^{-\frac{1}{2}}(\tau)(\ptau^j\varphi)^2w^3z^4dz,
    \end{aligned}
\end{equation}
\begin{equation}
    \begin{aligned}
        I_2
        =&\intol\frac{4}{3}\ptau\left[(2\varphi+\varphi^2)(3\varphi+z\pz\varphi)\right](3\ptau^{2}\varphi+z\ptau^{2}\pz\varphi)w^4z^2dz\\
        &+\intol \frac{4}{3}\ptau\left[\left(2z\pz\varphi+2\varphi \cdot z\pz\varphi\right)\cdot(3\varphi+z\pz\varphi)\right]\ptau^{2}\varphi\cdot w^4z^2dz\\
        \le&C\omega_1\intol\lamt^{\frac{1}{2}}(\tau)\left((\ptau^{2}\varphi)^2+(z\ptau^{2}\pz\varphi)^2\right)w^4z^2 dz\\
        &+C\omega_1\sum_{j=0}^1\intol\lamt^{-\frac{1}{2}}(\tau)\left((\ptau^j\varphi)^2+(z\ptau^j\pz\varphi)^2\right)w^4z^2dz,
    \end{aligned}
\end{equation}
\begin{equation}
    \begin{aligned}
        I_3
        =&\intol\frac{4}{3}\ptau\left[(\Flm)^2\left(3(\varphi^2+\frac{1}{3}\varphi^3)+(2\varphi+\varphi^2)z\pz\varphi\right)\right](3\ptau^{2}\varphi+z\ptau^{2}\pz\varphi)w^4z^2dz\\
        &+\intol\frac{4}{3}\ptau\left[(2z\pz\varphi+2\varphi\cdot z\pz\varphi)\left(3(\varphi^2+\frac{1}{3}\varphi^3)+(2\varphi+\varphi^2)z\pz\varphi\right)\right]\ptau^{2}\varphi\cdot w^4z^2dz\\
        \le&C\omega_1\intol\lamt^{\frac{1}{2}}(\tau)\left((\ptau^{2}\varphi)^2+(z\ptau^{2}\pz\varphi)^2\right)w^4z^2 dz\\
        &+C\omega_1\sum_{j=0}^1\intol \lamt^{-\frac{1}{2}}(\tau)\left((\ptau^{j}\varphi)^2+(z\ptau^j\pz\varphi)^2\right)w^4z^2dz.
    \end{aligned}
\end{equation}
For $I_4$, notice that $J-1=p(\varphi,z\pz\varphi)=3\varphi+z\pzphi+h.o.t.$ is just a polynomial of $\varphi$ and $z\pz\varphi$, we thus have
\begin{equation}
    \begin{aligned} I_4
        =&-\intol\ptau\left[(\Flm)^2\frac{28}{9}(J-1)^2\intol(1-\theta)[1+\theta(J-1)]^{-\frac{10}{3}}d\theta\right](3\ptau^{2}\varphi+z\ptau^{2}\pz\varphi)w^4z^2dz\\
        &-\intol\ptau\left[2(\Flm) \cdot z\pz\varphi\left(\frac{28}{9}(J-1)^2\intol(1-\theta)[1+\theta(J-1)]^{-\frac{10}{3}}d\theta\right)\right]\ptau^{2}\varphi\cdot w^4z^2dz\\
        \le&C\omega_1\intol\lamt^{\frac{1}{2}}(\tau)((\ptau^{2}\varphi)^2+(z\ptau^{2}\pz\varphi)^2)w^4z^2dz\\
        &+C\omega_1\sum_{j=0}^1\intol\lamt^{-\frac{1}{2}}(\tau)((\ptau^j\varphi)^2+(z\ptau^j\pz\varphi)^2) w^4z^2dz.\\
    \end{aligned}
\end{equation}
$I_5$ can be estimated as follows
\begin{equation}\label{Eq: I_5, temporal est}
    \begin{aligned}
        I_5=&\intotl\pz w^4\ptau\left((\Flm)^2-\frac{1}{(\Flm)^2}-4\varphi\right)\ptau^{2}\varphi\cdot z^3dz\\
        =&-\intotl w^4\ptau\pz\left((\Flm)^2-\frac{1}{(\Flm)^2}-4\varphi\right)\ptau^{2}\varphi\cdot z^3dz\\
        &-\intotl w^4\ptau\left((\Flm)^2-\frac{1}{(\Flm)^2}-4\varphi\right)(3\ptau^{2}\varphi+z\ptau^{2}\pz\varphi)z^2dz\\
        \le&C\omega_1\intotl\lamt^{\frac{1}{2}}(\tau)((\ptau^{2}\varphi)^2+(z\ptau^{2}\pz\varphi)^2)w^4z^2dzds\\
        &+C\omega_1\sum_{j=0}^1\intotl\lamt^{-\frac{1}{2}}(\tau)((\ptau^j\varphi)^2+(z\ptau^j\pz\varphi)^2)w^4z^2dzds.
    \end{aligned}
\end{equation}
Collecting (\ref{Eq: I_1, temporal est})-(\ref{Eq: I_5, temporal est}), we have
\begin{equation}\label{nonlinearity 1, final}
    \begin{aligned}
        (\star_3)\le&C\omega_1\lamt^{3\alpha-2}\intol\lamt^{\frac{1}{2}}(\tau)(\ptau^{2}\varphi)^2w^3z^4dz+C\omega_1\lamt^{3\alpha-2}\sum_{j=0}^1\intol\lamt^{-\frac{1}{2}}(\tau)(\ptau^j\varphi)^2w^3z^4dz\\
        &+C\omega_1\lamt^{3\alpha-2}\intol\lamt^{\frac{1}{2}}(\tau)((\ptau^{2}\varphi)^2+(z\ptau^{2}\pz\varphi)^2)w^4z^2dz\\
        &+C\omega_1\lamt^{3\alpha-2}\sum_{j=0}^1\intol\lamt^{-\frac{1}{2}}(\tau)((\ptau^j\varphi)^2+(z\ptau^j\pz\varphi)^2)w^4z^2dz\\
        \le&C\omega_1\lamt^{3\alpha-2}\intol\lamt^{\frac{1}{2}}(\tau)(\ptau^{2}\varphi)^2w^3z^4dz+C\omega_1\lamt^{3\alpha-2}\sum_{j=0}^1\intol\lamt^{-\frac{1}{2}}(\tau)(\ptau^j\varphi)^2w^3z^4dz\\
        &+C\omega_1\lamt^{3\alpha-2}\intol\lamt^{\frac{1}{2}}(\tau)(z\ptau^{2}\pz\varphi)^2w^4z^2dz+C\omega_1\lamt^{3\alpha-2}\sum_{j=0}^1\intol\lamt^{-\frac{1}{2}}(\tau)(z\ptau^j\pz\varphi)^2w^4z^2dz\\
        \le&C\omega_1\lamt^{-\frac{1}{2}}(\tau)\mathcal{E}_{\rm tem}(\tau)ds+C\omega_1\lamt^{-\frac{1}{2}}(\tau)\mathcal{D}_{\rm tem}(\tau)ds,
    \end{aligned}
\end{equation}
where we apply Hardy's inequality (\ref{hardy 2}) in the second "$\le$".

{\RaggedRight\underline{\textbf{Estimate of $V_{w,\alpha}[\varphi]$}:}} For the right-hand side of (\ref{Eq: pt^2 of the eq, 1}), we have
\begin{equation}
\begin{aligned}\label{Eq: ptG term, 1}
    (\star_4)=&\intol\lamt\ptau V_{w,\alpha}[\varphi]\ptau^2\varphi w^3z^4dz+\intol(3-3\alpha)\lamt' V_{w,\alpha}[\varphi]\ptau^2\varphi w^3z^4dz.
\end{aligned}
\end{equation}
For the first term on the right-hand side of (\ref{Eq: ptG term, 1}), by direct computation, we have
\begin{equation}
    \begin{aligned}
        &\intol\lamt\ptau V_{w,\alpha}[\varphi]\ptau^2\varphi w^3z^4dz\\
        =&\intol\lamt\frac{\ptau^2\varphi}{\Flm}\pz\left(\frac{4}{3}\mu z^3(\Flm)^2(w^3J^{-1})^\alpha\frac{(\Flm)z\ptau^2\pzphi-\ptau^2\varphi\cdot z\pzphi}{\Flm+z\pzphi}\right) dz\\
        &+\intol\lamt\frac{\ptau^2\varphi}{\Flm}\pz\left(\frac{4}{3}\mu w^{3\alpha}z^3\ptau\left(\frac{(\Flm)^2J^{-\alpha}}{\Flm+z\pzphi}\right)[(\Flm)z\ptauzphi-\ptauphi\cdot z\pzphi]\right) dz\\
        &+\intol\lamt\frac{-\ptauphi\ptau^2\varphi}{(\Flm)^2}\pz\left(\frac{4}{3}\mu z^3(\Flm)^2(w^3J^{-1})^\alpha\frac{(\Flm)z\ptauzphi-\ptauphi\cdot z\pzphi}{\Flm+z\pzphi}\right)dz\\
        =&I_1+I_2+I_3.
    \end{aligned}
\end{equation}
For $I_1$, we integrate by parts to get
\begin{equation}
    \begin{aligned}
        I_1=&\intol\lamt\frac{\ptau^2\varphi}{\Flm}\pz\left(\frac{4}{3}\mu z^3(\Flm)^2(w^3J^{-1})^\alpha\frac{(\Flm)z\ptau^2\pzphi-\ptau^2\varphi\cdot z\pzphi}{\Flm+z\pzphi}\right) dz\\
        =&-\intol\lamt\frac{4}{3}\mu z^2(w^3J^{-1})^\alpha\frac{[(\Flm)z\ptau^2\pzphi-\ptau^2\varphi\cdot z\pzphi]^2}{\Flm+z\pzphi}dz.
    \end{aligned}
\end{equation}
For $I_2$, we integrate by parts to get
\begin{equation}
    \begin{aligned}
        I_2=&\intol\lamt\frac{\ptau^2\varphi}{\Flm}\pz\left(\frac{4}{3}\mu w^{3\alpha}z^3\ptau\left(\frac{(\Flm)^2J^{-\alpha}}{\Flm+z\pzphi}\right)[(\Flm)z\ptauzphi-\ptauphi\cdot z\pzphi]\right) dz\\
        =&-\intol\lamt\frac{(\Flm)z\ptau^2\pz\varphi-\ptau^2\varphi\cdot z\pzphi}{(\Flm)^2}\frac{4}{3}\mu w^{3\alpha}z^2\ptau\left(\frac{(\Flm)^2J^{-\alpha}}{\Flm+z\pzphi}\right)[(\Flm)z\ptauzphi-\ptauphi\cdot z\pzphi]dz\\
        \le&\omega_1\intol\lamt\frac{4}{3}\mu z^2(w^3J^{-1})^\alpha\frac{[(\Flm)z\ptau^2\pzphi-\ptau^2\varphi\cdot z\pzphi]^2}{\Flm+z\pzphi}dz\\
        &+C\omega_1\intol\lamt\frac{4}{3}\mu z^2(w^3J^{-1})^\alpha\frac{[(\Flm)z\ptau\pzphi-\ptau\varphi\cdot z\pzphi]^2}{\Flm+z\pzphi}dz.
    \end{aligned}
\end{equation}
For $I_3$, we integrate by parts to get
\begin{equation}\label{Eq: example of time weight}
    \begin{aligned}
        I_3=&\intol\lamt\frac{-\ptauphi\ptau^2\varphi}{(\Flm)^2}\pz\left(\frac{4}{3}\mu z^3(\Flm)^2(w^3J^{-1})^\alpha\frac{(\Flm)z\ptauzphi-\ptauphi\cdot z\pzphi}{\Flm+z\pzphi}\right)dz\\
        =&\intol\lamt\pz\left(\frac{\ptau\varphi}{\Flm}\right)\frac{\ptau^2\varphi}{\Flm}\frac{4}{3}\mu z^3(\Flm)^2(w^3J^{-1})^\alpha\frac{(\Flm)z\ptauzphi-\ptauphi\cdot z\pzphi}{\Flm+z\pzphi}dz\\
        &+\intol\lamt\frac{\ptauphi}{\Flm}\pz\left(\frac{\ptau^2\varphi}{\Flm}\right)\frac{4}{3}\mu z^3(\Flm)^2(w^3J^{-1})^\alpha\frac{(\Flm)z\ptauzphi-\ptauphi\cdot z\pzphi}{\Flm+z\pzphi}dz\\
        \le&C\omega_1\intol\lamt\frac{4}{3}\mu z^2(w^3J^{-1})^\alpha\frac{[(\Flm)z\ptau^2\pzphi-\ptau^2\varphi\cdot z\pzphi]^2}{\Flm+z\pzphi}dz+\omega_1\intol\lamt^{3\alpha-1}\left(\frac{\ptau^2\varphi}{\Flm}\right)^2w^{3\alpha}z^2dz\\
        &+C\omega_1\intol\lamt^{3-3\alpha}\frac{4}{3}\mu z^2(w^3J^{-1})^\alpha\frac{[(\Flm)z\ptau\pzphi-\ptau\varphi\cdot z\pzphi]^2}{\Flm+z\pzphi}dz.
    \end{aligned}
\end{equation}
By applying Hardy's inequality twice, we have
\begin{equation}
\begin{aligned}
    &\intol\lamt^{3\alpha-1}\left(\frac{\ptau^2\varphi}{\Flm}\right)^2w^{3\alpha}z^2dz\\
    \lesssim&\intol\lamt^{3\alpha-1}\left[\left(\frac{\ptau^2\varphi}{\Flm}\right)^2+\left(\frac{(\Flm)\ptau^2\pz\varphi-\ptau^2\varphi\cdot\pzphi}{(\Flm)^2}\right)^2\right]w^{3\alpha+2}z^4dz\\
    \lesssim&\intol\lamt^{3\alpha-1}(\ptau^2\varphi)^2w^3z^4+\lamt^{3\alpha-1}\frac{[(\Flm)\ptau^2\pz\varphi-\ptau^2\varphi\cdot\pzphi]^2}{\Flm+z\pzphi}w^{3\alpha+2}z^4dz.
\end{aligned}
\end{equation}
Since $0<\alpha\le\frac{2}{3}$ implies $3\alpha-1\le1$, we have
\begin{equation}
\begin{aligned}
     I_3\le& C\omega_1\intol\lamt^{3\alpha-1}(\ptau^2\varphi)^2w^3z^4dz+C\omega_1\intol\lamt\frac{4}{3}\mu z^2(w^3J^{-1})^\alpha\frac{[(\Flm)z\ptau^2\pzphi-\ptau^2\varphi\cdot z\pzphi]^2}{\Flm+z\pzphi}dz\\
     &+C\omega_1\intol\lamt^{3-3\alpha}\frac{4}{3}\mu z^2(w^3J^{-1})^\alpha\frac{[(\Flm)z\ptau\pzphi-\ptau\varphi\cdot z\pzphi]^2}{\Flm+z\pzphi}dz.
\end{aligned}
\end{equation}
As a result, the first term on the right-hand side of (\ref{Eq: ptG term, 1}) can be controlled as
\begin{equation}
    \begin{aligned}
        &\intol\lamt\ptau V_{w,\alpha}[\varphi]\ptau^2\varphi w^3z^4dz\\
        \le&-(1-C\omega_1)\intol\lamt\frac{4}{3}\mu z^2(w^3J^{-1})^\alpha\frac{[(\Flm)z\ptau^2\pzphi-\ptau^2\varphi\cdot z\pzphi]^2}{\Flm+z\pzphi}dz\\
        &+C\omega_1\intol\lamt^{3\alpha-1}(\ptau^2\varphi)^2w^3z^4dz+C\omega_1\intol\lamt^{3-3\alpha}\frac{4}{3}\mu z^2(w^3J^{-1})^\alpha\frac{[(\Flm)z\ptau\pzphi-\ptau\varphi\cdot z\pzphi]^2}{\Flm+z\pzphi}dz.
    \end{aligned}
\end{equation}
While for the second term on the right-hand side of (\ref{Eq: ptG term, 1}), we have
\begin{equation}
    \begin{aligned}
        &\intol(3-3\alpha)\lamt' V_{w,\alpha}[\varphi]\ptau^2\varphi w^3z^4dz\\
        =&\intol(3-3\alpha)\lamt'\frac{\ptau^2\varphi}{\Flm}\pz\left(\frac{4}{3}\mu z^3(\Flm)^2(w^3J^{-1})^\alpha\frac{(\Flm)z\ptauzphi-\ptauphi\cdot z\pzphi}{\Flm+z\pzphi}\right)dz\\
                =&-\intol(3-3\alpha)\lamt'\pz\left(\frac{\ptau^2\varphi}{\Flm}\right)\frac{4}{3}\mu z^3(\Flm)^2(w^3J^{-1})^\alpha\frac{(\Flm)z\ptauzphi-\ptauphi\cdot z\pzphi}{\Flm+z\pzphi}dz\\
                \le&\frac{1}{6}\intol\lamt\mu z^2(w^3J^{-1})^\alpha\frac{[(\Flm)z\ptau^2\pzphi-\ptau^2\varphi\cdot z\pzphi]^2}{\Flm+z\pzphi}dz\\
     &+C\intol\lamt\frac{4}{3}\mu z^2(w^3J^{-1})^\alpha\frac{[(\Flm)z\ptau\pzphi-\ptau\varphi\cdot z\pzphi]^2}{\Flm+z\pzphi}dz.
    \end{aligned}
\end{equation}
So (\ref{Eq: ptG term, 1}) can be estimated as
\begin{equation}
    \begin{aligned}
        (\star_4)\le&-\intol\lamt\mu z^2(w^3J^{-1})^\alpha\frac{[(\Flm)z\ptau^2\pzphi-\ptau^2\varphi\cdot z\pzphi]^2}{\Flm+z\pzphi}dz\\
        &+C\intol\lamt^{3-3\alpha}\frac{4}{3}\mu z^2(w^3J^{-1})^\alpha\frac{[(\Flm)z\ptau\pzphi-\ptau\varphi\cdot z\pzphi]^2}{\Flm+z\pzphi}dz\\
        &+C\omega_1\intol\lamt^{3\alpha-1}(\ptau^2\varphi)^2w^3z^4dz.
    \end{aligned}
\end{equation}

Combining all the ingredients, we obtain
\begin{equation}\label{Eq: pt^2 of the eq, 3}
    \begin{aligned}
        &\frac{d}{d\tau}\left(\lamt^{3\alpha-2}\intol\frac{1}{2}\lamt(\ptau^2\varphi)^2w^3z^4+\frac{2}{3}(z\ptau\pzphi)^2w^4z^2+\frac{3}{2}\delta(\ptauphi)^2w^3z^4dz\right)\\
        &+\frac{5-3\alpha}{8}\lamt^{3\alpha-2}\intol\lamt'(\ptau^2\varphi)^2w^3z^4dz+\intol\lamt\mu z^2(w^3J^{-1})^\alpha\frac{[(\Flm)z\ptau^2\pzphi-\ptau^2\varphi\cdot z\pzphi]^2}{\Flm+z\pzphi}dz\\
        \le&C\intol\lamt^{3-3\alpha}\frac{4}{3}\mu z^2(w^3J^{-1})^\alpha\frac{[(\Flm)z\ptau\pzphi-\ptau\varphi\cdot z\pzphi]^2}{\Flm+z\pzphi}dz+C\lamt^{3\alpha-1}\intol(\ptauphi)^2w^3z^4dz\\
        &+C\omega_1\lamt^{-\frac{1}{2}}\mathcal{E}_{\rm tem}(\tau)ds+C\omega_1\lamt^{-\frac{1}{2}}\mathcal{D}_{\rm tem}(\tau)ds.
    \end{aligned}
\end{equation}
Integrate (\ref{Eq: pt^2 of the eq, 3}) in temporal variable over $[0,\tau]$, combined with (\ref{Eq: low order temporal est}), we obtain
\begin{equation}
    \tilde{\mathcal{E}}_{\rm tem}(\tau)+\int_0^\tau\mathcal{D}_{\rm tem}(s)ds\le C\mathcal{E}_{\rm tem}(0)+C\int_0^\tau\lamt^{-\frac{1}{2}}\mathcal{E}_{\rm tem}(s)ds.
\end{equation}
Combined with Lemma \ref{Lemma: Coercivity of temporal energy}, choosing $\sigma>0$ small enough, we obtain
\begin{equation}
    \mathcal{E}_{\rm tem}(\tau)+\int_0^\tau\mathcal{D}_{\rm tem}(s)ds\le C\mathcal{E}_{\rm tem}(0)+C\int_0^\tau\lamt^{-\min\{\frac{3}{2}\alpha,\frac{1}{2}\}}(s)\mathcal{E}_{\rm tem}(s)+\lamt^{-3\alpha}(s)\mathcal{D}_{\rm spt}(s)ds.
\end{equation}
\end{proof}
\section{Interior estimates}
In this section, we establish the interior estimates. We recall that the smooth cutoff function $\hat{\chi}(z)$ is defined as in (\ref{Eq: Cutoff}). For convenience, we introduce
\begin{equation}
    \begin{aligned}
        \tilde{\mathcal{E}}_{\rm int}(\tau)=&\frac{1}{2}\left(\lamt^{3\alpha-2}\intol\hat{\chi}\lamt(\ptau\varphi)^2w^3z^2+3\delta\hat{\chi}\varphi^2w^3z^2+\frac{4}{3}\hat{\chi}(z\pz\varphi)^2w^4dz\right)\\
        &+\frac{1}{2}\left(\lamt^{6\alpha-4}\intol\hat{\chi}\lamt(\ptau^{2}\varphi)^2w^3z^2+3\delta\hat{\chi}(\ptau\varphi)^2w^3z^2+\frac{4}{3}\hat{\chi}(z\ptau\pz\varphi)^2w^4dz\right)
    \end{aligned}
\end{equation}
\begin{lem}\label{lemma: interior est}Let $-\tilde\varepsilon<\delta$ for $\tilde\varepsilon>0$ small enough. Suppose (\ref{a priori assumptions})-(\ref{a priori assumptions 2}) holds for small constant $\omega_0$. Then
\begin{equation}\label{interior energy estimate}
    \begin{aligned}
        \mathcal{E}_{\rm int}(\tau)+\int_0^\tau\mathcal{D}_{\rm int}(s)ds\le& C\mathcal{E}_{\rm int}(0)+C\mathcal{E}_{\rm tem}(0)\\&+C\int_0^\tau\left(\lamt^{-\frac{1}{2}}\mathcal{E}_{\rm int}(s)+\lamt^{-\min\{\frac{3}{2}\alpha,\frac{1}{2}\}}\mathcal{E}_{\rm tem}(s)+\lamt^{-3\alpha}\mathcal{D}_{\rm spt}(s)\right)ds.
    \end{aligned}
\end{equation}
\end{lem}
\begin{proof}

We only present the highest order interior estimate. The following is the low order estimate
\begin{equation}\label{Eq: low order interior estimate}
\begin{aligned}
    &\lamt^{3\alpha-2}\intol\frac{1}{2}\lamt\hat{\chi}(\ptauphi)^2w^3z^2+\frac{2}{3}\hat{\chi}(z\pzphi)^2w^4+\frac{3}{2}\delta\hat{\chi}\varphi^2w^3z^2dz\\
    &+\frac{1}{3}\int_0^\tau\intol\lamt\mu\hat{\chi}(\psphi)^2w^{3\alpha}+\lamt\mu\hat{\chi}(z\ps\pzphi)^2w^{3\alpha}dzds\\
    \le&C\mathcal{E}_{\rm int}(0)+C\int_0^\tau\left(\lamt^{-\frac{1}{2}}\mathcal{E}_{\rm int}(s)+\lamt^{3\alpha-2-\frac{1}{2}}\mathcal{E}_{\rm tem}(s)+\mathcal{D}_{\rm tem}(s)\right)ds,
\end{aligned}
\end{equation}
which can be proceeded in a similar way as the high order estimate. 

We apply $\ptau$ to (\ref{phiEqn: linearized, G}) to get
\begin{equation}
    \begin{aligned}
\lamt\ptau^3\varphi+2\lamt'\ptau^2\varphi+\lamt''\ptau\varphi+3\delta\ptau\varphi+\mathcal{L}_{w^4}\ptau\varphi+\ptau\mathcal{N}_1[\varphi]=\ptau(\lamt^{3-3\alpha}V_{w,\alpha}[\varphi]).
    \end{aligned}
\end{equation}
Multiply it by $\hat{\chi}\lamt^{6\alpha-4}\ptau^2\varphi w^3z^2$, integrate it over $[0,1]$ and integrate by parts to obtain
\begin{equation}\label{Eq: HighOrderInteriorEst, diff}
    \begin{aligned}
        &\frac{d}{d\tau}\left(\lamt^{6\alpha-4}\intol\frac{1}{2}\lamt\hat{\chi}(\ptau^2\varphi)w^3z^2+\frac{2}{3}\hat{\chi}(z\ptauzphi)^2w^4+\frac{3}{2}\delta\hat{\chi}(\ptauphi)^2w^3z^2dz\right)\\
        &+\intol(\frac{7}{2}-3\alpha)\lamt^{6\alpha-4}\lamt'\hat{\chi}(\ptau^2\varphi)^2w^3z^2+\frac{2}{3}(4-6\alpha)\lamt^{6\alpha-5}\lamt'(z\ptauzphi)^2w^4\\
        &\ \ \ \ \ \ \ \ \ \ +\frac{3}{2}\delta(4-6\alpha)\lamt^{6\alpha-5}\lamt'\hat{\chi}(\ptauphi)^2w^3z^2dz\\
        &+\underbrace{\intol\lamt^{6\alpha-4}\lamt''\hat{\chi}\ptauphi\ptau^2\varphi\cdot w^3z^2+\frac{4}{3}\lamt^{6\alpha-4}\hat{\chi}'\ptau^2\varphi\cdot z\ptauzphi w^4z-\frac{8}{3}\lamt^{6\alpha-4}\hat{\chi}\ptau^2\varphi\cdot z\ptauzphi w^4dz}_{\square_1}\\
        &+\underbrace{\intol\lamt^{6\alpha-4}\ptau\mathcal{N}_1[\varphi]\hat{\chi}\ptau^2\varphi w^3z^2dz}_{\square_2}\underbrace{-\intol\lamt^{6\alpha-4}\hat{\chi}\ptau(\lamt^{3-3\alpha}V_{w,\alpha}[\varphi])\ptau^2\varphi w^3z^2dz}_{\square_3}=0.
    \end{aligned}
\end{equation}

For the commutator term $(\square_1)$, we estimate it as
\begin{equation}\label{Eq: InteriorEst, square1}
    \begin{aligned}
        (\square_1)\ge&-\sigma\intol\lamt^{12\alpha-7}\mu\hat{\chi}(\ptau^2\varphi)^2w^3z^2dz-C_\sigma\intol\lamt\mu\hat{\chi}(\ptauphi)^2w^{3\alpha}+\lamt\mu\hat{\chi}(z\ptauzphi)^2w^{3\alpha}dz\\
        &-C\lamt^{-\frac{1}{2}}\mathcal{E}_{\rm tem}(\tau).\\
    \end{aligned}
\end{equation}

The next goal is to estimate the nonlinear terms coming from $\mathcal{N}_1[\varphi]$ and $V_{w,\alpha}[\varphi]$.

{\RaggedRight\underline{\textbf{Estimate of $\mathcal{N}_1[\varphi]$}:}} Similar to the proof of (\ref{nonlinearity 1, final}), we have
\begin{equation}\label{Eq: interior, ptau N1}
    \begin{aligned}
       (\square_2)
        =&\lamt^{6\alpha-4}\intol\delta\ptau\left(1-2\varphi-\frac{1}{(\Flm)^2}\right)\cdot\hat{\chi}\ptau^2\varphi\cdot w^3z^2\\
        &\ \ \ \ \ \ \ -\frac{4}{3}\ptau\left[\left((\Flm)^2-1\right)\pz\left(\frac{w^4}{z^2}\pz\left(z^3\varphi\right)\right)\right]\hat{\chi}\ptau^2\varphi\cdot z\\
        &\ \ \ \ \ \ \ -\frac{4}{3}\ptau\left[(\Flm)^2\pz\left(\frac{w^4}{z^2}\pz\left(z^3(\varphi^2+\frac{1}{3}\varphi^3)\right)\right)\right]\hat{\chi}\ptau^2\varphi\cdot z\\
        &\ \ \ \ \ \ \ +\ptau\left[(\Flm)^2\pz\left(\frac{28}{9}w^4(J-1)^2\intol(1-\theta)[1+\theta(J-1)]^{-\frac{10}{3}}d\theta\right)\right]\hat{\chi}\ptau^2\varphi\cdot z\\
        &\ \ \ \ \ \ \ +\pz w^4\ptau\left((\Flm)^2-\frac{1}{(\Flm)^2}-4\varphi\right)\hat{\chi}\ptau^2\varphi\cdot zdz\\
        \ge&-C\omega_1\lamt^{3\alpha-2-\frac{1}{2}}\mathcal{E}_{\rm tem}(\tau)-C\omega_1\lamt^{-\frac{1}{2}}\mathcal{E}_{\rm int}(\tau)-C\omega_1\lamt^{9\alpha-8}\intol\hat{\chi}[\lamt(\ptau\varphi)^2+\lamt(z\ptau\pzphi)^2]w^{3\alpha}dz\\
        &-C\omega_1\lamt^{3\alpha-2}\intol\hat{\chi}[\lamt(\ptau^2\varphi)^2+\lamt(z\ptau^2\pzphi)^2]w^{3\alpha}dz.\\
    \end{aligned}
\end{equation}

{\RaggedRight\underline{\textbf{Estimate of $V_{w,\alpha}[\varphi]$}:}} By direct computation, we can write
\begin{equation}
    \begin{aligned}
        (\square_3)=&-\intol(3-3\alpha)\lamt^{3\alpha-2}\lamt'\hat{\chi}V_{w,\alpha}[\varphi]\ptau^2\varphi w^3z^2dz-\intol\hat{\chi}\lamt^{3\alpha-1}\ptau V_{w,\alpha}[\varphi]\ptau^2\varphi w^3z^2dz\\
        =&I_1+I_2.
    \end{aligned}
\end{equation}
We begin with the estimate of $I_2$. By direct computation, we have
\begin{equation}
    \begin{aligned}
        I_2=&-\intol\hat{\chi}\lamt^{3\alpha-1}\ptau V_{w,\alpha}[\varphi]\ptau^2\varphi w^3z^2dz\\
        =&-\intol\lamt^{3\alpha-1}\hat{\chi}\frac{1}{z^2}\frac{\ptau^2\varphi}{\Flm}\pz\left(\frac{4}{3}\mu z^3(\Flm)^2(w^3J^{-1})^\alpha\frac{(\Flm)z\ptau^2\pzphi-\ptau^2\varphi\cdot z\pzphi}{\Flm+z\pzphi}\right) dz\\
        &-\intol\lamt^{3\alpha-1}\hat{\chi}\frac{1}{z^2}\frac{\ptau^2\varphi}{\Flm}\pz\left(\frac{4}{3}\mu w^{3\alpha}z^3\ptau\left(\frac{(\Flm)^2J^{-\alpha}}{\Flm+z\pzphi}\right)[(\Flm)z\ptauzphi-\ptauphi\cdot z\pzphi]\right) dz\\
        &-\intol\lamt^{3\alpha-1}\hat{\chi}\frac{1}{z^2}\frac{-\ptauphi\ptau^2\varphi}{(\Flm)^2}\pz\left(\frac{4}{3}\mu z^3(\Flm)^2(w^3J^{-1})^\alpha\frac{(\Flm)z\ptauzphi-\ptauphi\cdot z\pzphi}{\Flm+z\pzphi}\right)dz\\
        =&I_{2,1}+I_{2,2}+I_{2,3}.
    \end{aligned}
\end{equation}
For the term $I_{2,1}$, we integrate by parts to get
\begin{equation}
    \begin{aligned}
        I_{2,1}=&\underbrace{\intol\lamt^{3\alpha-1}\frac{4}{3}\mu (\Flm)(w^3J^{-1})^\alpha\frac{(\Flm)z\ptau^2\pzphi-\ptau^2\varphi\cdot z\pzphi}{\Flm+z\pzphi}\hat{\chi}'z\ptau^2\varphi dz}_{K_1}\\
        &\underbrace{-\intol\lamt^{3\alpha-1}\frac{8}{3}\mu (\Flm)(w^3J^{-1})^\alpha\frac{(\Flm)z\ptau^2\pzphi-\ptau^2\varphi\cdot z\pzphi}{\Flm+z\pzphi}\hat{\chi}\ptau^2\varphi dz}_{K_2}\\
        &+\intol\lamt^{3\alpha-1}\hat{\chi}\frac{4}{3}\mu (w^3J^{-1})^\alpha\frac{[(\Flm)z\ptau^2\pzphi-\ptau^2\varphi\cdot z\pzphi]^2}{\Flm+z\pzphi} dz.\\
    \end{aligned}
\end{equation}
For $K_1$, we have
\begin{equation}
    \begin{aligned}
        K_1\ge&-C\intol\lamt^{6\alpha-3}(\ptau^2\varphi)^2w^3z^4+\lamt\mu z^2(w^3J^{-1})^\alpha\frac{[(\Flm)z\ptau^2\pzphi-\ptau^2\varphi\cdot z\pzphi]^2}{\Flm+z\pzphi}dz\\
        \ge&-C\mathcal{D}_{\rm tem}(\tau).
    \end{aligned}
\end{equation}
For $K_2$, notice the fact that $\frac{(\Flm)\ptau^2\pzphi-\ptau^2\varphi\cdot \pzphi}{(1+\varphi)^2}=\pz\left(\frac{\ptau^2\varphi}{\Flm}\right)$, we split it as follows
\begin{equation}
\begin{aligned}
    K_2=&-\intol\lamt^{3\alpha-1}\frac{8}{3}\mu \hat{\chi}\frac{(\Flm)\ptau^2\pzphi-\ptau^2\varphi\cdot \pzphi}{(1+\varphi)^2}\frac{\ptau^2\varphi}{1+\varphi}w^{3\alpha}zdz\\
    &-\intol\lamt^{3\alpha-1}\frac{8}{3}\mu \left[\frac{(\Flm)^4J^{-\alpha}}{1+\varphi+z\pzphi}-1\right]\hat{\chi}\frac{(\Flm)z\ptau^2\pzphi-\ptau^2\varphi\cdot z\pzphi}{(1+\varphi)^2}\frac{\ptau^2\varphi}{1+\varphi}w^{3\alpha}dz\\
    =&K_{2,1}+K_{2,2},
\end{aligned}
\end{equation}
where we can estimate $K_{2,1}$ as
\begin{equation}
    \begin{aligned}
        K_{2,1}=&-\intol\lamt^{3\alpha-1}\frac{4}{3}\mu\pz\left(\left(\frac{\ptau^2\varphi}{1+\varphi}\right)^2\right)\hat{\chi}w^{3\alpha}zdz\\
        =&\intol\lamt^{3\alpha-1}\frac{4}{3}\mu\left(\frac{\ptau^2\varphi}{1+\varphi}\right)^2(\hat{\chi}w^{3\alpha}+\hat{\chi}'w^{3\alpha}z+3\alpha\hat{\chi} w^{3\alpha-1}w'z)dz\\
        \ge&\intol\lamt^{3\alpha-1}\frac{7}{6}\mu\hat{\chi}\left(\frac{\ptau^2\varphi}{1+\varphi}\right)^2w^{3\alpha}dz-C\mathcal{D}_{\rm tem}(\tau).
    \end{aligned}
\end{equation}
For $K_{2,2}$, making use of the a priori assumption (\ref{a priori assumptions})-(\ref{a priori assumptions 2}), we have
\begin{equation}
    \begin{aligned}
        K_{2,2}\ge& -C\omega_1\intol\lamt^{3\alpha-1}\mu\hat{\chi}(\ptau^2\varphi)^2w^{3\alpha}dz\\
        &-C\omega_1\intol\lamt^{3\alpha-1}\mu\hat{\chi}[(\Flm)z\ptau^2\pzphi-\ptau^2\varphi\cdot z\pzphi]^2w^{3\alpha}dz.
    \end{aligned}
\end{equation}
As a result, 
\begin{equation}
    \begin{aligned}
        K_2\ge&\intol\lamt^{3\alpha-1}\frac{13}{12}\mu\hat{\chi}(\ptau^2\varphi)^2w^{3\alpha}dz\\
        &-C\omega_1\intol\lamt^{3\alpha-1}\mu\hat{\chi}[(\Flm)z\ptau^2\pzphi-\ptau^2\varphi\cdot z\pzphi]^2w^{3\alpha}dz-C\mathcal{D}_{\rm tem}(\tau).
    \end{aligned}
\end{equation}
And thus
\begin{equation}\label{Eq: InterEst, I_21}
    \begin{aligned}
        I_{2,1}\ge&\intol\lamt^{3\alpha-1}\frac{13}{12}\mu\hat{\chi}(\ptau^2\varphi)^2w^{3\alpha}dz\\
        &+\intol\lamt^{3\alpha-1}\hat{\chi}\frac{7}{6}\mu (w^3J^{-1})^\alpha[(\Flm)z\ptau^2\pzphi-\ptau^2\varphi\cdot z\pzphi]^2 dz-C\mathcal{D}_{\rm tem}(\tau).
    \end{aligned}
\end{equation}
As for $I_{2,2}$, we have
\begin{equation}\label{Eq: InterEst, I_22}
    \begin{aligned}
        I_{2,2}=&\intol\lamt^{3\alpha-1}\hat{\chi}'\frac{z\ptau^2\varphi}{\Flm}\frac{4}{3}\mu w^{3\alpha}\ptau\left(\frac{(\Flm)^2J^{-\alpha}}{\Flm+z\pzphi}\right)[(\Flm)z\ptauzphi-\ptauphi\cdot z\pzphi]dz\\
        &-\intol\lamt^{3\alpha-1}\hat{\chi}\frac{\ptau^2\varphi}{\Flm}\frac{8}{3}\mu w^{3\alpha}\ptau\left(\frac{(\Flm)^2J^{-\alpha}}{\Flm+z\pzphi}\right)[(\Flm)z\ptauzphi-\ptauphi\cdot z\pzphi]dz\\
        &+\intol\lamt^{3\alpha-1}\hat{\chi}\frac{(1+\varphi)z\ptau^2\pz\varphi-\ptau^2\varphi\cdot z\pzphi}{(\Flm)^2}\\
        &\ \ \ \ \ \ \ \ \ \ \ \ \ \ \ \ \ \ \ \ \ \ \ \ \ \ \cdot\frac{4}{3}\mu w^{3\alpha}\ptau\left(\frac{(\Flm)^2J^{-\alpha}}{\Flm+z\pzphi}\right)[(\Flm)z\ptauzphi-\ptauphi\cdot z\pzphi]dz\\
        \ge&-C\omega_1\mathcal{D}_{\rm tem}(\tau)\\
        &-C\omega_1\intol \lamt^{3\alpha-1}\mu\hat{\chi}(\ptau^2\varphi)^2w^{3\alpha}dz-C\omega_1\intol\lamt^{3\alpha-1}\mu\hat{\chi}[(\Flm)z\ptau^2\pzphi-\ptau^2\varphi\cdot z\pzphi]^2w^{3\alpha}dz\\
        &-C\omega_1\intol\lamt^{3\alpha-1}\mu\hat{\chi}[(\Flm)z\ptauzphi-\ptauphi\cdot z\pzphi]^2w^{3\alpha}dz.\\
    \end{aligned}
\end{equation}
For $I_{2,3}$, we have
\begin{equation}\label{Eq: InterEst, I_23}
    \begin{aligned}
        I_{2,3}=&-\intol\lamt^{3\alpha-1}\frac{4}{3}\mu z^3(\Flm)^2(w^3J^{-1})^\alpha\frac{(\Flm)z\ptauzphi-\ptauphi\cdot z\pzphi}{\Flm+z\pzphi}\pz\left(\hat{\chi}\frac{1}{z^2}\frac{\ptauphi\ptau^2\varphi}{(\Flm)^2}\right)\\
         \ge&-C\omega_1\mathcal{D}_{\rm tem}(\tau)\\
         &-C\omega_1\intol \lamt^{3\alpha-1}\mu\hat{\chi}(\ptau^2\varphi)^2w^{3\alpha}dz-C\omega_1\intol \lamt^{3\alpha-1}\mu\hat{\chi}(z\ptau^2\pz\varphi)^2w^{3\alpha}dz\\
        &-C\omega_1\intol\lamt^{3\alpha-1}\mu\hat{\chi}[(\Flm)z\ptauzphi-\ptauphi\cdot z\pzphi]^2w^{3\alpha}dz.
    \end{aligned}
\end{equation}
Collecting (\ref{Eq: InterEst, I_21})-(\ref{Eq: InterEst, I_23}), making use of the fact that
\begin{equation*}
    \intol\mu\hat{\chi}[(\Flm)z\ptau^2\pzphi-\ptau^2\varphi\cdot z\pzphi]^2w^{3\alpha}dz\lesssim\intol\mu\hat{\chi}[(z\ptau^2\pzphi)^2+\omega_1(\ptau^2\varphi)^2]w^{3\alpha}dz,
\end{equation*}
we arrive at
\begin{equation}\label{Eq: box 3, I2, final}
    \begin{aligned}
        I_2\ge&\intol\lamt^{3\alpha-1}\mu\hat{\chi}(\ptau^2\varphi)^2w^{3\alpha}dz+\intol \frac{1}{2}\lamt^{3\alpha-1}\mu\hat{\chi}(z\ptau^2\pz\varphi)^2w^{3\alpha}dz\\
        &-C\omega_1\intol\lamt^{3\alpha-1}\mu\hat{\chi}\left((z\ptauzphi)^2+(\ptauphi)^2\right)w^{3\alpha}dz-C\mathcal{D}_{\rm tem}(\tau).
    \end{aligned}
\end{equation}
For $I_1$, similar to the estimate of $I_2$, we have
\begin{equation}\label{Eq: box 3, I1, final}
    \begin{aligned}
        I_1\ge&-C_\sigma\intol\lamt^{3\alpha-1}\mu\hat{\chi}\left((z\ptauzphi)^2+(\ptauphi)^2\right)w^{3\alpha}dz-C\mathcal{D}_{\rm tem}(\tau)\\
        &-\sigma\intol\lamt^{3\alpha-1}\mu\hat{\chi}\left((z\ptau^2\pzphi)^2+(\ptau^2\varphi)^2\right)w^{3\alpha}dz,\\
    \end{aligned}
\end{equation}
where $\sigma>0$ is a small constant to be chosen such that the last term on the right-hand side of (\ref{Eq: box 3, I1, final}) can be absorbed by the right-hand side of (\ref{Eq: box 3, I2, final}). Combining (\ref{Eq: box 3, I2, final}) and (\ref{Eq: box 3, I1, final}), we obtain
\begin{equation}\label{Eq: box 3, final}
    \begin{aligned}
        (\square_3)\ge&\intol\frac{1}{3}\lamt^{3\alpha-1}\mu\hat{\chi}(\ptau^2\varphi)^2w^{3\alpha}dz+\intol \frac{1}{3}\lamt^{3\alpha-1}\mu\hat{\chi}(z\ptau^2\pz\varphi)^2w^{3\alpha}dz\\
        &-C\intol\lamt^{3\alpha-1}\mu\hat{\chi}\left((z\ptauzphi)^2+(\ptauphi)^2\right)w^{3\alpha}dz-C\mathcal{D}_{\rm tem}(\tau).\\
    \end{aligned}
\end{equation}

Now, we insert (\ref{Eq: InteriorEst, square1}), (\ref{Eq: interior, ptau N1}) and (\ref{Eq: box 3, final}) into (\ref{Eq: HighOrderInteriorEst, diff}) and integrate with respect to the temporal variable to obtain
\begin{equation}\label{Eq: HighOrderInteriorEst, integrated form}
    \begin{aligned}
        &\lamt^{6\alpha-4}\intol\frac{1}{2}\lamt\hat{\chi}(\ptau^2\varphi)w^3z^2+\frac{2}{3}\hat{\chi}(z\ptauzphi)^2w^4+\frac{3}{2}\delta\hat{\chi}(\ptauphi)^2w^3z^2dz\\
        &+\int_0^\tau\intol\frac{1}{3}\lamt^{3\alpha-1}\left(\mu\hat{\chi}(\ps^2\varphi)^2+\mu\hat{\chi}(z\ps^2\pz\varphi)^2\right)w^{3\alpha}dzds\\
        \le&C\mathcal{E}_{\rm int}(0)+C\int_0^\tau\intol\lamt\mu\hat{\chi}(\psphi)^2w^{3\alpha}+\lamt\mu\hat{\chi}(z\ps\pzphi)^2w^{3\alpha}dzds\\
        &+C\int_0^\tau\left(\lamt^{-\frac{1}{2}}\mathcal{E}_{\rm int}(s)+\lamt^{-\frac{1}{2}}\mathcal{E}_{\rm tem}(s)+\mathcal{D}_{\rm tem}(s)\right)ds.
    \end{aligned}
\end{equation}
Combined with (\ref{Eq: low order interior estimate}), we obtain
\begin{equation}
    \begin{aligned}
        \tilde{\mathcal{E}}_{\rm int}(\tau)+\int_0^\tau\mathcal{D}_{\rm int}(s)ds\le C\mathcal{E}_{\rm int}(0)+C\int_0^\tau\left(\lamt^{-\frac{1}{2}}\mathcal{E}_{\rm int}(s)+\lamt^{-\frac{1}{2}}\mathcal{E}_{\rm tem}(s)+\mathcal{D}_{\rm tem}(s)\right)ds.
    \end{aligned}
\end{equation}
Since
\begin{equation}
    \tilde{\mathcal{E}}_{\rm int}(\tau)\ge C_0\mathcal{E}_{\rm int}(\tau)-C_1\mathcal{E}_{\rm tem}(\tau),
\end{equation}
we obtain
\begin{equation}\label{Eq: interior energy estimate, conclusion}
    \mathcal{E}_{\rm int}(\tau)+\int_0^\tau\mathcal{D}_{\rm int}(s)ds\le C\mathcal{E}_{\rm int}(0)+C\mathcal{E}_{\rm tem}(\tau)+C\int_0^\tau\left(\lamt^{-\frac{1}{2}}\mathcal{E}_{\rm int}(s)+\lamt^{-\frac{1}{2}}\mathcal{E}_{\rm tem}(s)+\mathcal{D}_{\rm tem}(s)\right)ds.
\end{equation}
Inserting (\ref{Eq: Est2}) into (\ref{Eq: interior energy estimate, conclusion}), we prove (\ref{interior energy estimate}).
\end{proof}
\section{Higher-order spatial derivatives}
The goal of this section is to prove regularities of the spatial derivatives. To this end, we rewrite (\ref{phiEqn: linearized}) as
\begin{equation}\label{Eq: rewritten eq}
    \begin{aligned}
        &-\mu\lamt^{\frac{5}{2}-3\alpha}\mathcal{L}_{w^{3\alpha}}\ptauphi-\lamt^{-\frac{1}{2}}\mathcal{L}_{w^4}\varphi\\
        =&\lamt^{\frac{1}{2}}\ptau^2\varphi+\lamt^{-\frac{1}{2}}\lamt'\ptau\varphi+3\delta\lamt^{-\frac{1}{2}}\varphi+\lamt^{-\frac{1}{2}}\mathcal{N}_1[\varphi]-\lamt^{\frac{5}{2}-3\alpha}\mathcal{N}_2[\varphi].\\
    \end{aligned}
\end{equation}
Rewriting the pressure term in (\ref{Eq: rewritten eq}) as
\begin{equation}\label{rewriten operator}
    \mathcal{L}_{w^4}\varphi=w^{4-3\alpha}\mathcal{L}_{w^{3\alpha}}\varphi-\frac{4}{3}(4-3\alpha)w'\pzphi.
\end{equation}
Inserting (\ref{rewriten operator}) into (\ref{Eq: rewritten eq}), we obtain
\begin{equation}\label{Eq: rewritten eq 2}
    \begin{aligned}
        &-\mu\lamt^{\frac{5}{2}-3\alpha}\mathcal{L}_{w^{3\alpha}}\ptauphi-\lamt^{-\frac{1}{2}}w^{4-3\alpha}\mathcal{L}_{w^{3\alpha}}\varphi\\
        =&\lamt^{\frac{1}{2}}\ptau^2\varphi+\lamt^{-\frac{1}{2}}\lamt'\ptau\varphi+3\delta\lamt^{-\frac{1}{2}}\varphi-\frac{4}{3}(4-3\alpha)\lamt^{-\frac{1}{2}}w'\pzphi+\lamt^{-\frac{1}{2}}\mathcal{N}_1[\varphi]-\lamt^{\frac{5}{2}-3\alpha}\mathcal{N}_2[\varphi].
    \end{aligned}
\end{equation}
We first illustrate the outline of the main estimates in this section. In Lemma \ref{lemma: spatial derivatives}, we make use of the form (\ref{Eq: rewritten eq 2}) to obtain the control of
\begin{equation*}
    \mathcal{E}_{\rm spt}(\tau)+\int_0^\tau\mathcal{D}_{\rm spt}(s)ds.
\end{equation*}
Note that the higher order estimate (\ref{Eq: higher spatial est 2}) in Lemma \ref{Lemma: higher spatial est} and the point-wise estimate (\ref{Eq: intLinf}) in Lemma \ref{Lemma: Linfty control} require the point-wise in time estimate of $\mathcal{D}_{\rm spt}(\tau)$. On the other hand, the same control of $\mathcal{D}_{\rm spt}(\tau)$ is also used to bound the term $\int_0^\tau\lamt^{-3\alpha}\mathcal{D}_{\rm spt}(s)ds$ in (\ref{Eq: Est2}). In Lemma \ref{lemma: spatial derivatives 2}, by further rewriting (\ref{Eq: rewritten eq 2}) as
\begin{equation}\label{Eq: rewritten eq 3}
    \begin{aligned}
          -\mu\lamt^{\frac{5}{2}-3\alpha}\mathcal{L}_{w^{3\alpha}}\ptauphi=&\lamt^{\frac{1}{2}}\ptau^2\varphi+\lamt^{-\frac{1}{2}}\lamt'\ptau\varphi+3\delta\lamt^{-\frac{1}{2}}\varphi+\lamt^{-\frac{1}{2}}\mathcal{L}_{w^4}\varphi\\&+\lamt^{-\frac{1}{2}}\mathcal{N}_1[\varphi]-\lamt^{\frac{5}{2}-3\alpha}\mathcal{N}_2[\varphi]\\
    \end{aligned}
\end{equation}
and making full use of the elliptic operator $-\lamt^{\frac{5}{2}-3\alpha}\mu\mathcal{L}_{w^{3\alpha}}\ptauphi$, we can derive the desired point-wise in time estimate of $\mathcal{D}_{\rm spt}(\tau)$. As a result, the term $\int_0^\tau\lamt^{-3\alpha}\mathcal{D}_{\rm spt}(s)ds$ in (\ref{Eq: Est2}) then becomes part of $\int_0^\tau\lamt^{-\min\{\frac{3}{2}\alpha,\frac{1}{2}\}}\mathfrak{E}(s)ds$ in the total energy estimate (\ref{Eq: Total energy est}).

The following Lemma shows the estimates of $\mathcal{E}_{\rm spt}(\tau)$ and $\mathcal{D}_{\rm spt}(\tau)$ imply the weighted $L^2$ estimates of the higher order derivatives of $\varphi$.
\begin{lem}\label{Lemma: higher spatial est}
     Let $-\tilde\varepsilon<\delta$ for $\tilde\varepsilon>0$ small enough. Suppose (\ref{a priori assumptions}) holds for small constant $\omega_0$. Then
     \begin{equation}\label{Eq: higher spatial est 1}
    \begin{aligned}
        \intol\mu^2(\pz^2\varphi)^2 w^{6\alpha-3}z^{2}dz+\intol\mu^2(\pz\varphi)^2 w^{6\alpha-5}dz
        \le C\mathcal{E}_{\rm spt}(\tau),
    \end{aligned}
\end{equation}
     \begin{equation}\label{Eq: higher spatial est 2}
    \begin{aligned}
       \intol\lamt\mu^2(\ptau\pz^2\varphi)^2w^{6\alpha-3}z^{2}dz+\intol\lamt\mu^2(\ptau\pz\varphi)^2w^{6\alpha-5}dz\le C\mathcal{D}_{\rm spt}(\tau).
    \end{aligned}
\end{equation}
\end{lem}
\begin{proof}
   To prove \ref{Eq: higher spatial est 1}, we first notice that
   \begin{equation}
       \intol\mu^2(\pzphi)^2w^{6\alpha-5} dz=\intol\mu^2(w^{3\alpha}z^4\pzphi)^2w^{-5}z^{-8} dz.
   \end{equation}
   Via a similar argument as in \cite{JM}, we can prove $w^{3\alpha}z^4\pzphi|_{z=0,1}=0$ in the sense of trace. By Hardy's inequality (\ref{hardy 2}),
\begin{equation}\label{Hardy's 0}
    \begin{aligned}
       \intol\mu^2(w^{3\alpha}z^4\pzphi)^2w^{-5}z^{-8} dz
       \le C\intol\mu^2\left(\frac{1}{z^4w^3}\pz\left(w^{3\alpha}z^4\pz\varphi\right)\right)^2w^3z^{2}dz.
    \end{aligned}
\end{equation}
Since
\begin{equation}
    \begin{aligned}
        w^{3\alpha-3}\pz^2\varphi=\frac{1}{z^4w^3}\pz\left(w^{3\alpha}z^4\pz\varphi\right)-\frac{3\alpha w^{3\alpha-1}w'z^4+4w^{3\alpha}z^3}{z^4w^3}\pzphi,
    \end{aligned}
\end{equation}
and
\begin{equation}
    \begin{aligned}
        \intol\left|\frac{3\alpha w^{3\alpha-1}w'z^4+4w^{3\alpha}z^3}{z^4w^3}\pzphi\right|^2w^{3}z^{2}dz\lesssim\intol(\pzphi)^2w^{6\alpha-5} dz,
    \end{aligned}
\end{equation}
we have
\begin{equation}
    \begin{aligned}
        \intol\mu^2(\pz^2\varphi)^2 w^{6\alpha-3}z^2 dz\le C\intol\mu^2\left(\frac{1}{z^4w^3}\pz\left(w^{3\alpha}z^4\pz\varphi\right)\right)^2w^3z^{2}dz.
    \end{aligned}
\end{equation}
Thus we prove (\ref{Eq: higher spatial est 1}). The proof of (\ref{Eq: higher spatial est 2}) is almost the same.
\end{proof}

The following Lemma shows the estimates of $\mathcal{E}_{\rm spt}(\tau)$ and $\mathcal{D}_{\rm spt}(\tau)$ imply the point-wise controls of $\varphi$ and the derivatives of $\varphi$.
\begin{lem}\label{Lemma: Linfty control}
    Let $-\tilde\varepsilon<\delta$ for $\tilde\varepsilon>0$ small enough. Suppose (\ref{a priori assumptions}) holds for small constant $\omega_0$. Then
    \begin{equation}
        \begin{aligned}
            &|\varphi(\tau,z)|^2\lesssim\mathcal{E}_{\rm tem}(\tau)+\mathcal{E}_{\rm spt}(\tau),\\
            &|z\pz\varphi(\tau,z)|^2\lesssim \mathcal{E}_{\rm spt}(\tau),
        \end{aligned}
    \end{equation}
    and
    \begin{equation}\label{Eq: intLinf}
    \begin{aligned}
    &\lamt(\tau)|\ptauphi(\tau,z)|^2\lesssim\mathcal{E}_{\rm tem}(\tau)+\mathcal{D}_{\rm spt}(\tau),\\
    &\lamt(\tau)|z\ptau\pz\varphi(\tau,z)|^2\lesssim\mathcal{D}_{\rm spt}(\tau).
    \end{aligned}
\end{equation}
\end{lem}
\begin{proof}
Lemma \ref{Lemma: Est2}, (\ref{Eq: higher spatial est 1}) and Hardy's inequality imply
\begin{equation}
    \begin{aligned}
        \intol\varphi^2dz\lesssim\intol\varphi^2w^3z^4dz+\intol(\pzphi)^2w^{6\alpha-5}dz\lesssim\mathcal{E}_{\rm tem}(\tau)+\mathcal{E}_{\rm spt}(\tau).
    \end{aligned}
\end{equation}
Combined with Sobolev embedding, we obtain
\begin{equation}\label{Pointwise Estimate near origin: ptau_jphi}
    \begin{aligned}
        |\varphi(\tau,z)|^2\lesssim\intol\varphi^2+(\pzphi)^2dz\lesssim\mathcal{E}_{\rm tem}(\tau)+\mathcal{E}_{\rm spt}(\tau).
    \end{aligned}
\end{equation}
At the same time, by Sobolev embedding and Cauchy-Schwarz inequality, making use of (\ref{Eq: higher spatial est 1}) and the fact that $0<\alpha\le\frac{2}{3}$, we have
\begin{equation}\label{Pointwise Estimate near origin: zptau_jpzphi}
    \begin{aligned}
        |z\pz\varphi(\tau,z)|^2\lesssim&\intol\left(|(z\pz\varphi)^2|+|\pz((z\pzphi)^2)|\right) dz\\
        \lesssim&\intol\left((z\pz\varphi)^2+z(\pzphi)^2+|z^2\pzphi\pz^2\varphi|\right) dz\\
        \lesssim&\intol\left((\pz\varphi)^2w^{3-6\alpha}z+(\pz^2\varphi)^2w^{6\alpha-3}z^3\right) dz\lesssim\mathcal{E}_{\rm spt}(\tau).
    \end{aligned}
\end{equation}

Similarly, Lemma \ref{Lemma: Est2}, (\ref{Eq: higher spatial est 2}) and Hardy's inequality yield
\begin{equation}
    \begin{aligned}
        \intol\lamt(\ptau\varphi)^2dz\lesssim&\intol\left(\lamt(\ptau\varphi)^2w^3z^4+\lamt(\ptau\pzphi)^2w^{6\alpha-5}\right)dz\lesssim\mathcal{E}_{\rm tem}(\tau)+\mathcal{D}_{\rm spt}(\tau),
    \end{aligned}
\end{equation}
by Sobolev embedding, we have
\begin{equation}
    \begin{aligned}
        \lamt(\tau)|\ptauphi(\tau,z)|^2\lesssim\intol\lamt(\ptau\varphi)^2+\lamt(\ptau\pz\varphi)^2w^{6\alpha-5}dz\lesssim\mathcal{E}_{\rm tem}(\tau)+\mathcal{D}_{\rm spt}(\tau).
    \end{aligned}
\end{equation}
At the same time, similar to the argument in (\ref{Pointwise Estimate near origin: zptau_jpzphi}),
\begin{equation}
    \begin{aligned}
        \lamt(\tau)|z\ptau\pz\varphi(\tau,z)|^2\lesssim&\lamt\intol\left(|(z\ptau\pz\varphi)^2|+|\pz((z\ptau\pzphi)^2)|\right) dz\\
        \lesssim&\lamt\intol\left((\ptau\pz\varphi)^2w^{3-6\alpha}z+(\ptau\pz^2\varphi)^2w^{6\alpha-3}z^3\right) dz
        \lesssim\mathcal{D}_{\rm spt}(\tau).
    \end{aligned}
\end{equation}
\end{proof}

The following Lemma shows the point-wise in time estimate of $\mathcal{E}_{\rm spt}(\tau)$ and the time integrability of $\mathcal{D}_{\rm spt}(\tau)$.
\begin{lem}\label{lemma: spatial derivatives}
Let $-\tilde\varepsilon<\delta$ for $\tilde\varepsilon>0$ small enough. Suppose (\ref{a priori assumptions}) holds for small constant $\omega_0$. Then
\begin{equation}\label{spatial energy estimate}
    \begin{aligned}
        &\mathcal{E}_{\rm spt}(\tau)+\int_0^\tau\mathcal{D}_{\rm spt}(s)ds\\
        \le&C\mathcal{E}_{\rm spt}(0)+C\mathcal{E}_{\rm int}(0)+C\mathcal{E}_{\rm tem}(0)\\
        &+C\int_0^\tau\left(\lamt^{-1}\mathcal{E}_{\rm spt}(s)+\lamt^{-\frac{1}{2}}\mathcal{E}_{\rm int}(s)+\lamt^{-\min\{\frac{3}{2}\alpha,\frac{1}{2}\}}\mathcal{E}_{\rm tem}(s)+\lamt^{-3\alpha}\mathcal{D}_{\rm spt}(s)\right)ds.
    \end{aligned}
\end{equation}
\end{lem}
\begin{proof}
Taking the square of (\ref{Eq: rewritten eq 2}), multiplying by $\lamt^{6\alpha-4}w^{3}z^{2}$ and integrating over $[0,1]$, we get
\begin{equation}\label{kth derivative of phiEqn: Spt, square directly}
    \begin{aligned}
        &\intol\lamt^{6\alpha-4}(-\lamt^{\frac{5}{2}-3\alpha}\mu\mathcal{L}_{w^{3\alpha}}\ptau\varphi-\lamt^{-\frac{1}{2}}w^{4-3\alpha}\mathcal{L}_{w^{3\alpha}}\varphi)^2w^3z^2dz\\
        \le&\underbrace{C\intol\lamt^{6\alpha-5}(w'\pzphi)^2w^3z^2dz}_{\star}+\underbrace{C\intol\lamt^{6\alpha-4}\left[\lamt^{\frac{1}{2}}\ptau^2\varphi+\lamt^{-\frac{1}{2}}\lamt'\ptau\varphi+3\delta\lamt^{-\frac{1}{2}}\varphi\right]^2w^{3}z^{2}dz}_{\star'}\\
        &+C\intol\lamt^{6\alpha-4}\left(\lamt^{-\frac{1}{2}}\mathcal{N}_1[\varphi]\right)^2w^{3}z^{2}dz+C\intol\lamt^{6\alpha-4}\left(\lamt^{\frac{5}{2}-3\alpha}\mathcal{N}_2[\varphi]\right)^2w^{3}z^{2}dz.
    \end{aligned}
\end{equation}
Making use of the following fact $|w'|\lesssim z$, we have
\begin{equation}
    (\star)\lesssim\intol\lamt^{6\alpha-5}(z\pzphi)^2w^3z^2dz\lesssim\lamt^{-1}\mathcal{E}_{\rm spt}(\tau).
\end{equation}
While for the pure temporal derivative term, we can estimate as follows
\begin{equation}
    (\star')\lesssim \intol\lamt^{6\alpha-4}\left(\lamt\sum_{j=1}^{2}(\ptau^j\varphi)^2+\lamt^{-1}\varphi^2\right)w^{3}z^{2}dz\lesssim \lamt^{-1}\mathcal{E}_{\rm tem}(\tau)+\mathcal{D}_{\rm tem}(\tau)+\mathcal{D}_{\rm int}(\tau).
\end{equation}
For the left-hand side of (\ref{kth derivative of phiEqn: Spt, square directly}), we have
\begin{equation}
\begin{aligned}
&(\text{L.H.S. of }(\ref{kth derivative of phiEqn: Spt, square directly}))\\
    =&\intol\lamt^{6\alpha-4}(-\lamt^{\frac{5}{2}-3\alpha}\mu\mathcal{L}_{w^{3\alpha}}\ptau\varphi-\lamt^{-\frac{1}{2}}w^{4-3\alpha}\mathcal{L}_{w^{3\alpha}}\varphi)^2w^3z^2dz\\
    =&\intol\mu^2\lamt(\mathcal{L}_{w^{3\alpha}}\ptau\varphi)^2w^3z^2+\lamt^{6\alpha-5}(w^{4-3\alpha}\mathcal{L}_{w^{3\alpha}}\varphi)^2w^3z^2+2\mu\lamt^{3\alpha-2} w^{4-3\alpha}\mathcal{L}_{w^{3\alpha}}\ptau\varphi\cdot\mathcal{L}_{w^{3\alpha}}\varphi\cdot w^3z^2 dz\\
    =&\frac{d}{d\tau}\intol\mu\lamt^{3\alpha-2}(\mathcal{L}_{w^{3\alpha}}\varphi)^2w^{7-3\alpha}z^2dz+\intol(2-3\alpha)\mu\lamt^{3\alpha-3}\lamt'(\mathcal{L}_{w^{3\alpha}}\varphi)^2w^{7-3\alpha}z^2dz\\
    &+\intol\mu^2\lamt(\mathcal{L}_{w^{3\alpha}}\ptau\varphi)^2w^3z^2+\lamt^{6\alpha-5}(w^{4-3\alpha}\mathcal{L}_{w^{3\alpha}}\varphi)^2w^3z^2dz.\\
\end{aligned}
\end{equation}
As a result, (\ref{kth derivative of phiEqn: Spt, square directly}) becomes
\begin{equation}\label{kth derivative of phiEqn: Spt, square directly Ineq 3}
    \begin{aligned}
       &\frac{d}{d\tau}\intol\mu\lamt^{3\alpha-2}(\mathcal{L}_{w^{3\alpha}}\varphi)^2w^{7-3\alpha}z^2dz+\intol(2-3\alpha)\mu\lamt^{3\alpha-3}\lamt'(\mathcal{L}_{w^{3\alpha}}\varphi)^2w^{7-3\alpha}z^2dz\\
    &+\intol\mu^2\lamt(\mathcal{L}_{w^{3\alpha}}\ptau\varphi)^2w^3z^2+\lamt^{6\alpha-5}(w^{4-3\alpha}\mathcal{L}_{w^{3\alpha}}\varphi)^2w^3z^2dz\\
        \le&C\intol\lamt^{6\alpha-4}\left(\lamt^{-\frac{1}{2}}\mathcal{N}_1[\varphi]\right)^2w^{3}z^{2}dz+C\intol\lamt^{6\alpha-4}\left(\lamt^{\frac{5}{2}-3\alpha}\mathcal{N}_2[\varphi]\right)^2w^{3}z^{2}dz\\
        &+\lamt^{-1}\mathcal{E}_{\rm tem}(\tau)+C\lamt^{-1}\mathcal{E}_{\rm spt}(\tau)+\mathcal{D}_{\rm tem}(\tau)+\mathcal{D}_{\rm int}(\tau).\\
    \end{aligned}
\end{equation}

Next, we estimate the nonlinear terms. 

{\RaggedRight\underline{\textbf{Estimate of $\mathcal{N}_1[\varphi]$}:}} In this step, we aim to prove
\begin{equation}\label{Eq: Estimate of N_1}
    \begin{aligned}
        \intol \lamt^{6\alpha-4}\left(\lamt^{-\frac{1}{2}}\mathcal{N}_1[\varphi]\right)^2w^{3}z^{2}dz
        \le C\omega_0\lamt^{-1}\mathcal{E}_{\rm tem}(\tau)+C\omega_1\lamt^{6\alpha-5}\intol\mu^2(\mathcal{L}_{w^{3\alpha}}\varphi)^2w^{3}z^2dz.
    \end{aligned}
\end{equation}

Recall the definition of $\mathcal{N}_1[\varphi]$ from (\ref{Eq: nonlinearity, N1}), by direct computation, we have
\begin{equation}
    \begin{aligned}
        &\intol \lamt^{6\alpha-4}\left(\lamt^{-\frac{1}{2}}\mathcal{N}_1[\varphi]\right)^2w^{3}z^{2}dz\\
        \le&C \intol\lamt^{6\alpha-4}\left[\lamt^{-\frac{1}{2}}\left(1-2\varphi-\frac{1}{(\Flm)^2}\right)\right]^2w^{3}z^{2}dz\\
        &+C \intol\lamt^{6\alpha-4}\left[\lamt^{-\frac{1}{2}}\left(\frac{(\Flm)^2-1}{zw^3}\pz\left(\frac{w^4}{z^2}\pz\left(z^3\varphi\right)\right)\right)\right]^2w^{3}z^{2}dz\\
        &+C \intol \lamt^{6\alpha-4}\left[\lamt^{-\frac{1}{2}}\left(\frac{(\Flm)^2}{zw^3}\pz\left[\frac{w^4}{z^2}\pz\left(z^3(\varphi^2+\frac{1}{3}\varphi^3)\right)\right]\right)\right]^2 w^{3}z^{2}dz\\
        &+C \intol \lamt^{6\alpha-4}\left[\lamt^{-\frac{1}{2}}\left(\frac{(\Flm)^2}{zw^3}\pz\left(\frac{28}{9}w^4(J-1)^2\intol(1-\theta)[1+\theta(J-1)]^{-\frac{10}{3}}d\theta\right)\right)\right]^2 w^{3}z^{2}dz\\
        &+C \intol\lamt^{6\alpha-4} \left[\lamt^{-\frac{1}{2}}\left( \frac{4w'}{z}\left((\Flm)^2-\frac{1}{(\Flm)^2}-4\varphi\right)\right) \right]^2 w^{3}z^{2}dz\\
        =&:\sum_{j=1}^5I_j.
    \end{aligned}
\end{equation}
For $I_1$, we have
\begin{equation}\label{Eq: I_1, spt est}
    \begin{aligned}
        I_1\lesssim&\lamt^{6\alpha-5}\intol\left(\frac{-3\varphi^2-2\varphi^3}{(\Flm)^2}\right)^2w^{3}z^{2}dz\lesssim\omega_1\lamt^{6\alpha-5}\intol\varphi^2w^{3}z^{2}dz.
    \end{aligned}
\end{equation}
For $I_2$, we have
\begin{equation}\label{Eq: I_2, spt est}
    \begin{aligned}
        I_2\lesssim&\lamt^{6\alpha-5}\intol\left[\frac{(\Flm)^2-1}{zw^3}\pz\left(\frac{w^4}{z^2}\pz\left(z^3\varphi\right)\right)\right]^2w^{3}z^{2}dz\\
        \lesssim&\omega_1\lamt^{6\alpha-5}\intol(\pz^2\varphi)^2 w^{5}z^{2}+(\pzphi)^2( w^{5}+ w^{3}z^{4})+\varphi^2w^{3}z^{2}dz\\
        \lesssim&\omega_1\lamt^{6\alpha-5}\intol\mu^2(\mathcal{L}_{w^{3\alpha}}\varphi)^2w^{3}z^2dz+\omega_1\lamt^{6\alpha-5}\intol\varphi^2w^{3}z^{2}dz.
    \end{aligned}
\end{equation}
For $I_3$, we have
\begin{equation}\label{Eq: I_3, spt est}
    \begin{aligned}
        I_3\lesssim&\lamt^{6\alpha-5}\intol\left[\frac{(\Flm)^2}{zw^3}\pz\left[\frac{w^4}{z^2}\pz\left(z^3(\varphi^2+\frac{1}{3}\varphi^3)\right)\right]\right]^2 w^{3}z^{2}dz\\
        \lesssim&\omega_1\lamt^{6\alpha-5}\intol(\pz^2\varphi)^2 w^{5}z^{2}+(\pzphi)^2(w^{5}+ w^{3}z^{4})+\varphi^2w^{3}z^{2}dz\\
        \lesssim&\omega_1\lamt^{6\alpha-5}\intol\mu^2(\mathcal{L}_{w^{3\alpha}}\varphi)^2w^{3}z^2dz+\omega_1\lamt^{6\alpha-5}\intol\varphi^2w^{3}z^{2}dz,
    \end{aligned}
\end{equation}
since
\begin{equation}
    \begin{aligned}
        &\left|\frac{(\Flm)^2}{zw^3}\pz\left[\frac{w^4}{z^2}\pz\left(z^3(\varphi^2+\frac{1}{3}\varphi^3)\right)\right]\right|\\
        \lesssim&\omega_1\left[\left|\pz^2\varphi\right|w+\left|\pzphi\right|(z+\frac{w}{z})+|\varphi|\right].\\
    \end{aligned}
\end{equation}
For $I_4$, we can estimate as follows
\begin{equation}\label{Eq: I_4, spt est}
    \begin{aligned}
        I_4\lesssim&\left| \lamt^{6\alpha-4}\intol \left[\lamt^{-\frac{1}{2}}\left(\frac{(\Flm)^2}{zw^3}\pz\left(\frac{28}{9}w^4(J-1)^2\intol(1-\theta)[1+\theta(J-1)]^{-\frac{10}{3}}d\theta\right)\right)\right]^2 w^{3}z^{2}dz\right|\\
       \lesssim&\omega_1\lamt^{6\alpha-5}\intol(\pz^2\varphi)^2 w^{5}z^{2}+(\pzphi)^2( w^{5} + w^{3}z^{4})+\varphi^2w^{3}z^{2}dz\\
        \lesssim&\omega_1\lamt^{6\alpha-5}\intol\mu^2(\mathcal{L}_{w^{3\alpha}}\varphi)^2w^{3}z^2dz+\omega_1\lamt^{6\alpha-5}\intol\varphi^2w^{3}z^{2}dz,
    \end{aligned}
\end{equation}
since
\begin{equation}
    \begin{aligned}
        &\left|\frac{1}{zw^3}\pz\left(w^4(J-1)^2\intol(1-\theta)[1+\theta(J-1)]^{-\frac{10}{3}}d\theta\right)\right|\\
        \lesssim&\left|\intol\frac{(1-\theta)(J-1)^2}{[1+\theta(J-1)]^{\frac{10}{3}}}d\theta\right|\\
        &+\left|\frac{w}{z^2}\intol\frac{(1-\theta)(J-1)(4(\Flm)^2z\pzphi+2(\Flm)z^2(\pzphi)^2+(\Flm)^2z^2\pz^2\varphi)}{[1+\theta(J-1)]^{\frac{10}{3}}}d\theta\right|\\
        &+\left|\frac{w}{z^2}\intol\frac{\theta(1-\theta)(J-1)^24(\Flm)^2z\pzphi+2(\Flm)z^2(\pzphi)^2+(\Flm)^2z^2\pz^2\varphi}{[1+\theta(J-1)]^{\frac{13}{3}}}d\theta\right|\\
        \lesssim&\omega_1\left[\left|\pz^2\varphi\right|w+\left|\pzphi\right|(z+\frac{w}{z})+|\varphi|\right],
    \end{aligned}
\end{equation}
where we make use of the fact
\begin{equation}\label{Eq: pz of J}
    \pz J=4(\Flm)^2\pzphi+2(\Flm)z(\pzphi)^2+(\Flm)^2z\pz^2\varphi.
\end{equation}
For $I_5$, we have
\begin{equation}\label{Eq: I_5, spt est}
    \begin{aligned}
        I_5\lesssim&\left| \lamt^{6\alpha-4}\intol \left[
        \lamt^{-\frac{1}{2}}\left( \frac{4w'}{z}\left((\Flm)^2-\frac{1}{(\Flm)^2}-4\varphi\right)
        \right)\right]^2 w^{3}z^{2}dz\right|\\
        \lesssim&\lamt^{6\alpha-5}\intol
        \left|
        \left(\frac{\varphi^2\cdot p(\varphi)}{(\Flm)^2}\right)\right|^2 w^{3}z^{2}dz\lesssim\omega_1\lamt^{6\alpha-5}\intol\varphi^2w^{3}z^{2}dz.
    \end{aligned}
\end{equation}
Collecting (\ref{Eq: I_1, spt est}), (\ref{Eq: I_2, spt est}), (\ref{Eq: I_3, spt est}), (\ref{Eq: I_4, spt est}) and (\ref{Eq: I_5, spt est}), we obtain (\ref{Eq: Estimate of N_1}).

{\RaggedRight\underline{\textbf{Estimate of $\mathcal{N}_2[\varphi]$:}}} In this step, we prove
\begin{equation}\label{Eq: Estimate of N_2}
    \begin{aligned}
    \intol\lamt^{6\alpha-4}\left(\lamt^{\frac{5}{2}-3\alpha}\mathcal{N}_2[\varphi]\right)^2w^{3}z^{2}dz
    \lesssim C\omega_0\mathcal{D}_{\rm tem}(\tau)+\omega_0\mu^2\lamt\intol(\mathcal{L}_{w^{3\alpha}}\ptau\varphi)^2w^3z^2dz.\\
    \end{aligned}
\end{equation}

Recall the definition of $\mathcal{N}_2[\varphi]$ from (\ref{Eq: N_2}), by direct computation, we have
\begin{equation}
\medmath{
  \begin{aligned}
        &\intol\lamt^{6\alpha-4}\left(\lamt^{\frac{5}{2}-3\alpha}\mathcal{N}_2[\varphi]\right)^2w^{3}z^{2}dz\\
        \lesssim&\lamt\intol\left[\frac{1}{z^4w^3}\left(\frac{1}{\Flm}-1\right)\pz\left(\frac{4}{3}\mu z^4w^{3\alpha}\ptau\pz\varphi\right)\right]^2w^{3}z^{2}dz\\
        &+\lamt\intol\left[\frac{1}{z^4w^3(\Flm)}\pz\left(\frac{4}{3}\mu z^3w^{3\alpha}(J^{-\alpha-1}-1)z\ptau\pz\varphi\right)\right]^2w^{3}z^{2}dz\\
        &+\lamt\intol\left[\frac{1}{z^4w^3(\Flm)}\pz\left(\frac{4}{3}\mu z^3w^{3\alpha}J^{-\alpha-1}([(1+\varphi)^5-1]z\ptau\pz\varphi-(\Flm)^4z\pz\varphi\ptauphi)\right)\right]^2w^{3}z^{2}dz\\
        =&:\sum_{i=1}^3II_i.
    \end{aligned}
    }
\end{equation}
For $II_1$, we have
\begin{equation}\label{Eq: II_1}
    \begin{aligned}
        II_1\lesssim\omega_1\intol\mu^2\lamt(\mathcal{L}_{w^{3\alpha}}\ptau\varphi)^2w^3z^2dz.
    \end{aligned}
\end{equation}
For $II_2$, we recall (\ref{Eq: pz of J}) and calculate directly to obtain
\begin{equation}\label{Eq: II_2}
    \begin{aligned}
        II_2
        \lesssim&\lamt\intol\left[\frac{1}{z^4w^3}\pz\left(\frac{4}{3}\mu z^3w^{3\alpha}(J^{-\alpha-1}-1)z\ptau\pz\varphi\right)\right]^2w^{3}z^{2}dz\\
        =&\lamt\intol\left[\frac{1}{z^4w^3}\left(\frac{16}{3}\mu z^2w^{3\alpha}(J^{-\alpha-1}-1)z\ptau\pz\varphi+4\alpha\mu z^3w^{3\alpha-1}w'(J^{-\alpha-1}-1)z\ptau\pz\varphi\right.\right.\\
        &\left.\left.\ \ \ \ \ \ \ \ \ \ \ \ \ \ \ \ \ +\frac{4}{3}\mu z^2w^{3\alpha}p(\varphi,z\pz\varphi)z\ptau\pz\varphi-\frac{4}{3}(\alpha+1)\mu z^3w^{3\alpha}(\Flm)^2z\pz^2\varphi \cdot z\ptau\pz\varphi J^{-\alpha-2}\right.\right.\\
        &\left.\left.\ \ \ \ \ \ \ \ \ \ \ \ \ \ \ \ \ +\frac{4}{3}\mu z^3w^{3\alpha}(J^{-\alpha-1}-1)z\ptau\pz^2\varphi\right)\right]^2w^{3}z^{2}dz\\
        \lesssim&\omega_1\lamt\intol\mu^2(\ptau\pz^2\varphi)^2w^{6\alpha-3}z^{2}+\mu^2(\ptau\pzphi)^2w^{6\alpha-5}dz\\
        &+\lamt\intol\mu^2(z\ptau\pz\varphi\cdot\pz^2\varphi)^2w^{6\alpha-3}z^2dz,
    \end{aligned}
\end{equation}
where $p(\varphi,z\pzphi)$ is a polynomial without constant term. For the first integral on the right-hand side of (\ref{Eq: II_2}), we apply Lemma \ref{Lemma: higher spatial est} to get
\begin{equation}\label{Eq: II_2,www}
    \begin{aligned}
        \omega_1\lamt\intol\mu^2(\ptau\pz^2\varphi)^2w^{6\alpha-3}z^2+\mu^2(\ptau\pzphi)^2w^{6\alpha-5}dz\lesssim\omega_1\intol\mu^2\lamt(\mathcal{L}_{w^{3\alpha}}\ptau\varphi)^2w^3z^2dz.
    \end{aligned}
\end{equation}
The last term on the right-hand side of (\ref{Eq: II_2}) can be estimated as follows
\begin{equation}\label{Eq: product term estimate, 1}
    \begin{aligned}
        &\lamt\intol\mu^2(z\ptau\pz\varphi\cdot\pz^2\varphi)^2w^{6\alpha-3}z^2dz\\
        \lesssim&\mu^2\lamt(\tau)\|z\ptau\pzphi\|^2_{L^\infty}\intol(\pz^2\varphi)^2w^{6\alpha-3}z^2dz\\
        \lesssim&\mu^2\lamt(\tau)\|z\ptau\pzphi\|^2_{L^\infty}\mathcal{E}_{\rm spt}(\tau)\lesssim\omega_0\intol\mu^2\lamt(\mathcal{L}_{w^{3\alpha}}\ptau\varphi)^2w^3z^2dz,\\
    \end{aligned}
\end{equation}
where we also use (\ref{Eq: higher spatial est 1}) in Lemma \ref{Lemma: higher spatial est}, (\ref{Eq: intLinf}) in Lemma \ref{Lemma: Linfty control} and the fact following from the a priori assumptions (\ref{a priori assumptions}) that $\mathcal{E}_{\rm spt}(\tau)\le\omega_0^2$. Inserting (\ref{Eq: II_2,www}) and (\ref{Eq: product term estimate, 1}) into (\ref{Eq: II_2}) we thus have
\begin{equation}\label{Eq: II_2 final_}
    \begin{aligned}
        II_2
        \lesssim&\omega_0\intol\mu^2\lamt(\mathcal{L}_{w^{3\alpha}}\ptau\varphi)^2w^3z^2dz.
    \end{aligned}
\end{equation}
For $II_3$, we have
\begin{equation}
    \medmath{
    \begin{aligned}
        II_3=&\lamt\intol\left[\frac{1}{z^4w^3(\Flm)}\pz\left(\frac{4}{3}\mu z^3w^{3\alpha}J^{-\alpha-1}([(1+\varphi)^5-1]z\ptau\pz\varphi-(\Flm)^4z\pz\varphi\ptauphi)\right)\right]^2w^{3}z^{2}dz\\
        \lesssim&\lamt\intol\left[\frac{1}{z^4w^3}\pz\left(\frac{4}{3}\mu z^3w^{3\alpha}J^{-\alpha-1}[(1+\varphi)^5-1]z\ptau\pz\varphi\right)\right]^2w^{3}z^{2}dz\\
        &+\lamt\intol\left[\frac{1}{z^4w^3}\pz\left(\frac{4}{3}\mu z^3w^{3\alpha}J^{-\alpha-1}(\Flm)^4z\pz\varphi\ptauphi\right)\right]^2w^{3}z^{2}dz\\
        =&:II_{3,1}+II_{3,2}.\\
    \end{aligned}}
\end{equation}
For $II_{3,1}$, since
\begin{equation}
    \medmath{
    \begin{aligned}
        &\left|\frac{1}{z^4w^3}\pz\left(\frac{4}{3}\mu z^3w^{3\alpha}J^{-\alpha-1}[(1+\varphi)^5-1]z\ptau\pz\varphi\right)\right|\\
        =&\left|\frac{1}{z^4w^3}\left(\frac{16}{3}\mu z^2w^{3\alpha}J^{-\alpha-1}[(1+\varphi)^5-1]z\ptau\pz\varphi+4\alpha\mu z^3w^{3\alpha-1}w'J^{-\alpha-1}[(1+\varphi)^5-1]z\ptau\pz\varphi\right.\right.\\
        &\left.\left.\ \ \ \ \ \ \ \ \ \ +\frac{4}{3}\mu(-\alpha-1) z^3w^{3\alpha} J^{-\alpha-2}\pz J[(1+\varphi)^5-1]z\ptau\pz\varphi-\frac{4}{3}(\alpha+1)\mu z^3w^{3\alpha}J^{-\alpha-1}(1+\varphi)^4\cdot z\pzphi\cdot \ptau\pz\varphi\right.\right.\\
        &\left.\left.\ \ \ \ \ \ \ \ \ \ +\frac{4}{3}\mu z^3w^{3\alpha}J^{-\alpha-1}[(1+\varphi)^5-1]z\ptau\pz^2\varphi\right)\right|\\
        \lesssim&\omega_1\left[|\mu\ptau\pzphi|\left(\frac{w^{3\alpha-3}}{z}+w^{3\alpha-4}z\right)+|\mu\ptau\pz^2\varphi|w^{3\alpha-3}\right]\\
        &+\left|\frac{1}{z^4w^3}\left(\mu z^3w^{3\alpha} J^{-\alpha-2}(\Flm)^2[(1+\varphi)^5-1]z\ptau\pz\varphi\cdot z\pz^2\varphi\right)\right|\\
        &+\left|\frac{1}{z^4w^3}\left(\mu z^2w^{3\alpha} J^{-\alpha-2}\left[4(\Flm)^2z\pzphi+2(\Flm)(z\pzphi)^2\right][(1+\varphi)^5-1]z\ptau\pz\varphi\right)\right|\\
        \lesssim&\omega_1\left[|\mu\ptau\pzphi|\left(\frac{w^{3\alpha-3}}{z}+w^{3\alpha-4}z\right)+|\mu\ptau\pz^2\varphi|w^{3\alpha-3}\right]+|z\ptau\pzphi\cdot\mu\pz^2\varphi|w^{3\alpha-3},
    \end{aligned}
    }
\end{equation}
we have
\begin{equation}
    \begin{aligned}
        II_{3,1}\lesssim&\omega_1\lamt\intol\mu^2(\ptau\pz^2\varphi)^2w^{6\alpha-3}z^2+\mu^2(\ptau\pzphi)^2w^{6\alpha-5}dz\\
        &+\lamt\intol\mu^2(z\ptau\pz\varphi\cdot\pz^2\varphi)^2w^{6\alpha-3}z^2dz.
    \end{aligned}
\end{equation}
For $II_{3,2}$, since
\begin{equation}
    \medmath{
    \begin{aligned}
        &\left|\frac{1}{z^4w^3}\pz\left(\frac{4}{3}\mu z^3w^{3\alpha}J^{-\alpha-1}(\Flm)^4z\pz\varphi\ptauphi\right)\right|\\
        =&\left|\frac{1}{z^4w^3}\left(\frac{16}{3}\mu z^2w^{3\alpha}J^{-\alpha-1}(\Flm)^4z\pz\varphi\ptauphi+4\alpha\mu z^3w^{3\alpha-1}w'J^{-\alpha-1}(\Flm)^4z\pz\varphi\ptauphi\right.\right.\\
        &\left.\left.\ \ \ \ \ \ \ \ \ \ \ \ \ \ +\frac{4}{3}(-\alpha-1)\mu z^3w^{3\alpha} J^{-\alpha-2}\pz J(\Flm)^4z\pz\varphi\ptauphi+\frac{16}{3}\mu z^2w^{3\alpha}J^{-\alpha-1}(\Flm)^3(z\pz\varphi)^2\ptauphi\right.\right.\\
        &\left.\left.\ \ \ \ \ \ \ \ \ \ \ \ \ \ +\frac{4}{3}\mu z^3w^{3\alpha}J^{-\alpha-1}(\Flm)^4\ptauphi\cdot z\pz^2\varphi+\frac{4}{3}\mu z^2w^{3\alpha}J^{-\alpha-1}(\Flm)^4z\pz\varphi\cdot z\ptau\pzphi\right)\right|\\
        \lesssim&\omega_1|\mu\ptau\pzphi|\frac{w^{3\alpha-3}}{z}+|\ptau\varphi\cdot\mu\pzphi|\left(\frac{w^{3\alpha-3}}{z}+w^{3\alpha-4}z\right)+|\ptau\varphi\cdot\mu\pz^2\varphi|w^{3\alpha-3}\\
        &+\left|\frac{1}{z^4w^3}\left(\mu z^3w^{3\alpha} J^{-\alpha-2}(\Flm)^6z\pz\varphi\ptauphi\cdot z\pz^2\varphi\right)\right|\\
        &+\left|\frac{1}{z^4w^3}\left(\mu z^2w^{3\alpha} J^{-\alpha-2}[4(\Flm)^2z\pzphi+2(\Flm)(z\pzphi)^2](\Flm)^4z\pz\varphi\ptauphi\right)\right|\\
        \lesssim&\omega_1|\mu\ptau\pzphi|\frac{w^{3\alpha-3}}{z}+|\ptau\varphi\cdot\mu\pzphi|\left(\frac{w^{3\alpha-3}}{z}+w^{3\alpha-4}z\right)+|\ptau\varphi\cdot\mu\pz^2\varphi|w^{3\alpha-3},
    \end{aligned}
    }
\end{equation}
it follows that
\begin{equation}
    \begin{aligned}
        II_{3,2}\lesssim&\omega_1\lamt\intol\mu^2(\ptau\pzphi)^2w^{6\alpha-3}dz\\
        &+\lamt\intol\mu^2(\ptau\varphi\cdot\pzphi)^2(w^{6\alpha-3}+w^{6\alpha-5}z^4)+\mu^2(\ptau\varphi\cdot\pz^2\varphi)^2w^{6\alpha-3}z^2dz.
    \end{aligned}
\end{equation}
As a result, it holds that
\begin{equation}\label{Eq: II_3, 1}
    \begin{aligned}
        II_3\lesssim&\omega_1\lamt\intol\mu^2(\ptau\pz^2\varphi)^2w^{6\alpha-3}z^2+\mu^2(\ptau\pzphi)^2w^{6\alpha-5}dz\\
        &+\lamt\intol\mu^2(z\ptau\pz\varphi\cdot\pz^2\varphi)^2w^{6\alpha-3}z^2+\mu^2(\ptau\varphi\cdot\pzphi)^2(w^{6\alpha-3}+w^{6\alpha-5}z^4)dz\\
        &+\lamt\intol\mu^2(\ptau\varphi\cdot\pz^2\varphi)^2w^{6\alpha-3}z^2dz.
    \end{aligned}
\end{equation}
For the first line on the right-hand side of (\ref{Eq: II_3, 1}), it can be estimated in the same way as in (\ref{Eq: II_2,www}). For the product terms coming from the second and third lines of (\ref{Eq: II_3, 1}), similar to (\ref{Eq: product term estimate, 1}), we estimate as follows
\begin{equation}
    \begin{aligned}
        &\lamt\intol\mu^2(z\ptau\pz\varphi\cdot\pz^2\varphi)^2w^{6\alpha-3}z^2+\mu^2(\ptau\varphi\cdot\pzphi)^2(w^{6\alpha-3}+w^{6\alpha-5}z^4)dz\\
        &+\lamt\intol\mu^2(\ptau\varphi\cdot\pz^2\varphi)^2w^{6\alpha-3}z^2dz\\
        \lesssim&\mu^2\lamt(\|\ptau\varphi\|_{L^\infty}^2+\|z\ptau\pzphi\|^2_{L^\infty})\intol(\pz^2\varphi)^2w^{6\alpha-3}z^2dz\\
        &+\mu^2\lamt\|\ptau\varphi\|_{L^\infty}^2\intol(\pz\varphi)^2w^{6\alpha-5}dz\\
        \lesssim&\mu^2\lamt(\|\ptau\varphi\|_{L^\infty}^2+\|z\ptau\pzphi\|^2_{L^\infty})\mathcal{E}_{\rm spt}(\tau)\\
        \lesssim&\omega_0^2\mu^2\lamt\intol(\ptau\varphi)^2w^3z^4+(\mathcal{L}_{w^{3\alpha}}\ptau\varphi)^2w^3z^2dz,
    \end{aligned}
\end{equation}
where the last step follows from Lemma \ref{Lemma: Linfty control}. We thus have
\begin{equation}\label{Eq: II_3 final_}
    \begin{aligned}
        II_3\lesssim&\omega_0\mu^2\lamt\intol(\ptau\varphi)^2w^3z^4dz+\omega_0\mu^2\lamt\intol(\mathcal{L}_{w^{3\alpha}}\ptau\varphi)^2w^3z^2dz.
    \end{aligned}
\end{equation}
Collecting (\ref{Eq: II_1}), (\ref{Eq: II_2 final_}) and (\ref{Eq: II_3 final_}), we have (\ref{Eq: Estimate of N_2}).

Collecting all the estimates of the terms in (\ref{kth derivative of phiEqn: Spt, square directly Ineq 3}) and integrating over $[0,\tau]$, we get
\begin{equation}\label{kth derivative of phiEqn: Spt, square directly Ineq 5}
    \begin{aligned}
       &\intol\mu\lamt^{3\alpha-2}(\mathcal{L}_{w^{3\alpha}}\varphi)^2w^{7-3\alpha}z^2dz\\
    &+\int_0^\tau\intol\mu^2\lamt(\mathcal{L}_{w^{3\alpha}}\ps\varphi)^2w^3z^2+\lamt^{6\alpha-5}(w^{4-3\alpha}\mathcal{L}_{w^{3\alpha}}\varphi)^2w^3z^2dzds\\
        \le&C\mathcal{E}_{\rm spt}(0)+C\int_0^\tau\left(\lamt^{-1}\mathcal{E}_{\rm tem}(s)+\lamt^{-1}\mathcal{E}_{\rm spt}(s)+\mathcal{D}_{\rm tem}(s)+\mathcal{D}_{\rm int}(s)\right)ds
    \end{aligned}
\end{equation}
Combined with 
\begin{equation}
    \begin{aligned}
        &\intol\mu^2(\mathcal{L}_{w^{3\alpha}}\varphi)^2w^3z^2dz\\\lesssim&\intol\mu^2(\mathcal{L}_{w^{3\alpha}}\varphi_0)^2w^3z^2dz+\int_0^\tau\intol\mu^2\lamt(\mathcal{L}_{w^{3\alpha}}\ps\varphi)^2w^3z^2dzds,
    \end{aligned}
\end{equation}
also making use of (\ref{Eq: Est2}) and (\ref{interior energy estimate}), we have (\ref{spatial energy estimate}).
\end{proof}
At last, we carry out the point-wise in time estimate of $\mathcal{D}_{\rm spt}(\tau)$.
\begin{lem}\label{lemma: spatial derivatives 2}
Let $-\tilde\varepsilon<\delta$ for $\tilde\varepsilon>0$ small enough. Suppose (\ref{a priori assumptions}) holds for small constant $\omega_0$. Then
\begin{equation}\label{spatial energy estimate 2}
    \begin{aligned}
        \mathcal{D}_{\rm spt}(\tau)\le C\mathcal{E}_{\rm int}(\tau)+C\mathcal{E}_{\rm tem}(\tau)+C\lamt^{-1}\mathcal{E}_{\rm spt}(\tau).
    \end{aligned}
\end{equation}
\end{lem}
\begin{proof}
Take the square of (\ref{Eq: rewritten eq 3}), multiply it by $\lamt^{6\alpha-4}w^3z^2$ and integrate the resulting equation in $z$ over $[0,1]$ to get
\begin{equation}\label{Eq: 222220}
\begin{aligned}
    &\intol\lamt\mu^2(\mathcal{L}_{w^{3\alpha}}\ptauphi)^2w^3z^2dz\\
    \le&C\mathcal{E}_{\rm int}(\tau)+C\mathcal{E}_{\rm tem}(\tau)+C\lamt^{-1}\mathcal{E}_{\rm spt}(\tau)\\
    &+\intol\lamt^{6\alpha-5}(\mathcal{N}_1[\varphi])^2w^3z^2dz+\intol\lamt(\mathcal{N}_2[\varphi])^2w^3z^2dz.
\end{aligned}
\end{equation}
The estimate of the $\mathcal{N}_1[\varphi]$ term is almost the same as (\ref{Eq: Estimate of N_1}):
\begin{equation}\label{Eq: 22222}
\begin{aligned}
    \intol \lamt^{6\alpha-5}\left(\mathcal{N}_1[\varphi]\right)^2w^{3}z^{2}dz
        \le C\omega_0\lamt^{-1}\mathcal{E}_{\rm tem}(\tau)+C\omega_0\lamt^{-1}\mathcal{E}_{\rm spt}(\tau).
\end{aligned}
\end{equation}
As for the $\mathcal{N}_2[\varphi]$ term, we have
\begin{equation}\label{Eq: 222221}
    \begin{aligned}
    \intol \lamt\left(\mathcal{N}_2[\varphi]\right)^2w^{3}z^{2}dz    
        \le&C\omega_0\mathcal{E}_{\rm tem}(\tau)+C\omega_0\mu^2\lamt\intol(\mathcal{L}_{w^{3\alpha}}\ptau\varphi)^2w^3z^2dz.\\
    \end{aligned}
\end{equation}
Inserting (\ref{Eq: 22222}) and (\ref{Eq: 222221}) into (\ref{Eq: 222220})
\begin{equation}\label{Eq: 3333}
\begin{aligned}
    &\intol\lamt\mu^2(\mathcal{L}_{w^{3\alpha}}\ptauphi)^2w^3z^2dz\le C\mathcal{E}_{\rm int}(\tau)+C\mathcal{E}_{\rm tem}(\tau)+C\lamt^{-1}\mathcal{E}_{\rm spt}(\tau).\\
\end{aligned}
\end{equation}
While for the term $\intol\lamt^{6\alpha-5}(w^{4-3\alpha}\mathcal{L}_{w^{3\alpha}}\varphi)^2w^3z^2dz$ in $\mathcal{D}_{\rm spt}(\tau)$, we have
\begin{equation}\label{Eq: 4444}
    \intol\lamt^{6\alpha-5}(w^{4-3\alpha}\mathcal{L}_{w^{3\alpha}}\varphi)^2w^3z^2dz\le C\lamt^{-1}\mathcal{E}_{\rm spt}(\tau).
\end{equation}
Combining (\ref{Eq: 3333}) and (\ref{Eq: 4444}), we have (\ref{spatial energy estimate 2}).

\end{proof}

\section{Proof of Theorem \ref{Thm: stability of GW, LG}}
\begin{proof}[Proof of Theorem \ref{Thm: stability of GW, LG}]
    Let $-\tilde{\varepsilon}<\delta$ for $\tilde{\varepsilon}>0$ small enough such that the coercivity in Lemma \ref{Lemma: Coercivity of temporal energy} holds. Local-in-time existence and uniqueness of compatible radial strong solution to the problem (\ref{phiEqn})-(\ref{BC, phi}) can be established by adapting the finite-difference method in \cite{LXZ1,LXZ2} to our zero bulk viscosity problem. Under the a priori assumption $\mathfrak{E}(\tau)\le\omega_0^2$ for small enough $\omega_0$, suppose the initial data $\varphi_0$ satisfies $\mathfrak{E}(0)<\varepsilon_0$ for $\varepsilon_0>0$ sufficiently small, the temporal derivatives estimate in Lemma \ref{Lemma: Est2}, interior estimate in Lemma \ref{lemma: interior est} and spatial derivatives estimates in Lemma \ref{lemma: spatial derivatives}, \ref{lemma: spatial derivatives 2} hold, which imply the uniform in time estimate
    \begin{equation*}
        \mathfrak{E}(\tau)+\int_0^\tau\mathfrak{D}(s)ds\le C\mathfrak{E}(0)+C\int_0^\tau\lamt^{-\min\{\frac{3}{2}\alpha,\frac{1}{2}\}}\mathfrak{E}(s)ds.
    \end{equation*}
    for $\tau\in[0,+\infty)$. By a continuity argument, we can extend the solution $\varphi$ to any finite time. Thus the proof is completed.
\end{proof}

\section*{Acknowledgments}
    The author was supported in part by the NSF grant DMS-2306910.
\section*{Appendix}
\appendix
\section{Reformulation of the system}\label{Chapter: Appendix Reform}
In this section, we reformulate the system (\ref{nsp1})-(\ref{StressTen}) with $\alpha\in(0,\frac{2}{3}]$ into Lagrangian coordinates. We rewrite the system as
\begin{equation}\label{nsp3}
\begin{cases}
    &\partial_t\rho+\mathrm{div}(\rho \boldsymbol{u})=0\text{ in $\Omega(t)$},\\
&\partial_t(\rho\boldsymbol{u})+\mathrm{div}(\rho\boldsymbol{u}\otimes\boldsymbol{u})+\nabla \rho^{\frac{4}{3}}=\mathrm{div}\left(\mu\rho^{\alpha}(\nabla\boldsymbol{u}+\nabla\boldsymbol{u}^T-\frac{2}{3}(\mathrm{div}\boldsymbol{u})I_{3\times3})\right)-\rho\nabla\varPhi\text{ in $\Omega(t)$},\\
    &\Delta\varPhi=4\pi\rho\text{ in $\mathbb{R}^3$ with }\lim_{|x|\rightarrow\infty}\varPhi(x)=0,\\
    &\rho(t,\bx)>0\text{ in }\Omega(t),\\
    &\rho=0\text{ and }\left(p I_{3\times3}-\mu\rho^{\alpha}(\nabla\boldsymbol{u}+\nabla\boldsymbol{u}^T-\frac{2}{3}(\mathrm{div}\boldsymbol{u})I_{3\times3})\right)\cdot\bnu=0\text{ on $\Gamma(t)=\partial\Omega(t)$},\\
    &\mathcal{V}(\Gamma(t))=\bu\cdot\bnu,\\
    &\rho(0,\bx)=\rho_0(\bx),\ \bu(0,\bx)=\bu_0(\bx)\text{ on }\Omega(0)=\Omega_0.
\end{cases}
\end{equation}
Let $\bet(t,\bx): B_1(0)\rightarrow\Omega(t)$ be the Lagrangian flow map satisfying 
\begin{equation}
    \begin{aligned}
        \frac{d\bet}{dt}(t,\bx)=\bu(t,\bet(t,\bx)),\\
        \bet(0,\bx)=\bet_0(\bx),
    \end{aligned}
\end{equation}
where $\bet_0(\bx):\ B_1(0)\rightarrow\Omega(0)$ is a diffeomorphism with positive Jacobian determinant. We can define the Lagrangian variables
\begin{equation}
    \begin{cases}
        f(t,\bx)=\rho(t,\bet(t,\bx)),\\
        \bv(t,\bx)=\bu(t,\bet(t,\bx)),\\
        \varPsi(t,\bet(t,\bx)),
    \end{cases}
\end{equation}
and
\begin{equation}
\begin{aligned}
    \scrA=[D\bet]^{-1},\ \scrJ=\det D\bet,\ a=\scrJ\scrA.
\end{aligned}
\end{equation}
Then the continuity equation in (\ref{nsp3}) yields
\begin{equation}
    \pt f+f\scrA_i^j\partial_j\bv^i=0.
\end{equation}
For the momentum equation, we have
\begin{equation}
    f\pt v^i+\scrA^k_i\partial_kf^{\frac{4}{3}}=\scrA^k_j\partial_k\left(\mu f^{\alpha}(\scrA^l_j\partial_lv^i+\scrA^l_i\partial_lv^j-\frac{2}{3}\scrA^l_s\partial_lv^s\delta^i_j)\right)-f\scrA^k_i\partial_k\varPsi.
\end{equation}
While the Poisson equation can be written as
\begin{equation}
    \scrA^l_i\partial_l(\scrA^k_i\partial_k\varPsi)=4\pi f.
\end{equation}
Inserting $\pt \scrJ=\scrJ\scrA^k_i\partial_kv^j$ into the continuity equation, we obtain
\begin{equation}
    \pt f+f\frac{\pt \scrJ}{\scrJ}=0,
\end{equation}
which implies
\begin{equation}
    f(t,\bx)\scrJ(t,\bx)=\rho_0(\bet_0(\bx))\scrJ(0,\bx).
\end{equation}
We choose $\bet_0(\bx)$ such that $\rho_0(\bet_0(\bx))\scrJ(0,\bx)=w^3(|\bx|)$, where $w$ is the solution to (\ref{General LE}). The existence of such choice of $\bet_0(\bx)$ is ensured by \cite{DacorognaMoser}. The system (\ref{nsp3})$_1$-(\ref{nsp3})$_3$ then can be rewritten as
\begin{equation}
\begin{cases}
    w^3\pt v^i+a^k_i\partial_k(w^4\scrJ^{-\frac{4}{3}})=a^k_j\partial_k\left(\mu w^{3\alpha}\scrJ^{-\alpha}(\scrA^l_j\partial_lv^i+\scrA^l_i\partial_lv^j-\frac{2}{3}\scrA^l_s\partial_lv^s\delta^i_j)\right)-w^3\scrA^k_i\partial_k\varPsi,\\
    \scrA^l_i\partial_l(\scrA^k_i\partial_k\varPsi)=4\pi w^3\scrJ^{-1}.
\end{cases}
\end{equation}
Similarly, the boundary conditions can be rewritten as
\begin{equation}
    \left.\left(w^4\scrJ^{-\frac{4}{3}}\delta^i_j-w^{3\alpha}\scrJ^{-\alpha}(\scrA^l_j\partial_lv^i+\scrA^l_i\partial_lv^j-\frac{2}{3}\scrA^l_k\partial_lv^k\delta^i_j)\right)\scrA^i_rn^r\right|_{\partial B_1(0)}=0,
\end{equation}
where $\boldsymbol{n}:=\frac{\bx}{|\bx|}$. In a radially symmetric setting, we can let
\begin{equation}
    \bet(t,\bx)=\chi(t,z)\bx,\ z=|\bx|.
\end{equation}
By direct computation, we have
\begin{equation}
\begin{aligned}
    \scrJ=\det([\chi(\delta_i^j+\frac{\pz\chi}{z\chi}x_ix^j)])=\chi^2(\chi+z\pz\chi),\\
    \scrA^i_j=\frac{\delta^i_j}{\chi}-\frac{\pz\chi}{z\chi(\chi+z\pz\chi)}x_jx^i,
\end{aligned}
\end{equation}
and thus
\begin{equation}
\begin{aligned}
    \scrA^j_k\partial_j=&\frac{x_k}{z(\chi+z\pz\chi)}\pz,\ a^j_k\partial_j=\frac{\chi^2x_k}{z}\pz,\\
    \scrA^j_k\partial_jv^i=&\left(\frac{\delta^j_k}{\chi}-\frac{\pz\chi}{z\chi(\chi+z\pz\chi)}x_kx^j\right)\partial_j(x^i\pt\chi)\\
    =&\frac{\pt\chi}{\chi}\delta^i_k+\frac{\chi\pt\pz\chi-\pt\chi\pz\chi}{z\chi(\chi+z\pz\chi)}x_kx^i.
\end{aligned}
\end{equation}
As a result, the viscosity term becomes
\begin{equation}
\medmath{\begin{aligned}
    &a^k_j\partial_k\left(w^{3\alpha}\scrJ^{-\alpha}(\scrA^l_j\partial_lv^i+\scrA^l_i\partial_lv^j-\frac{2}{3}\scrA^l_s\partial_lv^s\delta^i_j)\right)\\
    =&\left(\chi(\chi+z\pz\chi)\delta^k_j-\frac{\chi\pz\chi}{z}x_jx^k\right)\partial_k\left(w^{3\alpha}\scrJ^{-\alpha}\frac{\chi\cdot z\pt\pz\chi-\pt\chi\cdot z\pz\chi}{\chi(\chi+z\pz\chi)}(2\frac{x_jx^i}{z^2}-\frac{2}{3}\delta^i_j)\right)\\
     =&\left(\chi(\chi+z\pz\chi)\delta^k_j-\frac{\chi\pz\chi}{z}x_jx^k\right)\frac{2(\delta_{kj}x^i+\delta^i_kx_j)z^2-4x_jx^ix_k}{z^4}w^{3\alpha}\scrJ^{-\alpha}\frac{\chi\cdot z\pt\pz\chi-\pt\chi\cdot z\pz\chi}{\chi(\chi+z\pz\chi)}\\
     &+\left(\chi(\chi+z\pz\chi)\delta^k_j-\frac{\chi\pz\chi}{z}x_jx^k\right)(2\frac{x_jx^i}{z^2}-\frac{2}{3}\delta^i_j)\frac{x_k}{z}\pz\left(w^{3\alpha}\scrJ^{-\alpha}\frac{\chi\cdot z\pt\pz\chi-\pt\chi\cdot z\pz\chi}{\chi(\chi+z\pz\chi)}\right)\\
     =&\frac{4}{z}w^{3\alpha}\scrJ^{-\alpha}(\chi\cdot z\pt\pz\chi-\pt\chi\cdot z\pz\chi)\frac{x^i}{z}+\frac{4}{3}\chi^2\pz\left(w^{3\alpha}\scrJ^{-\alpha}\frac{\chi\cdot z\pt\pz\chi-\pt\chi\cdot z\pz\chi}{\chi(\chi+z\pz\chi)}\right)\frac{x^i}{z},\\
\end{aligned}}
\end{equation}
where we use $\bv=\bx\pt\chi$. For the Poisson equation, we have
\begin{equation}
    \frac{1}{(z\chi)^2(\chi+z\pz\chi)}\pz\left(\frac{(z\chi)^2}{\chi+z\pz\chi}\pz\varPsi\right)=\frac{4\pi w^3}{\chi^2(\chi+z\pz\chi)},
\end{equation}
then
\begin{equation}
    w^3\scrA^k_i\partial_k\varPsi=w^3\frac{x_i}{z(\chi+z\pz\chi)}\pz\varPsi=\left(\frac{w^3}{z^3\chi^2}\int_0^z4\pi w^3s^2ds\right)\frac{x_i}{z}.
\end{equation}
So the momentum equation becomes
\begin{equation}
\begin{aligned}
    &\pt^2\chi+\frac{\chi^2}{zw^3}\pz\left(w^4\scrJ^{-\frac{4}{3}}\right)+\frac{1}{z^3\chi^2}\int_0^z4\pi w^3s^2ds\\
    =&\frac{1}{\chi w^3z^4}\pz\left(\frac{4}{3}\mu z^3\chi^2\left(w^3\scrJ^{-1}\right)^{\alpha}\frac{\chi\cdot z\pt\pz\chi-z\pz\chi\pt\chi}{\chi+z\pz\chi}\right),\\
    \end{aligned}
\end{equation}
and the right-hand side can also be written as
\begin{equation}
\begin{aligned}
   \frac{1}{\chi w^3z^4}\pz\left(\frac{4}{3}\mu z^3\chi^2\left(w^3\scrJ^{-1}\right)^{\alpha}\frac{\chi\cdot z\pt\pz\chi-z\pz\chi\pt\chi}{\chi+z\pz\chi}\right).
\end{aligned}
\end{equation}

Similarly, the boundary conditions can be written as
\begin{equation}
    \left.\frac{4}{3}\mu\left(w^3\scrJ^{-1}\right)^{\alpha}\left(\frac{\chi\cdot z\pt\pz\chi-\pt\chi\cdot z\pz\chi}{\chi(\chi+z\pz\chi)}\right)\right|_{z=1}=0.
\end{equation}

\bibliographystyle{abbrv}
    \bibliography{Lookup}

@article{CAO2026114572,
title = {Expanding solutions to free boundary 3D spherically symmetric compressible Navier-Stokes-Poisson equations near the Lane-Emden stars},
journal = {Journal of Differential Equations},
volume = {478},
pages = {114572},
year = {2026},
month = {10},
doi = {https://doi.org/10.1016/j.jde.2026.114572},
author = {Han Cao},
}

@book{Chan,
    author ={Chandrasekhar, Subrahmanyan},
    title ={An Introduction to the Study of Stellar Structure} ,
    publisher ={University of Chicago Press},
    year = {1939}
}

@article{CHLWW,
author = {Chen, Gui-Qiang and Huang, Feimin and Li, Tianhong and Wang, Weiqiang and Wang, Yong},
year = {2024},
month = {03},
pages = {77},
title = {Global Finite-Energy Solutions of the Compressible Euler–Poisson Equations for General Pressure Laws with Large Initial Data of Spherical Symmetry},
volume = {405},
journal = {Commun. Math. Phys.},
doi = {10.1007/s00220-023-04916-1}
}

@article{CZZ,
author = {Chen, Gui-Qiang and Zhang, Jiawen and Zhu, Shengguo},
year = {2026},
month = {01},
pages = {},
title = {Global Well-Posedness of the Vacuum Free Boundary Problem for the Degenerate Compressible Navier-Stokes Equations With Large Data of Spherical Symmetry},
journal={arXiv},
doi = {10.48550/arXiv.2601.06620}
}

@article{CCL,
author = {Cheng, Ming and Cheng, Xing and Lin, Zhiwu},
year = {2025},
month = {06},
pages = {},
title = {Expanding Solutions Near Unstable Lane-Emden Stars},
volume = {406},
journal = {Commun. Math. Phys.},
doi = {10.1007/s00220-025-05308-3}
}

@article{CLW,
author = {Cheng, Ming and Lin, Zhiwu and Wang, Yucong},
year = {2025},
month = {10},
pages = {111239},
title = {Turning point principle for the stability of viscous gaseous stars},
volume = {290},
journal = {J. Funct. Anal.},
doi = {10.1016/j.jfa.2025.111239}
}

@article{CS1,
author = {Coutand, Daniel and Shkoller, Steve},
year = {2011},
month = {03},
pages = {328 - 366},
title = {Well-Posedness in Smooth Function Spaces for Moving-Boundary 1-D Compressible Euler Equations in Physical Vacuum},
volume = {64},
journal = {Commun. Pure Appl. Math.},
doi = {10.1002/cpa.20344}
}

@article{CS2,
author = {Coutand, Daniel and Shkoller, Steve},
year = {2010},
month = {03},
pages = {515-616},
title = {Digital Object Identifier ( Well-Posedness in Smooth Function Spaces for the Moving-Boundary Three-Dimensional Compressible Euler Equations in Physical Vacuum},
volume = {206},
journal = {Arch. Ration. Mech. Anal.},
doi = {10.1007/s00205-012-0536-1}
}

@article{CLS,
author = {Coutand, Daniel and Lindblad, Hans and Shkoller, Steve},
year = {2009},
month = {06},
pages = {559-587},
title = {A Priori Estimates for the Free-Boundary 3D Compressible Euler Equations in Physical Vacuum},
volume = {296},
journal = {Commun. Math. Phys.},
doi = {10.1007/s00220-010-1028-5}
}

@article{DL,
title = {Global existence of weak solution for the compressible Navier–Stokes–Poisson system for gaseous stars},
author = {Qin Duan and Hai-Liang Li},
journal = {J. Differential Equations},
volume = {259},
number = {10},
pages = {5302-5330},
year = {2015},
issn = {0022-0396},
doi = {https://doi.org/10.1016/j.jde.2015.06.029},
}

@article{DLYY,
author = {Deng, Yinbin and Liu, Tai-Ping and Yang, Tong and Yao, Zheng-an},
year = {2002},
month = {09},
pages = {261-285},
title = {Solutions of Euler-Poisson Equations for Gaseous Stars},
volume = {164},
journal = {Arch. Ration. Mech.  Anal.},
doi = {10.1007/s00205-002-0209-6}
}

@article{DacorognaMoser,
author = {Dacorogna, Bernard and Moser, Jürgen},
year = {1990},
month = {02},
pages = {1-26},
title = {On a partial differential equation involving the Jacobian determinant},
volume = {7},
journal = {Ann. Inst. H. Poincaré Anal. Non Linéaire},
doi = {10.1016/S0294-1449(16)30307-9}
}

@article{FL,
author = {Fu, Chun-Chieh and Lin, Song-Sun},
year = {1998},
month = {10},
pages = {461-469},
title = {On the critical mass of the collapse of a gaseous star in spherically symmetric and isentropic motion},
volume = {15},
journal = {Jpn. J. Ind. and Appl. Math.},
doi = {10.1007/BF03167322}
}

@article{ZF,
author = {Fang, Daoyuan and Zhang, Ting},
year = {2006},
month = {10},
pages = {223-253},
title = {Global Behavior of Compressible Navier-Stokes Equations with a Degenerate Viscosity Coefficient},
volume = {182},
journal = {Arch. Ration. Mech. Anal},
doi = {10.1007/s00205-006-0425-6}
}

@article{GL,
author = {Gu, Xumin and Lei, Zhen},
year = {2014},
month = {05},
pages = {662-723},
title = {Local Well-posedness of the three dimensional compressible Euler--Poisson equations with physical vacuum},
volume = {105},
journal = {J. Math. Pures Appl.},
doi = {10.1016/j.matpur.2015.11.010}
}

@article{GW,
author = {Goldreich, Peter and Weber, Stephen},
year = {1980},
month = {05},
pages = {991-997},
title = {Homologously collapsing stellar cores},
volume = {238},
journal = {Astrophys. J.},
doi = {10.1086/158065}
}

@article{GLX,
author = {Guo, Zhenhua and Li, Hai-Liang and Xin, Zhouping},
year = {2012},
month = {01},
pages = {371-412},
title = {Lagrange Structure and Dynamics for Solutions to the Spherically Symmetric Compressible Navier-Stokes Equations},
volume = {309},
journal = {Commun. Math. Phys.},
doi = {10.1007/s00220-011-1334-6}
}

@article{HJ,
author = {Had\v{z}i\'{c}, Mahir and Jang, Juhi},
year = {2016},
month = {05},
pages = {},
title = {Nonlinear Stability of Expanding Star Solutions of the Radially Symmetric Mass-Critical Euler-Poisson System},
volume = {71},
journal = {Commum. Pure Appl. Math.},
doi = {10.1002/cpa.21721}
}

@article{HJ2,
author = {Hadzic, Mahir and Jang, Juhi},
year = {2018},
month = {12},
pages = {},
title = {Expanding large global solutions of the equations of compressible fluid mechanics},
volume = {214},
journal = {Invent. math.},
doi = {10.1007/s00222-018-0821-1}
}

@article{HJ3,
author = {Hadzic, Mahir and Jang, J.},
year = {2019},
month = {09},
pages = {},
title = {A Class of Global Solutions to the Euler–Poisson System},
volume = {370},
journal = {Commun. Math. Phys.},
doi = {10.1007/s00220-019-03525-1}
}

@article{HJL,
author = {Hadzic, Mahir and Jang, Juhi and Lam, King Ming},
year = {2022},
month = {12},
pages = {},
title = {Nonradial stability of self-similarly expanding Goldreich-Weber stars},
journal = {arXiv},
doi = {10.48550/arXiv.2212.11420}
}

@article{J1,
author = {Jang, Juhi},
year = {2008},
month = {05},
pages = {265-307},
title = {Nonlinear Instability in Gravitational Euler–Poisson Systems for $\gamma=\frac{6}{5}$},
volume = {188},
journal = {Arch. Ration. Mech.  Anal.},
doi = {10.1007/s00205-007-0086-0}
}

@article{J2,
author = {Jang, Juhi},
year = {2007},
month = {07},
pages = {797–863},
title = {Local Well-Posedness of Dynamics of Viscous Gaseous Stars},
volume = {195},
journal = {Arch. Ration. Mech. Anal.},
doi = {10.1007/s00205-009-0253-6}
}

@article{J3,
author = {Jang, Juhi},
year = {2014},
month = {09},
pages = {1418-1465},
title = {Nonlinear Instability Theory of Lane-Emden Stars},
volume = {67},
journal = {Commum. Pure Appl. Math.},
doi = {10.1002/cpa.21499}
}

@article{JM,
author = {Jang, Juhi and Masmoudi, Nader},
year = {2015},
month = {01},
pages = {61-111},
title = {Well-posedness of Compressible Euler Equations in a Physical Vacuum},
volume = {68},
journal = {Communications on Pure and Applied Mathematics},
doi = {10.1002/cpa.21517}
}

@article{JM0,
author = {Jang, Juhi and Masmoudi, Nader},
year = {2009},
month = {05},
pages = {1327-1385},
title = {Well‐posedness for compressible Euler equations with physical vacuum singularity},
volume = {62},
journal = {Communications on Pure and Applied Mathematics},
doi = {10.1002/cpa.20285}
}

@article{JT,
author = {Jang, Juhi and Tice, Ian},
year = {2013},
month = {01},
pages = {1121-1181},
title = {Instability theory of the Navier-Stokes-Poisson equations},
volume = {6},
journal = {Anal. PDE},
doi = {10.2140/apde.2013.6.1121}
}

@article{KL,
author = {Kong, Huihui and Li, Hai-Liang},
year = {2017},
month = {01},
pages = {1-34},
title = {Free boundary value problem to 3D spherically symmetric compressible Navier–Stokes–Poisson equations},
volume = {68},
journal = {Z. Angew. Math. Phys.},
doi = {10.1007/s00033-016-0763-7}
}

@book{KMP,
author = {Kufner, Alois and Maligranda, Lech and Persson, Larse-Erik},
year = {2007},
pages = {162},
title = {The Hardy Inequality. About Its History and Some Related Results},
publisher = {Vydavatelsky Servis Publishing House, Pilsen}
}

@article{LSS,
author = {Lin, Song-Sun},
year = {1997},
month = {05},
pages = {539-569},
title = {Stability of Gaseous Stars in Spherically Symmetric Motions},
volume = {28},
journal = {SIAM J. Math. Anal.},
doi = {10.1137/S0036141095292883}
}

@article{LWZ,
author = {Lin, Zhiwu and Wang, Yucong and Zhu, Hao},
year = {2024},
month = {07},
pages = {843-880},
title = {Nonlinear stability of non-rotating gaseous stars},
volume = {391},
journal = {Math. Ann.},
doi = {10.1007/s00208-024-02940-7}
}

@article{LXY,
author = {Liu, Tai-Ping and Yang, Tong},
year = {1997},
month = {10},
pages = {},
title = {Vacuum States for Compressible Flow},
volume = {4},
journal = {Discrete Contin. Dyn. Syst.},
doi = {10.3934/dcds.1998.4.1}
}

@article{LY,
author = {Liu, Tai-Ping and Yang, Tong},
year = {2000},
month = {01},
pages = {495-509},
title = {Compressible flow with vacuum and physical singularity},
volume = {7},
journal = {Methods Appl. Anal.},
doi = {10.4310/MAA.2000.v7.n3.a7}
}

@article{LXZ3,
author = {Luo, Tao and Xin, Zhouping and Zeng, Huihui},
year = {2014},
month = {02},
pages = {763-831},
title = {Well-Posedness for the Motion of Physical Vacuum of the Three-dimensional Compressible Euler Equations with or without Self-Gravitation},
volume = {213},
journal = {Arch. Ration. Mech. Anal.},
doi = {10.1007/s00205-014-0742-0}
}

@article{LuoWangZeng2024,
author = {Luo, Tao and Wang, Yan-Lin and Zeng, Huihui},
year = {2024},
month = {09},
pages = {1-33},
title = {Nonlinear asymptotic stability of gravitational hydrostatic equilibrium for viscous white dwarfs with symmetric perturbations},
volume = {63},
journal = {Calculus of Variations and Partial Differential Equations},
doi = {10.1007/s00526-024-02831-4}
}

@article{LZ,
author = {Lin, Zhiwu and Zeng, Chongchun},
year = {2022},
month = {11},
pages = {2511-2572},
title = {Separable Hamiltonian PDEs and Turning Point Principle for Stability of Gaseous Stars},
volume = {75},
journal = {Commum. Pure Appl. Math.},
doi = {10.1002/cpa.22027}
}

@article{LWX,
author = {Li, Hai-Liang and Wang, Yuexun and Xin, Zhouping},
year = {2024},
month = {10},
pages = {3555-3639},
title = {On the vacuum free boundary problem of the viscous Saint-Venant system for shallow water in two dimensions},
volume = {391},
journal = {Math. Ann.},
doi = {10.1007/s00208-024-03010-8}
}

@article{LXin,
title = {On the expanding configurations of viscous radiation gaseous stars: Thermodynamic model},
journal = {J. Differential Equations},
volume = {268},
number = {6},
pages = {2717-2751},
year = {2020},
issn = {0022-0396},
doi = {https://doi.org/10.1016/j.jde.2019.09.043},
author = {Xin Liu}
}

@article{LXin2,
author = {Liu, Xin},
year = {2018},
month = {12},
pages = {6100-6155},
title = {A Model of Radiational Gaseous Stars},
volume = {50},
journal = {SIAM J. Math. Anal.},
doi = {10.1137/17M1133476}
}

@article{Liu2019Isentropic,
  title={On the expanding configurations of viscous radiation gaseous stars: the isentropic model},
  author={Liu, Xin},
  journal={Nonlinearity},
  volume={32},
  number={8},
  pages={2975-3011},
  year={2019},
  doi={10.1088/1361-6544/ab10d5}
}

@article{LXZ1,
author = {Luo, Tao and Xin, Zhouping and Zeng, Huihui},
year = {2016},
month = {06},
pages = {90-182},
title = {On Nonlinear Asymptotic Stability of the Lane-Emden Solutions for the Viscous Gaseous Star Problem},
volume = {291},
journal = {Adv. Math.},
doi = {10.1016/j.aim.2015.12.022}
}

@article{LXZ2,
author = {Luo, Tao and Xin, Zhouping and Zeng, Huihui},
year = {2016},
month = {11},
pages = {657–702},
title = {Nonlinear Asymptotic Stability of the Lane-Emden Solutions for the Viscous Gaseous Star Problem with Degenerate Density Dependent Viscosities},
volume = {347},
journal = {Commun. Math. Phys.},
doi = {10.1007/s00220-016-2753-1}
}

@article{Makino,
author = {Makino, Tetu},
year = {1992},
month = {07},
pages = {615-624},
title = {Blowing Up Solutions of the Euler-Poisson Equation for the Evolution of Gaseous Stars},
volume = {21},
journal = {Transport Theory Stat. Phys.},
doi = {10.1080/00411459208203801}
}

@article{Parmeshwar,
author = {Parmeshwar, Shrish},
year = {2022},
month = {05},
pages = {},
title = {Global Existence for the N Body Euler–Poisson System},
volume = {244},
journal = {Arch. Ration. Mech. Anal.},
doi = {10.1007/s00205-022-01758-4}
}

@article{PHJ,
author = {Parmeshwar, Shrish and Hadzic, Mahir and Jang, Juhi},
year = {2020},
month = {10},
pages = {1},
title = {Global expanding solutions of compressible Euler equations with small initial densities},
volume = {79},
journal = {Quart. Appl. Math.},
doi = {10.1090/qam/1580}
}

@article{Rein,
author = {Rein, Gerhard},
year = {2003},
month = {06},
pages = {115-130},
title = {Non-Linear Stability of Gaseous Stars},
volume = {168},
journal = {Arch. Ration. Mech.  Anal.},
doi = {10.1007/s00205-003-0260-y}
}

@article{Rickard,
author = {Rickard, Calum},
year = {2021},
month = {08},
pages = {},
title = {Global Solutions to the Compressible Euler Equations with Heat Transport by Convection Around Dyson’s Isothermal Affine Solutions},
volume = {241},
journal = {Arch. Ration. Mech. Anal.},
doi = {10.1007/s00205-021-01669-w}
}

@article{YZ,
author = {Yang, Tong and Zhu, Changjiang},
year = {2002},
month = {10},
pages = {329-363},
title = {Compressible Navier–Stokes Equations with Degenerate Viscosity Coefficient and Vacuum},
volume = {230},
journal = {Commun. Math. Phys.},
doi = {10.1007/s00220-002-0703-6}
}

@article{WangZhangZhu2026Jun,
author={Wang, Demin and Zhang, Jiawen and Zhu, Shengguo},
year = {2026},
month = {06},
pages = {},
journal={arXiv},
title = {On a Local Existence Theorem for the Evolution Equation of Viscous Gaseous Stars in a Physical Vacuum},
doi = {10.48550/arXiv.2606.22822}
}

@article{WangZhangZhu2026Jul,
author = {Wang, Demin and Zhang, Jiawen and Zhu, Shengguo},
year = {2026},
month = {07},
pages = {},
journal={arXiv},
title = {Global dynamics of viscous gaseous stars in a physical vacuum},
doi = {10.48550/arXiv.2607.09189}
}

@article{Sideris,
author = {Sideris, Thomas},
year = {2017},
month = {07},
pages = {},
title = {Existence and asymptotic behavior of affine motion of 3D ideal fluids surrounded by vacuum},
volume = {225},
journal = {Arch. Ration. Mech. Anal.},
doi = {10.1007/s00205-017-1106-3}
}

@article{shkollerSideris,
author = {Shkoller, Steve and Sideris, Thomas},
year = {2019},
month = {10},
pages = {},
title = {Global Existence of Near-Affine Solutions to the Compressible Euler Equations},
volume = {234},
journal = {Arch. Ration. Mech. Anal.},
doi = {10.1007/s00205-019-01387-4}
}
\end{document}